\documentclass[aap, preprint,noautosecdot,a4wide]{imsart}

\RequirePackage[OT1]{fontenc}
\RequirePackage{amsthm,amsmath,amsfonts,amssymb,bbm,mathrsfs,mathtools}
\RequirePackage{upref}
\RequirePackage[shortlabels]{enumitem}
\RequirePackage{subcaption}
\RequirePackage{graphicx, caption, tikz}
\RequirePackage[colorlinks,citecolor=blue,urlcolor=blue]{hyperref}
\RequirePackage{multirow, tabularx, booktabs}
\RequirePackage{todonotes}
\RequirePackage{xcolor}

\newcommand{\prob}{\mathbb{P}}
\newcommand{\Prob}[1]{\prob\left(#1\right)}
\newcommand{\Probn}[1]{\prob_n\left(#1\right)}

\newcommand{\expec}{\mathbb{E}}
\newcommand{\Exp}[1]{\expec\left[#1\right]}
\newcommand{\Expn}[1]{\expec_n\left[#1\right]}

\newcommand{\Var}[1]{\textup{Var}\left(#1\right)}
\newcommand{\Varn}[1]{\textup{Var}_n\left(#1\right)}

\newcommand{\plim}{\ensuremath{\stackrel{\sss{\prob}}{\longrightarrow}}}

\newcommand{\ind}[1]{\mathbbm{1}_{\left\{#1\right\}}}

\newcommand{\bigO}[1]{O\left(#1\right)}
\newcommand{\bigOp}[1]{O_{\sss\prob}\left(#1\right)}
\newcommand{\bigOps}{O_{\sss\prob}}

\newcommand{\sss}[1]{\scriptscriptstyle{#1}}
\newcommand\abs[1]{\left|#1\right|}
\newcommand{\me}{\textup{e}}
\newcommand{\dd}{{\rm d}}

\newcommand{\cj}[1]{\textcolor{black}{#1}}
\newcommand{\op}{o_{\sss\prob}}
\newcommand{\gind}{{\sss{(\mathrm{ind})}}}
\newcommand{\gsub}{{\sss{(\mathrm{sub})}}}

\newcommand{\ECMn}{{\rm ECM}^{\sss{(n)}}}

\newcommand{\CMnD}{{\rm CM}_n(\boldsymbol{D})}
\newcommand{\ECMnD}{{\rm ECM}^{\sss{(n)}}(\boldsymbol{D})}

\newcommand{\Nn}{N^{\sss(n)}}
\newcommand{\Mn}{M^{\sss(n)}}
\newcommand{\Ecal}{{\mathcal{E}}}

\startlocaldefs
\numberwithin{equation}{section}
\endlocaldefs

\newcommand*{\swap}[2]{#2#1}

\newcommand{\cs}[1]{{\color{black}#1}}

\newcommand{\eqn}[1]{\begin{equation}#1\end{equation}}
\newcommand{\eqan}[1]{\begin{align}#1\end{align}}

\allowdisplaybreaks

\newtheorem{theorem}{Theorem}[section]

\newtheorem{lemma}[theorem]{Lemma}

\graphicspath{{Figures/}}
\numberwithin{equation}{section}
\setenumerate{label={\upshape{(\roman*)}}} 

\begin{document}
	
	\begin{frontmatter}
		\title{Optimal subgraph structures in scale-free configuration models}
		\runtitle{Optimal subgraph structures in configuration models}
		
		\begin{aug}
			\author{\fnms{Remco} \snm{van der Hofstad}\thanksref{t1,t2}\ead[label=e1]{}},
			\author{\fnms{Johan S. H.} \snm{van Leeuwaarden}\thanksref{t1,t3}\ead[label=e1]{}}
			\and
			\author{\fnms{Clara} \snm{Stegehuis}\corref{}\thanksref{t1,t4}
				\ead[label=e3]{c.stegehuis@utwente.nl}}		
			\thankstext{t1}{Supported by NWO TOP grant 613.001.451}
			\thankstext{t2}{Also supported by NWO Gravitation Networks grant 024.002.003 and NWO VICI grant 639.033.806.}
			\thankstext{t3}{Also supported by NWO Gravitation Networks grant 024.002.003, an NWO TOP-GO grant and an ERC Starting Grant}
			\runauthor{R. van der Hofstad, J. van Leeuwaarden and C. Stegehuis}
			
			\affiliation{Eindhoven University of Technology\thanksmark{t2}, Tilburg University\thanksmark{t3} and Twente University\thanksmark{t4}}
			
		\end{aug}
		
		\begin{abstract}
				Subgraphs reveal information about the geometry and functionalities of complex networks. For scale-free networks with unbounded degree fluctuations, 
			we obtain the asymptotics of the number of times a small connected graph  occurs as a subgraph or as an induced subgraph. We obtain these results by analyzing the configuration model with degree exponent $\tau\in(2,3)$ and introducing a novel class of optimization problems. For any given subgraph, the unique optimizer describes the degrees of the vertices that together span the subgraph. 
			We find that subgraphs typically occur between vertices with specific degree ranges. In this way, we can count and characterize {\it all} subgraphs. We refrain from double counting in the case of multi-edges, essentially counting the subgraphs in the {\it erased} configuration model. 		\end{abstract}
		
		
		
	\end{frontmatter}

\section{Introduction}

Scale-free networks often have degree distributions  that follow power laws
with exponent $\tau\in(2,3)$ \cite{albert1999,faloutsos1999,jeong2000,vazquez2002}. Many networks have been reported to satisfy these conditions,  including 
metabolic networks, the internet and social networks. Scale-free networks come with the presence of hubs, i.e., vertices of extremely high degrees. 

Another property of real-world scale-free networks is that the clustering coefficient (the probability that two uniformly chosen neighbors of a vertex are neighbors themselves) decreases with the vertex degree~\cite{boguna2003,colomer2012,maslov2004,stegehuis2017,vazquez2002}, again following a power law. Thus, two neighbors of a hub are less likely to connect.
The triangle is the most studied network subgraph, because it not only describes the clustering coefficient, but also signals hierarchy and community structure~\cite{ravasz2003}.
However, other subgraphs such as larger cliques are equally important for understanding network organization \cite{benson2016,tsourakakis2017}. Indeed, subgraph counts might vary considerably across different networks \cite{milo2004,milo2002,wuchty2003} and any given network may have a set of 
statistically significant subgraphs (also called motifs). 
Statistical relevance can be expressed by comparing a real-world network to some mathematically tractable model. This comparison filters out the effect of the degree sequence and the network size on the subgraph count. A popular statistic takes the subgraph count, 
subtracts the expected number of subgraphs in a model, and divides by the standard deviation in the model~\cite{gao2017,milo2004,onnela2005}. Such a standardized test statistic sheds light on whether a subgraph is overrepresented in comparison to the model. 
This raises the question of what model to use. A natural candidate is the uniform simple graph with the same degrees as the original network. 

For $\tau>3$, when the degree distribution has a finite second moment, it is easy to generate such graphs using the configuration model, a random graph model that creates random graphs with any given degree sequence~\cite{bollobas1980,hofstad2009}. For  $\tau\in(2,3)$, however, the configuration model fails to create simple graphs with high probability~\cite{janson2009c}.
We therefore consider the erased configuration model~\cite{britton2006}\cite[Chapter 7]{hofstad2009},  which constructs a configuration model and then removes all self-loops and merges multiple edges. For an erased configuration model with degree exponent $\tau\in(2,3)$, we count how often a small connected graph $H$ occurs as a subgraph or as an induced subgraph, where edges not present in $H$ are also not allowed to be present in the subgraph. 

We find that every (induced) subgraph $H$, typically occurs between vertices in the erased configuration model with degrees in highly specific ranges that depend on the precise subgraph $H$. An example of these typical degree ranges for subgraphs on 4 vertices is shown in Figure~\ref{fig:graphlet4} (which will be discussed in more detail in Section~\ref{sec:motif45}).
In this paper we show that many subgraphs consist exclusively of $\sqrt{n}$-degree vertices, including cliques of all sizes. Hence, in such subgraphs, hubs (of degree close to the maximal value $n^{1/(\tau-1)}$) are unlikely to participate in a typical subgraph. Hubs can be part, however, of other subgraphs. We define optimization problems that find these optimal degree ranges for every subgraph. 

We next define the model.
\begin{figure}[tb]
	\definecolor{mycolor1}{RGB}{230,37,52}%
	\tikzstyle{every node}=[circle,fill=black!25,minimum size=8pt,inner sep=0pt]
	\tikzstyle{S1}=[fill=mycolor1!60]
	\tikzstyle{S2}=[fill=mycolor1!60!black]
	\tikzstyle{S3}=[fill=mycolor1]
	\tikzstyle{n1}=[fill=mycolor1!20]
	\begin{subfigure}[t]{0.16\linewidth}
		\centering
		\begin{tikzpicture}
		\tikzstyle{edge} = [draw,thick,-]
		\node[S3] (a) at (0,0) {};
		\node[S3] (b) at (1,0) {};
		\node[S3] (c) at (0,1) {};
		\node[S3] (d) at (1,1) {};
		\draw[edge] (a)--(b);
		\draw[edge] (c)--(b);
		\draw[edge] (d)--(b);
		\draw[edge] (a)--(c);
		\draw[edge] (a)--(d);
		\draw[edge] (c)--(d);
		\end{tikzpicture}	
		\caption{$n^{6-2\tau}$}
		\label{fig:K4}
	\end{subfigure}
	\begin{subfigure}[t]{0.16\linewidth}
		\centering
		\begin{tikzpicture}
		\tikzstyle{edge} = [draw,thick,-]
		\node (a) at (0,0) {};
		\node (b) at (1,0) {};
		\node (c) at (0,1) {};
		\node (d) at (1,1) {};
		\draw[edge] (a)--(b);
		\draw[edge] (c)--(b);
		\draw[edge] (d)--(b);
		\draw[edge] (a)--(c);
		\draw[edge] (c)--(d);
		\end{tikzpicture}	
		\caption{$n^{6-2\tau}\log(n)$}
		\label{fig:squareextra}
	\end{subfigure}
	\begin{subfigure}[t]{0.16\linewidth}
		\centering
		\begin{tikzpicture}
		\tikzstyle{edge} = [draw,thick,-]
		\node[S3] (a) at (0,0) {};
		\node[S3] (b) at (1,0) {};
		\node[S3] (c) at (0,1) {};
		\node[S3] (d) at (1,1) {};
		\draw[edge] (a)--(b);
		\draw[edge] (d)--(b);
		\draw[edge] (a)--(c);
		\draw[edge] (c)--(d);
		\end{tikzpicture}	
		\caption{$n^{6-2\tau}$}
		\label{fig:square}
	\end{subfigure}
	\begin{subfigure}[t]{0.16\linewidth}
		\centering
		\begin{tikzpicture}
		\tikzstyle{every node}=[circle,fill=black!25,minimum size=8pt,inner sep=0pt]
		\tikzstyle{edge} = [draw,thick,-]
		\node[n1] (a) at (0,0) {};
		\node[S1] (b) at (1,0) {};
		\node[S2] (c) at (0,1) {};
		\node[S1] (d) at (1,1) {};
		\draw[edge] (c)--(b);
		\draw[edge] (d)--(b);
		\draw[edge] (a)--(c);
		\draw[edge] (c)--(d);
		\end{tikzpicture}	
		\caption{$n^{7-2\tau-\frac{1}{\tau-1}}$}
		\label{fig:paw}
	\end{subfigure}
	\begin{subfigure}[t]{0.16\linewidth}
		\centering
		\begin{tikzpicture}
		\tikzstyle{every node}=[circle,fill=black!25,minimum size=8pt,inner sep=0pt]
		\tikzstyle{edge} = [draw,thick,-]
		\node[n1] (a) at (0,0) {};
		\node[n1] (b) at (1,0) {};
		\node[S2] (c) at (0,1) {};
		\node[n1] (d) at (1,1) {};
		\draw[edge] (c)--(b);
		\draw[edge] (a)--(c);
		\draw[edge] (c)--(d);
		\end{tikzpicture}	
		\caption{$n^{\frac{3}{\tau-1}}$}
		\label{fig:wedge4}
	\end{subfigure}
	\begin{subfigure}[t]{0.16\linewidth}
		\centering
		\begin{tikzpicture}
		\tikzstyle{every node}=[circle,fill=black!25,minimum size=8pt,inner sep=0pt]
		\tikzstyle{edge} = [draw,thick,-]
		\node[n1] (a) at (0,0) {};
		\node[n1] (b) at (1,0) {};
		\node (c) at (0,1) {};
		\node (d) at (1,1) {};
		\draw[edge] (d)--(b);
		\draw[edge] (a)--(c);
		\draw[edge] (c)--(d);
		\end{tikzpicture}	
		\caption{$n^{4-\tau}\log(n)$}
	\end{subfigure}

\vspace{-0.5cm}
\begin{subfigure}{\linewidth}
	\centering
	\begin{tikzpicture}
	\node[S2,label={[label distance=0.05cm]0:$n^{1/(\tau-1)}$}] (a) at (5,0) {};
	\node[S3,label={[label distance=0cm]0:$\sqrt{n}$}] (b) at (3.6,0) {};
	\node[S1,label={[label distance=0.05cm]0:$n^{(\tau-2)/(\tau-1)}$}] (c) at (1,0) {};
	\node[n1,label={[label distance=0.05cm]0:$1$}] (c) at (0,0) {};
	\node[label={[label distance=0.05cm]0:non-unique}] (d) at (7,0) {};
	\end{tikzpicture}
\end{subfigure}
\vspace{-0.8cm}
\caption{Optimal structures and asymptotic counts of induced subgraphs on 4 vertices. The vertex colors indicate the typical degrees and the scaling of the number of subgraphs is given below the pictures. Our results do not apply to the gray vertices.}
	\label{fig:graphlet4}
\end{figure}

	\paragraph{The erased configuration model.}
Let $[n]=\{1,2,\ldots,n\}$. Given a \emph{degree sequence}, i.e., a sequence of $n$ positive integers $\boldsymbol{D}=(D_1,D_2,\ldots,D_n)$, the \emph{configuration model} is a (multi)graph with vertex set $[n]$, where vertex $v\in[n]$ has degree $D_v$. It is defined as follows, see e.g., \cite{bollobas2001} or \cite[Chapter 7]{hofstad2009}: given a degree sequence  with $\sum_{v\in[n]} D_v$ even, we start with $D_v$ free half-edges adjacent to vertex $v$, for $v\in[n]$. The configuration model is constructed by successively pairing, uniformly at random, free half-edges into edges and removing them from the set of free half-edges, until no free half-edges remain. Conditionally on obtaining a simple graph, the resulting graph is a {\em uniform} sample from the ensemble of simple graphs with the prescribed degree sequence~\cite[Chapter 7]{hofstad2009}. This is why the configuration model is often used as a {\em model} for real-world networks with given degrees. 
The {\em erased configuration model} is the model where all multiple edges are merged and all self-loops are removed.

In this paper, we study the setting where the degree distribution has {\em infinite variance}. Then the number of 
erased edges is large~\cite{Hoorn2015} (yet small compared to the total number of edges). 
%
In particular, we take the degrees to be an i.i.d.\ copies of a random variable $D$ such that
	\begin{equation}
	\label{D-tail}
	\prob(D=k)=c k^{-\tau}(1+o(1)), \quad {\rm as} \ k\rightarrow \infty,
	\end{equation}
where $\tau\in(2,3)$ so that $\Exp{D^2}=\infty$ and 
	\begin{equation}
	\Exp{D}=\mu<\infty.
	\end{equation} 
When this sample constructs a degree sequence such that the sum of the degrees is odd, we add an extra half-edge to the last vertex. This does not affect our computations. In this setting, $D_{\max}$ is of order $n^{1/(\tau-1)}$, where $D_{\max}=\max_{v\in[n]}D_v$ denotes the maximal degree of the degree sequence. Denote the erased configuration model on $n$ vertices by $\ECMn$ when the degrees are an independent and identically distributed (i.i.d.) sample of~\eqref{D-tail}, and  $\ECMnD$ when the degree sequence equals $\boldsymbol{D}$.

\paragraph{Quenched and annealed.}
Note that the erased configuration model as defined above has two sources of randomness: the independent and identically distributed (i.i.d.) degrees and the random pairing of the half-edges in constructing the graph. Studying the behavior of subgraphs in the erased configuration model once the degree sequence has been fixed corresponds to the \emph{quenched} setting, whereas the erased configuration model with random degrees corresponds to the \emph{annealed} setting. Our main result on the number of subgraphs in the erased configuration model is in the annealed setting. However, in the proof of our results we often study subgraph counts in the quenched setting. Throughout this paper, we denote the probability of an event $\mathcal{E}$ in the quenched setting by\cj{
\begin{equation}
\Probn{\mathcal{E}}=\Prob{\mathcal{E}\mid (D_v)_{v\in[n]}},
\end{equation}}
and we define $\expec_n$ and $\text{Var}_n$ accordingly. 

\paragraph{Subgraph counts.}
Let $H=(V_H,\Ecal_H)$ be a small, connected graph. We denote the induced subgraph count of $H$, the number of subgraphs of $\ECMn$ that are isomorphic to $H$, by $N^\gind(H)$. 
We denote the subgraph count, the number of occurrences of $H$ as a subgraph of $\ECMn$, by $N^\gsub(H)$.

Throughout this paper, we denote the sampled degree of a vertex $\cj{v\in[n]}$ in the erased configuration model by $\cj{D_v}$. Note that this may not be the same as the actual degree of a vertex in the erased configuration model, since self-loops are removed and multiple edges are merged. Since we study subgraphs $H$, we sometimes also  need to use the degree of a vertex in $H$ inside the subgraph. We denote the degree of a vertex $i$ of a subgraph $H$ by $d_i$.


\paragraph{Paper outline.}
We present our main results in Section~\ref{sec:main}, including the theorems that characterize all optimal subgraphs in terms of the solutions to optimization problems. We also apply these theorems to describe the optimal configurations of all subgraphs with 4 and 5 vertices, and present an outlook for further use of our results. We provide an overview of the proof structures in Section~\ref{sec:overview}.
We then prove the first part of the main theorems for subgraphs in Section~\ref{sec:maxcont} and for $\sqrt{n}$-optimal subgraphs in Section~\ref{sec:proof2}. The proofs of some lemmas introduced along the way are deferred to Section~\ref{sec:prooflem1}. The proof of the second part of the main theorem can be found in Section~\ref{sec:proofsec2}. We finally show how the proofs for subgraphs can be adjusted to prove the theorems on induced subgraphs in Section~\ref{sec:graphlets}.
		
\paragraph{Notation.}\label{sec:notation}
We say that a sequence of events $(\mathcal{E}_n)_{n\geq 1}$ happens with high probability (w.h.p.) if $\lim_{n\to\infty}\Prob{\mathcal{E}_n}=1$ and we use $\plim $ for convergence in probability. We write $f(n)=o(g(n))$ if $\lim_{n\to\infty}f(n)/g(n)=0$, and $f(n)=O(g(n))$ if $|f(n)|/g(n)$ is uniformly bounded. We write $f(n)=\Theta(g(n))$ if $f(n)=O(g(n) )$ as well as $g(n)=O(f(n))$. We say that $X_n=O_{\sss{\prob}}(g(n))$ for a sequence of random variables $(X_n)_{n\geq 1}$ if $|X_n|/g(n)$ is a tight sequence of random variables, and $X_n=o_{\sss{\prob}}(g(n))$ if $X_n/g(n)\plim 0$.

\section{Main results}\label{sec:main}
The key insight obtained in this paper is that the creation of subgraphs is crucially affected by the following trade-off, inherently present in power-law networks: {On} the one hand, hubs contribute substantially to the subgraph count, because they are well connected, and therefore potentially contribute to many subgraphs. On the other hand, hubs are by definition rare. This should be contrasted with lower-degree vertices that occur more frequently, but typically take part in fewer connections and hence fewer subgraphs. 
Therefore, one may expect every subgraph to consist of a selection of vertices with specific degrees
that `optimizes' this trade-off and hence maximizes the probability that the subgraph occurs. 

Let $\ECMn|_{\boldsymbol{v}}$ denote the induced subgraph of the erased configuration model on vertices $\boldsymbol{v}$. Write the probability that a subgraph $H=(V_H,\Ecal_H)$ with $|V_H|=k$ is created on $k$ uniformly chosen vertices $\boldsymbol{v}=(v_1, \ldots, v_k)$ in $\ECMn$ as
\begin{equation}\label{eq:pHpresent}
\Prob{\ECMn|_{\boldsymbol{v}}\supseteq H}=\sum_{\boldsymbol{D}'}\Prob{\ECMn|_{\boldsymbol{v}}\supseteq H \mid D_{\boldsymbol{v}}=\boldsymbol{D}'}\Prob{D_{\boldsymbol{v}}=\boldsymbol{D}'},
\end{equation}
where the sum is over all possible degrees on $k$ vertices $\boldsymbol{D}'=(D_i')_{i\in [k]}$, and $D_{\boldsymbol{v}}=(D_{v_i})_{i\in[k]}$ denotes the degrees of the randomly chosen set of $k$ vertices. 
We show that for every (induced) subgraph, there is a specific range of $D_1',\ldots,D_k'$ that gives the maximal contribution to~\eqref{eq:pHpresent}, large enough even to completely ignore all other degree ranges. 

We show that \cs{when~\eqref{eq:pHpresent} is maximized by a unique range of degrees}, there are only four possible ranges of degrees that maximize the term inside the sum in~\eqref{eq:pHpresent}. These ranges are constant degrees, or degrees proportional to $n^{(\tau-2)/(\tau-1)}$, to $\sqrt{n}$ or to $n^{1/(\tau-1)}$. 

\subsection{An optimization problem}
We now present the optimization problems that maximizes the summand in~\eqref{eq:pHpresent}, first for subgraphs and later for induced subgraphs. 
Let $H=(V_H,\Ecal_H)$ be a small, connected graph on $k\geq 3$ vertices. Denote the set of vertices of $H$ that have degree one inside $H$ by $V_1$. 
Let $\mathcal{P}$ be all partitions of $V_H\setminus V_1$ into three disjoint sets $S_1,S_2,S_3$.
 This partition into $S_1,S_2$ and $S_3$ corresponds to the following typical orders of magnitude \cs{of the degrees of the vertices of $H$ embedded in $\ECMn$}: $S_1$ denotes the vertices with degree proportional to $n^{(\tau-2)/(\tau-1)}$, $S_2$ the ones with degrees proportional to $n^{1/(\tau-1)}$, and $S_3$ the vertices with degrees proportional to $\sqrt{n}$. The optimization problem finds the partition of the vertices into these three orders of magnitude that maximizes the contribution to the number of (induced) subgraphs. When a vertex in $H$ has degree 1, its degree in $\ECMn$ is typically small, i.e., it does not grow {with} $n$. 

Given a partition $\mathcal{P}=(S_1,S_2,S_3)$ of $V_H\setminus V_1$, let $\Ecal_{S_i}$ denote the set of edges in $H$ between vertices in $S_i$ and $E_{S_i}=|\Ecal_{S_i}|$ its size, $\Ecal_{S_i,S_j}$ the set of edges between vertices in $S_i$ and $S_j$ and $E_{S_i,S_j}=|\Ecal_{S_i,S_j}|$ its size, and finally $\Ecal_{S_i,V_1}$ the set of edges between vertices in $V_1$ and $S_i$ and $E_{S_i,V_1}=|\Ecal_{S_i,V_1}|$ its size. We now define the optimization problem for subgraphs that optimizes the summand in~\eqref{eq:pHpresent} as
	\begin{equation}
	\label{eq:maxeqsub}
	B^\gsub(H) = \max_{\mathcal{P}}\left[\abs{S_1}-\abs{S_2}-\frac{2E_{S_1}+E_{S_1,S_3}+E_{S_1,V_1}-E_{S_2,V_1}}{\tau-1}\right].
	\end{equation}
The first two terms in the optimization problem give a positive contribution for all vertices in $S_1$, \cj{which have} relatively low degree, and a negative contribution for vertices in $S_2$, which have high degrees. Therefore, the first two terms in the optimization problem capture that high-degree vertices are rare, and low-degree vertices abundant. The last term gives a negative contribution for all edges between vertices with relatively low degrees in the subgraph. This captures the other part of the trade-off: high-degree vertices are more likely to connect to other vertices than low degree vertices. 
Note that $B^\gsub(H)\geq 0$, since putting all vertices in $S_3$ yields zero. 

For induced subgraphs, we define the similar optimization problem
	\begin{align}
	B^\gind(H)  = & \max_{\mathcal{P}^\gind}\left[\abs{S_1}-\abs{S_2}-\frac{2E_{S_1}+E_{S_1,S_3}+E_{S_1,V_1}-E_{S_2,V_1}}{\tau-1}\right],\nonumber\\
	&\qquad\quad \text{s.t.}\quad  \{{i,j}\}\in \Ecal_H \quad \forall  {i}\in S_2, {j}\in S_2\cup S_3,
	\label{eq:maxeqind}
	\end{align}
where again $\mathcal{P}^\gind{=(S_1,S_2,S_3)}$ is a partition of $V_H\setminus V_1$ into three sets. The constraint in~\eqref{eq:maxeqind} ensures that edges that are not present in $H$ are not present in the subgraph. 
Again, $B^\gind(H)\geq 0$ because $S_3=V_H\setminus V_1$ is a valid solution. 

Our main result shows that indeed the optimization problems~\eqref{eq:maxeqsub} and~\eqref{eq:maxeqind} find the typical vertex degrees for any (induced) subgraph and determine the scaling of the number of subgraphs.
We then investigate a special class of subgraphs, where the optimal contribution to~\eqref{eq:maxeqsub} or~\eqref{eq:maxeqind} is $S_3=V_H$, i.e., (induced) subgraphs where all typical vertex degrees are proportional to $\sqrt{n}$. For this class, which contains for instance cliques of all sizes, we present sharp asymptotics.

\subsection{General subgraphs}
 Let $S_1^\gsub,S_2^\gsub,S_3^\gsub$ be a maximizer of~\eqref{eq:maxeqsub}. Furthermore, for any $\boldsymbol{\alpha}=(\alpha_1,\ldots, \alpha_k)$ such that $\alpha_i\in[0,1/(\tau-1)]$, define
 	\begin{equation}\label{eq:Mnalph}
 	M_n^{(\boldsymbol{\alpha})}(\varepsilon)=\{ ({v_1,\ldots, v_k})\colon  D_{{v_i}}\in[\varepsilon,1/\varepsilon] (\mu n) ^{\alpha_i}\ \forall i\in[k] \}.
 	\end{equation}
These are the sets of vertices $(v_1,\ldots, v_k)$ such that ${D_{v_1}}$ is proportional to $n^{\alpha_1}$ and ${D_{v_2}}$ proportional to $n^{\alpha_2}$ and so on. 
Denote the number of subgraphs with vertices in $M_n^{(\boldsymbol{\alpha})}(\varepsilon)$ by $N^\gsub(H,M_n^{(\boldsymbol{\alpha})}(\varepsilon))$. Define the vector $\boldsymbol{\alpha}^\gsub$ as
 \begin{equation}\label{eq:alphasub}
	 {\alpha}_i^\gsub=\begin{cases}
		 (\tau-2)/(\tau-1)& i\in S^\gsub_1,\\
		 1/(\tau-1)& i\in S^\gsub_2,\\
		 \tfrac{1}{2} & i\in S^\gsub_3,\\
		 0 & i\in V_1.
	 \end{cases}
 \end{equation}
 For induced subgraphs, let $S_1^\gind,S_2^\gind,S_3^\gind$ be a maximizer of~\eqref{eq:maxeqind}, and define $\boldsymbol{\alpha}^\gind$ as in~\eqref{eq:alphasub}, replacing $S_i^\gsub$ by $S_i^\gind$.  
The next theorem shows that sets of vertices in $M_n^{\boldsymbol{\alpha}^\gsub}(\varepsilon)$ or $M_n^{\boldsymbol{\alpha}^\gind}(\varepsilon)$ contain a large number of subgraphs, and computes the scaling of the number of (induced) subgraphs:

\begin{theorem}[General (induced) subgraphs]\label{thm:motifs}
	Let $H$ be a subgraph on $k$ vertices such that the solution to~\eqref{eq:maxeqsub} is unique.
	\begin{enumerate}[(i)]
		\item 
For any $\varepsilon_n$ such that $\lim_{n\to\infty}\varepsilon_n=0$,
	\begin{equation}
	\frac{N^\gsub\big(H,M_n^{(\boldsymbol{\alpha}^\gsub)}\left(\varepsilon_n\right)\big) }{N^\gsub(H)}\plim 1.
	\end{equation}
	\item Furthermore, for any fixed $0<\varepsilon<1$,\cs{
	\begin{equation}\label{eq:Nsubmag}
	\frac{N^\gsub(H,M_n^{(\boldsymbol{\alpha}^\gsub)}(\varepsilon))}{n^{\frac{3-\tau}{2}(k_{2+}+B^\gsub(H))+k_1/2}} \leq  f(\varepsilon)+\op(1),
	\end{equation}
and
	\begin{equation}\label{eq:Nsubmaglow}
	\frac{N^\gsub(H,M_n^{(\boldsymbol{\alpha}^\gsub)}(\varepsilon))}{n^{\frac{3-\tau}{2}(k_{2+}+B^\gsub(H))+k_1/2}} \geq  \tilde{f}(\varepsilon)+\op(1),
	\end{equation}
for some functions $f(\varepsilon),\tilde{f}(\varepsilon)<\infty $ not depending on $n$.} Here $k_{2+}$ denotes the number of vertices in $H$ of degree at least 2, and $k_1$ the number of degree-one vertices in {the subgraph} $H$.
\end{enumerate}For induced subgraphs the same statements hold, replacing $\rm{\scriptstyle (sub)}$ by $\rm{\scriptstyle (ind)}$ \and the optimization problem in~\eqref{eq:maxeqsub} by that in~\eqref{eq:maxeqind}.
\end{theorem}
Theorem~\ref{thm:motifs}(ii) only provides the scaling in $n$ and some functions $f(\varepsilon), \tilde{f}(\varepsilon)$, which could tend to $\infty$ when $\varepsilon\searrow 0$. For subgraphs with $S_3=V_H$, we obtain more precise asymptotics in the next section.

%

\subsection{Sharp asymptotics for $\sqrt{n}$-class of subgraphs}
Now we study the special class of subgraphs for which the unique maximum of \eqref{eq:maxeqsub} or~\eqref{eq:maxeqind} is $S_3=V_H$. 
By the above interpretation of $S_1$, $S_2$ and $S_3$, we study (induced) subgraphs where the maximum contribution to the number of such subgraphs comes from vertices that have degrees proportional to $\sqrt{n}$ in $\ECMn$. 
Examples of subgraphs that fall into this category are all complete graphs. Bipartite graphs on the other hand, do not fall into the $\sqrt{n}$-class subgraphs, since we can use the two parts of the bipartition as $S_1$ and $S_2$ in such a way that~\eqref{eq:maxeqsub} results in a non-negative solution. The next theorem gives asymptotics for the number of $\sqrt{n}$-(induced) subgraphs:

\begin{theorem}[(Induced) subgraphs with $\sqrt{n}$ degrees]\label{thm:sqrtsub}
	Let $H$ be a connected graph on $k$ vertices with minimal degree 2 such that the solution to~\eqref{eq:maxeqsub} is unique, and $B^\gsub(H)=0$. Then,
	\begin{equation}
	\frac{N^\gsub(H)}{n^{\frac{k}{2}(3-\tau)}}\plim A^\gsub(H)<\infty,
	\end{equation}
	with
	\begin{equation}\label{eq:Asub}
	A^\gsub(H) = c^k\mu^{-\frac{k}{2}(\tau-1)}\int_{0}^{\infty}\cdots \int_{0}^{\infty}(x_1\cdots x_k)^{-\tau}\prod_{\mathclap{\{{i,j}\}\in \Ecal_{H}}}(1-\me^{-x_ix_j})\dd x_1\cdots \dd x_k.
	\end{equation}
	For induced subgraphs the same statements hold, replacing $\rm{\scriptstyle (sub)}$ by $\rm{\scriptstyle (ind)}$ and~\eqref{eq:maxeqsub} by~\eqref{eq:maxeqind}, where
		\begin{equation}\label{eq:Aind}
	\begin{aligned}[b]
	A^\gind(H) &= c^k\mu^{-\frac{k}{2}(\tau-1)}\int_{0}^{\infty}\cdots \int_{0}^{\infty}(x_1\cdots x_k)^{-\tau} \prod_{\mathclap{\{{i,j}\}\in \Ecal_H}} \ (1-\me^{-x_ix_j})\\
	& \quad\quad  \times \prod_{\mathclap{\{{i,j}\}\notin \Ecal_H}}   \me^{-x_ix_j}\dd x_1\cdots \dd x_k.
	\end{aligned}
	\end{equation}
\end{theorem}

In the erased configuration model, the probability that a vertex with degree $D_{v_i}$ connects to a vertex with degree $D_{v_j}$ can be approximated by $1-\me^{-D_{v_i}D_{v_j}/L_n}$, where $L_n\cj{=\sum_{v\in[n]}D_v}$ denotes the sum of all degrees. When rescaling, and taking $D_{v_i}\approx x_i\sqrt{n/\mu}$ and $D_{v_j}=x_j\sqrt{n/\mu}$, this results in the factors $1-\me^{-x_ix_j}$ in~\eqref{eq:Asub} for all edges $\{i,j\}\in \Ecal_H$ in subgraph $H$. For induced subgraphs, the fact that no other edges than the edges in $H$ are allowed to be present gives the extra factors $\me^{-x_ix_j}$ {for $\{i,j\}\notin \Ecal_H$}{} in~\eqref{eq:Aind}.

\subsection{Subgraphs on 4 and 5 vertices}\label{sec:motif45}
We apply Theorem~\ref{thm:motifs} to characterize the optimal subgraph configurations on 4 or 5 vertices. 
We find the partitions that maximize~\eqref{eq:maxeqsub} and~\eqref{eq:maxeqind}, and check whether this maximum is unique. If the maximum is indeed unique, then we can use Theorem~\ref{thm:motifs} to calculate the scaling of the number of such (induced) subgraphs. 
Figures~\ref{fig:graphlet4} and~\ref{fig:motif5} show the order of magnitude of the number of induced subgraphs on 4 and 5 vertices obtained in this way, together with the optimizing sets of~\eqref{eq:maxeqind}. 
For example, the optimal values of $S_1,S_2$ and $S_3$ for the subgraph in Figure~\ref{fig:paw} show that
	\begin{equation}
	B^\gind(H)=2-1+\frac{2+0+0-1}{\tau-1}=1+\frac{1}{\tau-1}.
	\end{equation}
By Theorem~\ref{thm:motifs}, the scaling of the induced subgraph in Figure~\ref{fig:paw} then equals 
	\begin{equation}
	n^{(3-\tau)(4-1/(\tau-1))/2+\tfrac{1}{2}}=n^{7-2\tau-\frac{1}{\tau-1}}.
	\end{equation}
The scaling of the other induced subgraphs are computed similarly.

\begin{figure}[tb]
	\definecolor{mycolor1}{RGB}{230,37,52}%
	\tikzstyle{every node}=[circle,fill=black!85,minimum size=8pt,inner sep=0pt]
	\tikzstyle{S1}=[fill=mycolor1!60]
	\tikzstyle{S2}=[fill=mycolor1!60!black]
	\tikzstyle{S3}=[fill=mycolor1]
	\tikzstyle{n1}=[fill=mycolor1!20]
	\begin{subfigure}[t]{0.16\linewidth}
		\centering
		\begin{tikzpicture}
		\tikzstyle{edge} = [draw,thick,-]
		\node[S3] (a) at (90:0.8) {};
		\node[S3] (b) at (162:0.8) {};
		\node[S3] (c) at (234:0.8) {};
		\node[S3] (d) at (306:0.8) {};
		\node[S3] (e) at (378:0.8) {};
		\draw[edge] (a)--(b);
		\draw[edge] (c)--(b);
		\draw[edge] (d)--(b);
		\draw[edge] (a)--(c);
		\draw[edge] (a)--(d);
		\draw[edge] (c)--(e);
		\draw[edge] (e)--(d);
		\draw[edge] (b)--(e);
		\draw[edge] (a)--(e);
		\draw[edge] (c)--(d);
		\end{tikzpicture}	
		\caption{$n^{\frac{5}{2}(3-\tau)}$}
	\end{subfigure}
	\begin{subfigure}[t]{0.16\linewidth}
		\centering
		\begin{tikzpicture}
		\tikzstyle{edge} = [draw,thick,-]
		\node[S3] (a) at (90:0.8) {};
		\node[S3] (b) at (162:0.8) {};
		\node[S3] (c) at (234:0.8) {};
		\node[S3] (d) at (306:0.8) {};
		\node[S3] (e) at (378:0.8) {};
		\draw[edge] (a)--(b);
		\draw[edge] (c)--(b);
		\draw[edge] (d)--(b);
		\draw[edge] (a)--(c);
		\draw[edge] (a)--(d);
		\draw[edge] (c)--(e);
		\draw[edge] (e)--(d);
		\draw[edge] (b)--(e);
		\draw[edge] (a)--(e);
		\end{tikzpicture}	
		\caption{$n^{\frac{5}{2}(3-\tau)}$}
	\end{subfigure}
	\begin{subfigure}[t]{0.16\linewidth}
		\centering
		\begin{tikzpicture}
		\tikzstyle{edge} = [draw,thick,-]
		\node[S3] (a) at (0,0) {};
		\node[S3] (b) at (1,0) {};
		\node[S3] (c) at (0,1) {};
		\node[S3] (d) at (1,1) {};
		\node[S3] (e) at (0.5,1.5) {};
		\draw[edge] (a)--(b);
		\draw[edge] (c)--(b);
		\draw[edge] (d)--(b);
		\draw[edge] (a)--(c);
		\draw[edge] (a)--(d);
		\draw[edge] (c)--(e);
		\draw[edge] (d)--(e);
		\draw[edge] (c)--(d);
		\end{tikzpicture}	
		\caption{$n^{\frac{5}{2}(3-\tau)}$}
	\end{subfigure}
	\begin{subfigure}[t]{0.16\linewidth}
		\centering
		\begin{tikzpicture}
		\tikzstyle{edge} = [draw,thick,-]
		\node[S3] (a) at (0,0) {};
		\node[S3] (b) at (1,0) {};
		\node[S3] (c) at (0,1) {};
		\node[S3] (d) at (1,1) {};
		\node[S3] (e) at (0.5,0.5) {};
		\draw[edge] (a)--(b);
		\draw[edge] (c)--(a);
		\draw[edge] (b)--(e);
		\draw[edge] (a)--(e);
		\draw[edge] (b)--(d);
		\draw[edge] (c)--(e);
		\draw[edge] (d)--(e);
		\draw[edge] (c)--(d);
		\end{tikzpicture}	
		\caption{$n^{\frac{5}{2}(3-\tau)}$}
	\end{subfigure}
	\begin{subfigure}[t]{0.16\linewidth}
		\centering
		\begin{tikzpicture}
		\tikzstyle{edge} = [draw,thick,-]
		\node[S3] (a) at (0,0) {};
		\node[S3] (b) at (1,0) {};
		\node[S3] (c) at (0,1) {};
		\node[S3] (d) at (1,1) {};
		\node[n1] (e) at (0.5,1.5) {};
		\draw[edge] (a)--(b);
		\draw[edge] (c)--(b);
		\draw[edge] (d)--(b);
		\draw[edge] (a)--(c);
		\draw[edge] (a)--(d);
		\draw[edge] (c)--(e);
		\draw[edge] (c)--(d);
		\end{tikzpicture}	
		\caption{$n^{\frac{13}{2}-2\tau}$}
	\end{subfigure}
	\begin{subfigure}[t]{0.16\linewidth}
		\centering
		\begin{tikzpicture}
		\tikzstyle{edge} = [draw,thick,-]
		\node[S3] (a) at (0,0) {};
		\node[S3] (b) at (1,0) {};
		\node[S3] (c) at (0,1) {};
		\node[S3] (d) at (1,1) {};
		\node[S3] (e) at (0.5,0.5) {};
		\draw[edge] (a)--(b);
		\draw[edge] (c)--(a);
		\draw[edge] (b)--(e);
		\draw[edge] (b)--(d);
		\draw[edge] (c)--(e);
		\draw[edge] (c)--(d);
		\draw[edge] (e)--(d);
		\end{tikzpicture}	
		\caption{$n^{\frac{5}{2}(3-\tau)}$}
	\end{subfigure}
	
	\begin{subfigure}[t]{0.16\linewidth}
		\centering
		\begin{tikzpicture}
		\tikzstyle{edge} = [draw,thick,-]
		\node[S3] (a) at (0,0) {};
		\node[S3] (b) at (1,0) {};
		\node[S3] (c) at (0,1) {};
		\node[S3] (d) at (1,1) {};
		\node[S3] (e) at (0.5,1.5) {};
		\draw[edge] (a)--(b);
		\draw[edge] (c)--(b);
		\draw[edge] (d)--(b);
		\draw[edge] (a)--(c);
		\draw[edge] (c)--(e);
		\draw[edge] (d)--(e);
		\draw[edge] (c)--(d);
		\end{tikzpicture}	
		\caption{$n^{\frac{5}{2}(3-\tau)}$}
	\end{subfigure}
	\begin{subfigure}[t]{0.16\linewidth}
		\centering
		\begin{tikzpicture}
		\tikzstyle{edge} = [draw,thick,-]
		\node[S1] (a) at (0,0) {};
		\node[S1] (b) at (1,0) {};
		\node[S2] (c) at (0,1) {};
		\node[S2] (d) at (1,1) {};
		\node[S1] (e) at (0.5,1.5) {};
		\draw[edge] (d)--(b);
		\draw[edge] (c)--(b);
		\draw[edge] (d)--(b);
		\draw[edge] (a)--(c);
		\draw[edge] (a)--(d);
		\draw[edge] (c)--(e);
		\draw[edge] (d)--(e);
		\draw[edge] (c)--(d);
		\end{tikzpicture}	
		\caption{$n^{9-3\tau}$}
		\label{fig:K23}
	\end{subfigure}
	\begin{subfigure}[t]{0.16\linewidth}
		\centering
		\begin{tikzpicture}
		\tikzstyle{edge} = [draw,thick,-]
		\node (a) at (0,0) {};
		\node (b) at (1,0) {};
		\node (c) at (0,1) {};
		\node (d) at (1,1) {};
		\node (e) at (0.5,0.5) {};
		\draw[edge] (a)--(b);
		\draw[edge] (e)--(a);
		\draw[edge] (b)--(e);
		\draw[edge] (e)--(d);
		\draw[edge] (c)--(e);
		\draw[edge] (c)--(d);
		\end{tikzpicture}	
		\caption{depends on $\tau$}
		\label{fig:bowtie}
	\end{subfigure}
	\begin{subfigure}[t]{0.16\linewidth}
		\centering
		\begin{tikzpicture}
		\tikzstyle{edge} = [draw,thick,-]
		\node[S3] (a) at (0,0) {};
		\node[S3] (b) at (1,0) {};
		\node[S3] (c) at (0,1) {};
		\node[S3] (d) at (1,1) {};
		\node[n1] (e) at (0.5,1.5) {};
		\draw[edge] (a)--(b);
		\draw[edge] (d)--(b);
		\draw[edge] (a)--(c);
		\draw[edge] (a)--(d);
		\draw[edge] (c)--(e);
		\draw[edge] (c)--(d);
		\end{tikzpicture}	
		\caption{$n^{\frac{13}{2}-2\tau}$}
	\end{subfigure}
	\begin{subfigure}[t]{0.16\linewidth}
		\centering
		\begin{tikzpicture}
		\tikzstyle{edge} = [draw,thick,-]
		\node[S1] (a) at (0,0) {};
		\node[S2] (b) at (1,0) {};
		\node[S2] (c) at (0,1) {};
		\node[S1] (d) at (1,1) {};
		\node[n1] (e) at (0.5,1.5) {};
		\draw[edge] (a)--(b);
		\draw[edge] (d)--(b);
		\draw[edge] (a)--(c);
		\draw[edge] (b)--(c);
		\draw[edge] (c)--(e);
		\draw[edge] (c)--(d);
		\end{tikzpicture}	
		\caption{$n^{6-2\tau-\frac{1}{\tau-1}}$}
	\end{subfigure}
	\begin{subfigure}[t]{0.16\linewidth}
		\centering
		\begin{tikzpicture}
		\tikzstyle{edge} = [draw,thick,-]
		\node[S3] (a) at (0,0) {};
		\node[S3] (b) at (1,0) {};
		\node[S3] (c) at (0,1) {};
		\node[S3] (d) at (1,1) {};
		\node[S3] (e) at (0.5,1.5) {};
		\draw[edge] (a)--(b);
		\draw[edge] (d)--(b);
		\draw[edge] (a)--(c);
		\draw[edge] (e)--(d);
		\draw[edge] (c)--(e);
		\draw[edge] (c)--(d);
		\end{tikzpicture}	
		\caption{$n^{\frac{5}{2}(3-\tau)}$}
	\end{subfigure}
	
	\begin{subfigure}[t]{0.16\linewidth}
		\centering
		\begin{tikzpicture}
		\tikzstyle{edge} = [draw,thick,-]
		\node[S3] (a) at (0,0) {};
		\node[S3] (b) at (1,0) {};
		\node[S3] (c) at (0,1) {};
		\node[S3] (d) at (1,1) {};
		\node[S3] (e) at (0.5,0.5) {};
		\draw[edge] (a)--(b);
		\draw[edge] (c)--(a);
		\draw[edge] (b)--(e);
		\draw[edge] (b)--(d);
		\draw[edge] (c)--(e);
		\draw[edge] (c)--(d);
		\end{tikzpicture}	
		\caption{$n^{\frac{5}{2}(3-\tau)}$}
		\label{fig:m5dom}
	\end{subfigure}
	\begin{subfigure}[t]{0.16\linewidth}
		\centering
		\begin{tikzpicture}
		\tikzstyle{edge} = [draw,thick,-]
		\node (a) at (0,0) {};
		\node (b) at (1,0) {};
		\node (c) at (0.5,1) {};
		\node[n1] (d) at (0.5,1.5) {};
		\node (e) at (0.5,0.5) {};
		\draw[edge] (a)--(b);
		\draw[edge] (e)--(a);
		\draw[edge] (b)--(e);
		\draw[edge] (e)--(c);
		\draw[edge] (c)--(d);
		\end{tikzpicture}	
		\caption{depends on $\tau$}
	\end{subfigure}
	\begin{subfigure}[t]{0.16\linewidth}
		\centering
		\begin{tikzpicture}
		\tikzstyle{edge} = [draw,thick,-]
		\node[S3] (a) at (90:0.8) {};
		\node[S3] (b) at (162:0.8) {};
		\node[S3] (c) at (234:0.8) {};
		\node[S3] (d) at (306:0.8) {};
		\node[S3] (e) at (378:0.8) {};
		\draw[edge] (a)--(b);
		\draw[edge] (c)--(b);
		\draw[edge] (d)--(c);
		\draw[edge] (a)--(e);
		\draw[edge] (d)--(e);
		\end{tikzpicture}	
		\caption{$n^{\frac{5}{2}(3-\tau)}$}
	\end{subfigure}
	\begin{subfigure}[t]{0.16\linewidth}
		\centering
		\begin{tikzpicture}
		\tikzstyle{edge} = [draw,thick,-]
		\node[S1] (a) at (0,0) {};
		\node[S1] (b) at (1,0) {};
		\node[n1] (c) at (0,1) {};
		\node[n1] (d) at (1,1) {};
		\node[S2] (e) at (0.5,0.5) {};
		\draw[edge] (a)--(b);
		\draw[edge] (e)--(a);
		\draw[edge] (b)--(e);
		\draw[edge] (e)--(d);
		\draw[edge] (c)--(e);
		\end{tikzpicture}	
		\caption{$n^{7-2\tau}$}
	\end{subfigure}
	\begin{subfigure}[t]{0.16\linewidth}
		\centering
		\begin{tikzpicture}
		\tikzstyle{edge} = [draw,thick,-]
		\node[S2] (a) at (0,0) {};
		\node[S2] (b) at (1,0) {};
		\node[n1] (c) at (0,1) {};
		\node[n1] (d) at (1,1) {};
		\node[S1] (e) at (0.5,0.5) {};
		\draw[edge] (a)--(b);
		\draw[edge] (e)--(a);
		\draw[edge] (b)--(e);
		\draw[edge] (b)--(d);
		\draw[edge] (c)--(a);
		\end{tikzpicture}	
		\caption{$n^{3-\tau+\frac{2}{\tau-1}}$}
	\end{subfigure}
	\begin{subfigure}[t]{0.16\linewidth}
		\centering
		\begin{tikzpicture}
		\tikzstyle{edge} = [draw,thick,-]
		\node[S3] (a) at (0,0) {};
		\node[S3] (b) at (1,0) {};
		\node[S3] (c) at (0,1) {};
		\node[S3] (d) at (1,1) {};
		\node[n1] (e) at (0.5,1.5) {};
		\draw[edge] (a)--(b);
		\draw[edge] (d)--(b);
		\draw[edge] (a)--(c);
		\draw[edge] (c)--(e);
		\draw[edge] (c)--(d);
		\end{tikzpicture}	
		\caption{$n^{\frac{13}{2}-2\tau}$}
	\end{subfigure}
	
	\begin{subfigure}[t]{0.16\linewidth}
		\centering
		\begin{tikzpicture}
		\tikzstyle{edge} = [draw,thick,-]
		\node[S3] (a) at (0,0) {};
		\node[S3] (b) at (1,0) {};
		\node[S3] (c) at (0,1) {};
		\node[n1] (d) at (1,1) {};
		\node[n1] (e) at (0.5,1.5) {};
		\draw[edge] (a)--(b);
		\draw[edge] (d)--(b);
		\draw[edge] (a)--(c);
		\draw[edge] (c)--(e);
		\end{tikzpicture}	
		\caption{$n^{\frac{11-3\tau}{2}}$}
	\end{subfigure}
	\begin{subfigure}[t]{0.16\linewidth}
		\centering
		\begin{tikzpicture}
		\tikzstyle{edge} = [draw,thick,-]
		\node[n1] (a) at (0,0) {};
		\node[n1] (b) at (1,0) {};
		\node[S2] (c) at (0,1) {};
		\node[S1] (d) at (1,1) {};
		\node[n1] (e) at (0.5,1.5) {};
		\draw[edge] (d)--(b);
		\draw[edge] (a)--(c);
		\draw[edge] (c)--(e);
		\draw[edge] (c)--(d);
		\end{tikzpicture}	
		\caption{$n^{4-\tau+\frac{1}{\tau-1}}$}
	\end{subfigure}
	\begin{subfigure}[t]{0.16\linewidth}
		\centering
		\begin{tikzpicture}
		\tikzstyle{edge} = [draw,thick,-]
		\node[n1] (a) at (0,0) {};
		\node[n1] (b) at (1,0) {};
		\node[n1] (c) at (0,1) {};
		\node[n1] (d) at (1,1) {};
		\node[S2] (e) at (0.5,0.5) {};
		\draw[edge] (e)--(a);
		\draw[edge] (b)--(e);
		\draw[edge] (e)--(d);
		\draw[edge] (c)--(e);
		\end{tikzpicture}	
		\caption{$n^{\frac{4}{\tau-1}}$}
	\end{subfigure}

	\vspace{-0.5cm}
	\begin{subfigure}{\linewidth}
		\centering
		\begin{tikzpicture}
		\node[S2,label={[label distance=0.05cm]0:$n^{1/(\tau-1)}$}] (a) at (5,0) {};
		\node[S3,label={[label distance=0cm]0:$\sqrt{n}$}] (b) at (3.6,0) {};
		\node[S1,label={[label distance=0.05cm]0:$n^{(\tau-2)/(\tau-1)}$}] (c) at (1,0) {};
		\node[n1,label={[label distance=0.05cm]0:$1$}] (c) at (0,0) {};
		\node[label={[label distance=0.05cm]0:depends on $\tau$}] (d) at (7,0) {};
		\end{tikzpicture}
	\end{subfigure}
	\vspace{-0.8cm}
	\caption{The scaling of the number of subgraphs in $n$ for induced subgraphs on 5 vertices. The vertex colors correspond to the typical degrees.}
	\label{fig:motif5}
\end{figure}

Most induced subgraphs in Figures~\ref{fig:graphlet4} and~\ref{fig:motif5} satisfy the constraint in Theorem~\ref{thm:motifs} that the solution to the optimization problem~\eqref{eq:maxeqsub} or~\eqref{eq:maxeqind} is unique. However, the gray vertices in Figure~\ref{fig:graphlet4} do not have unique optimizers, so that our theorems do not apply. Still, a similar analysis as in Section~\ref{sec:opt} shows that there exist ranges of degrees that give the major contribution to the rescaled number of such (induced) subgraphs. The only difference is that these ranges are wider than for the vertices with unique maximizers. For example, for the diamond subgraph in Figure~\ref{fig:squareextra} the major contribution is from vertices where the degrees of vertices at each side of an edge $\{i,j\}$ in the square around the diamond satisfy $D_iD_j=\Theta(n)$. Note that having all degrees proportional to $\sqrt{n}$ therefore is one of the main contributors. However, contributions where the bottom left vertex and the top right vertex have degrees proportional to $n^{\alpha}$ and the other two vertices have degrees $n^{1-\alpha}$ give an equal contribution for other values of $\alpha$. Using that $D_iD_j$ follows a power-law distribution with exponent $\tau$ with an extra logarithmic factor~\cite[Eq. (2.16)]{hofstad2014} then gives the extra factor $\log(n)$ in Figure~\ref{fig:squareextra}. 

The bow tie in Figure~\ref{fig:bowtie} has a unique optimal solution to~\eqref{eq:maxeqsub}, but it depends on $\tau$. For $\tau<7/3$, the maximum of~\eqref{eq:maxeqind} is uniquely attained at 0, so that the optimal composition is with all vertices of degree $\Theta(\sqrt{n})$. On the other hand, when $\tau>7/3$, $S_1$ contains the degree 2 vertices while the middle vertex is in $S_2$. This partition gives a contribution to~\eqref{eq:maxeqind} of
	\begin{equation}
	4-1-\frac{2\cdot 2}{\tau-1}= \frac{3\tau - 7}{\tau-1},
	\end{equation}
which is larger than zero when $\tau>7/3$.
Thus, for $\tau$ larger than $7/3$, the major contribution is when the middle vertex has degree $n^{1/(\tau-1)}$, and the other vertices have degrees $n^{(\tau-2)/(\tau-1)}$. 

When the maximal contribution to an induced subgraph comes from vertices with degrees proportional to $\sqrt{n}$, then by Theorem~\ref{thm:sqrtsub}, the number of such induced subgraphs converges to a constant when properly rescaled. When the maximal contribution contains vertices in $S_2$ and $S_1$, this may not hold. For example, counting the number of induced claws of Figure~\ref{fig:wedge4} is similar to counting the number of sets of three neighbors for every vertex. The only sets of neighbors that we do not count, are neighbors that are connected. This is a small fraction of the pairs of neighbors~\cite[Eq. (5)-(7)]{hofstad2017b}, thus the number of claws is approximately equal to
	\begin{equation}
	\cj{\sum_{v\in[n]}\tfrac 1 6 D_v(D_v-1)(D_v-2)\approx \tfrac{1}{6}\sum_{v\in[n]}D_v^3.}
	\end{equation}
Since the degrees are an i.i.d.\ sample from a power-law distribution, $\sum_{\cj{v\in[n]}}D_v^3$ converges to a stable law when normalized properly. 
Thus, when vertices of degrees proportional to $n^{1/(\tau-1)}$ contribute, the leading order of the number of (induced) subgraphs may contain stable random variables, in contrast {to} the deterministic leading order for $\sqrt{n}$ degrees of Theorem~\ref{thm:sqrtsub}.

The scaling of the number of (non-induced) subgraphs can be deduced from Figure~\ref{fig:graphlet4}. 
For example, we count the number of square subgraphs (the subgraph of Figure~\ref{fig:square}) by adding the contributions from the induced subgraphs in Figures~\ref{fig:K4},~\ref{fig:squareextra} and~\ref{fig:square}, that all contain a square, which shows that a square occurs $\Theta(n^{6-2\tau}\log(n))$ times as a subgraph.
The major contribution to the number of square subgraphs is from the induced subgraphs in Figure~\ref{fig:squareextra}, which indeed contains a square, and occurs more frequently than the subgraphs of Figures~\ref{fig:K4} and~\ref{fig:square}. In this manner we can infer the order of magnitude of the number of subgraphs from the number of induced subgraphs. 

\subsection{Discussion and outlook}

\paragraph{Uniqueness of the optimum.}
Theorem~\ref{thm:motifs} only holds when the optimum of~\eqref{eq:maxeqsub}, respectively~\eqref{eq:maxeqind}, is unique. Figures~\ref{fig:graphlet4} and~\ref{fig:motif5} show that for most subgraphs on 4 or 5 vertices, this is indeed the case. In Section~\ref{sec:maxcont}, we show that~\eqref{eq:maxeqsub} and~\eqref{eq:maxeqind} can both be interpreted as piecewise linear optimization problems over the optimal degrees of the vertices that together form the subgraph. Thus, if the optimum is not unique, then it is attained by an entire range of degrees. In Section~\ref{sec:maxcont} we show that in this situation the optimum is attained for vertices $v_i,v_j$ with degrees $D_{v_i},D_{v_j}$ such that $D_{v_i}D_{v_j}=\Theta(n)$ across some edges $\{i,j\}\in \Ecal_H$. One such example is the diamond of Figure~\ref{fig:squareextra} discussed in Section~\ref{sec:motif45}.
We believe that the number of subgraphs where the optimum is not unique scales as in Theorem~\ref{thm:motifs} with some additional multiplicative factors of $\log(n)$. Proving this remains open for further research. 


\paragraph{ Automorphisms of $H$.}
An automorphism of a graph $H$ is a map $V_H\mapsto V_H$ such that the resulting graph is isomorphic to $H$. 
In Theorems~\ref{thm:motifs} and~\ref{thm:sqrtsub} we count automorphisms of $H$ as separate copies of $H$, so that we may count multiple copies of $H$ on one set of vertices. Since $|V_H|$ is fixed, and Theorem~\ref{thm:motifs} only considers the scaling of the number of subgraphs, this does not influence Theorem~\ref{thm:motifs}. Because Theorem~\ref{thm:sqrtsub} studies the exact scaling of the number of subgraphs, to count the number of subgraphs without automorphisms, one should divide the results of Theorem~\ref{thm:sqrtsub} by the number of automorphisms of $H$.

\paragraph{Self-averaging.}

A random variable is called \emph{self-averaging} if its coefficient of variation tends to zero, otherwise it is called non-self-averaging. 
When the degree distribution follows a power-law with exponent $\tau\in(2,3)$, the number of subgraphs may be non-self-averaging~\cite{ostilli2014}, so that
\begin{equation}
\limsup_{n\to\infty}\frac{\Var{N^\gsub(H)}}{\Exp{N^\gsub(H)}^2}\neq 0.
\end{equation}
One such example is the triangle. While the triangle subgraph satisfies the conditions of Theorem~\ref{thm:sqrtsub}, so that the rescaled number of triangles converges in probability to a constant, it was shown in~\cite{ostilli2014}, the number of triangles is non-self-averaging in the annealed sense when $\tau$ is close to 3. 
\cs{This indicates that most realizations of $\ECMnD$ will have a number of triangles that is close to the value predicted by Theorem~\ref{thm:sqrtsub}. However, since the number of triangles is non-self-averaging making its standard deviation quite large, some realizations will have a number of triangles that is much larger or smaller than the value predicted in Theorem~\ref{thm:sqrtsub}.}



\paragraph{Other random graph models.}  
An interesting question is whether Theorems~\ref{thm:motifs} and \ref{thm:sqrtsub} also apply to other models that create simple power-law random graphs. 
A very natural model for simple power-law random graphs is the uniform random graph, which samples a uniform graph from the ensemble of all simple graphs on a given degree sequence, which we analyze  for triangles using similar techniques as in this paper in~\cite{gao2018}.

Another random graph model that generates simple power-law random graphs is the rank-1 inhomogeneous random graph~\cite{chung2002,boguna2003}. In this model, vertices have weights $h_i$, where the weights are an i.i.d.\ sample of a power-law random variable with exponent $\tau\in(2,3)$. Then, two vertices are connected with probability $f_n(h_i,h_j)$. Two standard connection probability functions are $f_n(h_i,h_j)=\min(h_ih_j/(\mu n),1)$~\cite{chung2002}, and $f_n(h_i,h_j)=1-\me^{-h_ih_j/(\mu n)}$~\cite{bollobas2007}. Conditionally on the weight sequence, the edge statuses are independent, which is different from the erased configuration model, where the edge statuses are not independent, even when conditioning on the degree sequence. 
We prove Theorems~\ref{thm:motifs} and \ref{thm:sqrtsub} for the erased configuration model by using the approximation $\Probn{X_{ij}=1}\approx 1-\me^{-D_iD_j/L_n}$. Therefore, Theorems~\ref{thm:motifs} and \ref{thm:sqrtsub} hold also for the rank-1 inhomogeneous random graph with these connection probabilities instead~\cite{stegehuis2019b}.

A third model that creates simple power-law random graphs, is the hyperbolic random graph where vertices are sampled in a disk, and connected if their hyperbolic distance is sufficiently small~\cite{krioukov2010}. The geometry in the hyperbolic random graph makes the presence of triangles and other subgraphs containing cycles likely.
By Theorem~\ref{thm:sqrtsub}, a complete graph on $k$ vertices occurs $\Theta(n^{\frac{k}{2}(3-\tau)})$ times as a subgraph in $\ECMn$. Interestingly, this is also true for hyperbolic random graphs for $k$ sufficiently large~\cite{Blasius2017}. 
 It would be interesting to investigate the presence of other subgraphs in hyperbolic random graphs. 

\cs{Another class of popular models, which create power-law random graphs dynamically, are those that incorporate preferential
	attachment. In these models, subgraph counts scale significantly differently from the erased configuration model and uniform random graphs~\cite{garavaglia2018}.}

\section{Overview of the proofs}\label{sec:overview}
We now provide an overview of the proof structure of Theorems~\ref{thm:motifs} and \ref{thm:sqrtsub}. 
Our main results study the annealed version $\ECMn$, with random degree sequence. 
In the proofs of Theorems~\ref{thm:motifs} and \ref{thm:sqrtsub}, we often first study the quenched version of $\ECMnD$ instead, where the degree sequence $\boldsymbol{D}$ is fixed. 

We relate $\cj{L_n=\sum_{v\in[n]}D_v}$, the total number of half-edges before erasure, to its expected value $\mu n$ by defining  the event
\begin{equation}\label{eq:Jn}
J_n = \big\{ \abs{L_n-\mu n}\leq n^{2/\tau}\big\}.
\end{equation}
By~\cite[Lemma~2.3]{hofstad2017}, $\Prob{J_n}\to 1$ as $n\to\infty$. When we condition on the degree sequence, we will work on the event $J_n$, so that we can write $L_n=\mu n(1+o(1))$. Similarly, when we work with $\expec_n$ and $\text{Var}_n$, we condition on the event $J_n$. \cs{We do not include $J_n$ into the notation of $\mathbb{P}_n$, since given $\boldsymbol{D}$, $J_n$ either happens with probability one, or with probability zero. This could be treated more formally by denoting
	\begin{equation}
	\mathbb{P}_n(\mathcal{E}) = \mathbbm{1}_{J_n}\mathbb{P}(\mathcal{E}\mid \boldsymbol{D}),
	\end{equation}
	but keep notation light, when using $\mathbb{P}_n$, we always assume the event $J_n$ to hold.}

Denote the indicator that an edge is present between vertices $u$ and $v$ by $X_{u,v}$.
To obtain the probability that a specific subgraph is present on a given set of vertices, we investigate the probability of a set of edges being present in the erased configuration model. In $\ECMnD$ (see the proof of Lemma~\ref{lem:condprob} for a more precise statement),
\begin{equation}\label{eq:pedgeappr}
\Probn{X_{u,v}=1}\approx 1-\me^{-D_uD_v/L_n}.
\end{equation}
However, subgraphs often contain more than just one edge, and edges in $\ECMnD$ are not present independently. In Section~\ref{sec:avoid}, we show that these dependencies are weak, so that we can use the approximation~\eqref{eq:pedgeappr} for all edges in a subgraph as if they were present independently.

We then compute the probability that a subgraph is present on a specific set of vertices as a function of their degrees, which shows that 
	\begin{equation}
	N^\gsub(H,M_n^{(\boldsymbol{\alpha})}(\varepsilon))=\Theta_{\sss{\prob}}\Big( n^{k+(1-\tau)\sum_i\alpha_i} \ \ \prod_{\mathclap{\{i,j\}\in \Ecal_H:\alpha_i+\alpha_j<1}} \ \  n^{\alpha_1+\alpha_j-1}\Big).
	\end{equation}
To prove Theorem~\ref{thm:motifs}(ii) we optimize this over $\boldsymbol{\alpha}=(\alpha_1,\dots,\alpha_k)$. Here $\varepsilon$ does not appear in the scaling, since it is independent of $n$. To prove Theorem~\ref{thm:motifs}(i) for $\varepsilon_n\downarrow 0$, we analyze $N^\gsub(H,M_n^{(\boldsymbol{\alpha})}(\varepsilon))$ in more detail in Section~\ref{sec:proofsec2}.

To prove the sharp asymptotics of Theorem~\ref{thm:sqrtsub}, we compute the contribution to the expectation and the variance of the number of subgraphs in $\ECMnD$ from vertices with degrees proportional to $\sqrt{n}$ in Section~\ref{sec:prooflem1}. We use a second moment method to show that the number of subgraphs concentrates around its expectation in $\ECMnD$. We then investigate the asymptotic behavior of this expectation. A first moment method which shows that the expected contribution to the number of subgraphs from vertices with other degrees is small completes the proof of Theorem~\ref{thm:sqrtsub}.

Theorem~\ref{thm:sqrtsub} for induced subgraphs can be proven similarly, the only difference being that we have to take into account that to form an induced subgraph, some edges are not allowed to be present in $\ECMn$. We explain how this changes the proof of Theorems~\ref{thm:motifs} and \ref{thm:sqrtsub} in more detail in Section~\ref{sec:graphlets}.

\section{Maximum contribution: proof of Theorem~\ref{thm:motifs}(ii)}
\label{sec:maxcont}
\subsection{The probability of avoiding a subgraph}\label{sec:avoid}
The edges of a subgraph are not present independently. The following lemma computes the probability that an edge is not present conditionally on other edges not being present:
\begin{lemma}\label{lem:condprob}
	Fix $m\in\mathbb{N}$ and $\varepsilon>0$. Let $\{\{u_i,v_i\}\}_{i\in[m+1]}$ be such that $u_i,v_i\in[n]$ for all $i\in[m+1]$ and $\{u_{m+1},v_{m+1}\}\neq \{u_i,v_i\}$ for all $i\in[m]$. Let 
	\begin{equation}
	\mathcal{E} = \{X_{u_i,v_i}=0\ \forall i \in [m]  \}.
	\end{equation}
If $D_{u_i},D_{v_i}\leq n^{1/(\tau-1)}/\varepsilon$ for $i\in[m+1]$, then
	\begin{equation}\label{eq:psubbound}
	\Probn{X_{u_{m+1},v_{m+1}}=0\mid \mathcal{E}}=\bigO{\me^{-{D_{u_{m+1}}D_{v_{m+1}}}/{4L_n}}}.
	\end{equation}
	\cs{Furthermore, when $D_{u_{m+1}}D_{v_{m+1}}\leq n/\varepsilon$, for $\gamma\in(\frac{\tau-2}{2(\tau-1)},\frac{\tau-2}{\tau-1})$,
	\begin{equation}\label{eq:psubdetail}
	\begin{aligned}[b]
	\Probn{X_{u_{m+1},v_{m+1}}=0\mid \mathcal{E}}& =\me^{-\frac{D_{u_{m+1}}D_{v_{m+1}}}{L_n}} \left(1+O \bigg(\frac{D_{u_{m+1}}D_{v_{m+1}}}{L_n}n^{-\gamma}\bigg)\right).
	\end{aligned}
	\end{equation}}
\end{lemma}
Throughout the rest of the paper, we mainly use~\eqref{eq:psubbound} to bound the probability that an edge between two high-degree vertices is absent, whereas we use~\eqref{eq:psubdetail} to compute asymptotic identities for the probability that a subgraph is present.

\begin{proof}
	For $m=0$ the claim is proven in~\cite[Eq (4.6) and (4.9)]{hofstad2005}, which states that for two vertices $u$ and $v$ with $D_u>D_v$, 
	\begin{equation}\label{eq:pijsmall}
	\Probn{X_{u,v}=0}=\me^{-D_uD_v/L_n}+O({D_u^2 D_v}/{L_n^2}),
	\end{equation}
and that, by using~\cite[Eq. (4.5)]{hofstad2005} \cs{as well as $L_n-2i+1\leq L_n$ and $1-x\leq \me^{-x}$,}
	\begin{equation}\label{eq:pijlargebound}
	\cs{\Probn{X_{u,v}=0}\leq \prod_{i=1}^{D_u/2}\Big(1-\frac{D_v}{L_n-2i+1}\Big)\leq \me^{-D_uD_v/2L_n}.}
	\end{equation}
Thus we assume that $m\geq 1$. Note that $\Omega:=\{\{u_i,v_i\}\}_{i\in[m]}$ may contain the same vertices multiple times. Denote the number of distinct vertices in ${\{\{u_i,v_i\}\}_{i\in[m]}}$ by $r$, and denote these distinct vertices by $w_1,\ldots,w_r$. Let $u_{m+1},v_{m+1}$ correspond to $w_r$ and $w_{r-1}$ (if they are present in $w_1,\ldots,w_r$ at all). The ordering of the other vertices may be arbitrary. 

We now construct $\ECMnD$ conditionally on the edges $\Omega$ not being present. We pair the half-edges of the erased configuration model attached to $w_1,\ldots,w_r$. First we pair all half-edges adjacent to $w_1$. Since we condition on the edges $\Omega$ not being present, no half-edge from $w_1$ is allowed to pair to any of its neighbors in $\Omega$. After that, we pair all remaining half-edges from $w_2$, conditionally on these half-edges not connecting to one of the neighbors of $w_2$ in $\Omega$, and so on. We continue until all of the forbidden edges $\Omega$ have at least one incident vertex whose half-edges have already been paired. Then, if we pair the rest of the half-edges, we know that none of the edges in $\Omega$ are present. 

Let $B$ denote the number of vertices we have to pair before all of the forbidden edges $\Omega$ have at least one incident vertex whose half-edges have already been paired. We never have to pair half-edges adjacent to $u_{m+1}$ or to $v_{m+1}$ (if they are present in ${\{\{u_i,v_i\}\}_{i\in[m]}}$), since they are last in the ordering, and $\{u_{m+1},v_{m+1}\}$ is not present in ${\{\{u_i,v_i\}\}_{i\in[m]}}$. Therefore, the half-edges incident to all forbidden neighbors of $u_{m+1}$ and $v_{m+1}$ in $\Omega$ have already been paired before arriving at $u_{m+1}$ or $v_{m+1}$. Let $\hat{X}_{u,v}$ denote the number of half-edges between $u$ and $v$ in the configuration model, so that the edge indicator of the erased configuration model can be written as $X_{u,v}=\mathbbm{1}\{\hat{X}_{u,v}>0\}$. Furthermore, let $\mathcal{F}_{\leq s}=\sigma((\hat{X}_{w_{i},j})_{i\leq s, j\in[n]})$ be the information about the pairings that have been constructed up to time $s$. 
	
	After pairing the half-edges incident to vertices in $[B]$, denote 
	\begin{equation}
	\tilde{L}_n=L_n-2\sum_{i\in[B]}(D_{w_i}-\hat{X}_{w_i,w_i}),
	\end{equation}
\cs{which equals the remaining half-edges after pairing the half-edges incident to $(w_i)_{i\in[{B}]}$. Here we subtract $D_{w_i}$ twice, since the pairing of every half-edge removes one half-edge incident to $w_i$, and one other half-edge, unless it is paired to another half-edge incident to $w_i$, giving rise to the term $\hat{X}_{w_i,w_i}$.}
	Define $\tilde{D}_{u_{m+1}}=D_{u_{m+1}}-\sum_{i\in [B]}\hat{X}_{i,u_{m+1}}$, and define ${D}_{\tilde{v}_{m+1}}$ similarly. These quantities are all measurable with respect to $\mathcal{F}_{\leq B}$. The probability that $u_{m+1}$ does not pair to $v_{m+1}$ is the probability that ${u}_{m+1}$ of degree $\tilde{D}_{u_{m+1}}$ does not connect to ${v}_{m+1}$ of degree $\tilde{D}_{v_{m+1}}$ in a configuration model with $\tilde{L}_n$ half-edges. Thus, using~\eqref{eq:pijsmall},
	\begin{equation}\label{eq:probHcond}
	\Probn{X_{u_{m+1},v_{m+1}}=0\mid \mathcal{F}_{\leq B}}= \me^{-\tilde{D}_{u_{m+1}}\tilde{D}_{v_{m+1}}/\tilde{L}_n}+\bigO{{\tilde{D}_{u_{m+1}}^2\tilde{D}_{v_{m+1}}}/{\tilde{L}_n^2}},
	\end{equation}
	where we have assumed w.l.o.g. that $\tilde{D}_{u_{m+1}}\geq \tilde{D}_{v_{m+1}}$. 
%
	
We now proceed to prove~\eqref{eq:psubdetail}. The probability that the $j$th half-edge incident to $w_i$ pairs to $u_{m+1}$ can be bounded as
	\begin{equation}
	\begin{aligned}[b]
	\Probn{j\text{th half-edge pairs to }u_{m+1}}& \leq \frac{D_{u_{m+1}}}{L_n-2j-3-2\sum_{s\in[i-1]}D_{w_s}}\\
	&\leq KD_{u_{m+1}}/L_n,
	\end{aligned}
	\end{equation}
	for some $K>0$. We have to pair at most $D_{w_i}\leq n^{1/(\tau-1)}/\varepsilon$ half-edges, since some of the half-edges incident to $w_i$ may have been used already in previous pairings. Therefore, we can stochastically dominate $\hat{X}_{w_i,u_{m+1}}$ by $Y_{\cs{w_i}}$, where $Y_{\cs{w_i}}\sim\text{Bin}(n^{1/(\tau-1)}/\varepsilon,KD_{u_{m+1}}/L_n)$, so that $\Exp{Y_{w_i}}= K_1n^{-\beta}D_{u_{m+1}}$ for some $K_1$, where $\beta=(\tau-2)/(\tau-1)$.
	
	Choose $\gamma\in(\frac{\tau-2}{2(\tau-1)},\frac{\tau-2}{\tau-1})$. By the Chernoff bound, for some $\tilde{K}>0$,
	\begin{align}
	\Prob{Y_{\cs{w_i}}>K_1n^{-\beta}D_{u_{m+1}}(1+n^\gamma)}\leq \me^{-\tilde{K}n^{2\gamma-\beta}D_{u_{m+1}}}.
	\end{align}
	Define the events
	\begin{align}
	\mathcal{B}_{n,u}& =\{\exists i\in[B]: \hat{X}_{w_i,u_{m+1}}>K_1n^{-\beta}D_{u_{m+1}}(1+n^\gamma)\}, \\
	 \mathcal{B}_{n,v}& =\{\exists i\in[B]: \hat{X}_{w_i,v_{m+1}}>K_1n^{-\beta}D_{v_{m+1}}(1+n^\gamma)\},
	\end{align}
and let $\mathcal{B}_{n,u}^c$ and $\mathcal{B}_{n,v}^c$ denote their respective complements, so that, by a union bound,
	\begin{align}\label{eq:Bnuc}
	\Prob{\mathcal{B}_{n,u}^c}\geq 1- B \me^{-\tilde{K}n^{2\gamma-\beta}D_{u_{m+1}}}.
	\end{align}	
	On the event $\mathcal{B}_{n,u}^c$,
	\begin{equation}
	\tilde{D}_{u_{m+1}}\geq D_{u_{m+1}}\big(1-\sum_{i\in[B]}\hat{X}_{w_i,u_{m+1}}\big)=D_{u_{m+1}}(1+O(n^{\gamma-\beta})).
	\end{equation}
	
Similarly, $\tilde{D}_{v_{m+1}}=D_{v_{m+1}}(1+O(n^{\gamma-\beta}))$ on $\mathcal{B}_{n,v}^c$, where $\Prob{\mathcal{B}_{n,v}^c}\geq 1- B \me^{-\tilde{K}n^{2\gamma-\beta}D_{v_{m+1}}}$.
	Then, when $D_{u_{m+1}}D_{v_{m+1}}=O(n)$ as assumed for \eqref{eq:psubdetail},~\eqref{eq:probHcond} becomes
	\begin{align}\label{eq:Pxuxvsmall}
	&\Probn{X_{u_{m+1},v_{m+1}}=0\mid \mathcal{F}_{B+1},\mathcal{B}_{n,u}^c,\mathcal{B}_{n,v}^c}\nonumber\\
	& =\me^{-\frac{D_{{u}_{m+1}}D_{{v}_{m+1}}}{L_n}(1+O(n^{-\gamma}))}+O\bigg(\frac{D_{{u}_{m+1}}^2D_{{v}_{m+1}}}{{L}_n^2}\bigg)\nonumber\\
	& =\me^{-\frac{D_{{u}_{m+1}}D_{{v}_{m+1}}}{{L}_n}}\left(1+O\left(\frac{D_{u_{m+1}}D_{v_{m+1}}}{L_n}n^{-\gamma}\right)\right),
	\end{align}
	where we have used that $D_{u_{m+1}}=O(n^{1/(\tau-1)})$.
Furthermore, $2\gamma-\beta>0$, whereas by assumption $D_{u_{m+1}}D_{v_{m+1}}/L_n=O(1)$, so that~\eqref{eq:Bnuc} together with~\eqref{eq:Pxuxvsmall} proves~\eqref{eq:psubdetail}.

To prove~\eqref{eq:psubbound}, we use~\eqref{eq:pijlargebound} and the fact that on the event $\mathcal{B}_{n,u}^c\cap \mathcal{B}_{n,v}^c$,  ${D}_{v_{m+1}}{D}_{u_{m+1}}\geq \tilde{D}_{v_{m+1}}\tilde{D}_{u_{m+1}}/2$ for $n$ sufficiently large to obtain
\begin{align}
\Probn{X_{u_{m+1},v_{m+1}}=0\mid \mathcal{F}_{B+1}\mathcal{B}_{n,u}^c,\mathcal{B}_{n,v}^c} & \leq \me^{-\tilde{D}_{u_{m+1}}\tilde{D}_{v_{m+1}}/2\tilde{L}_n}\nonumber\\
& \leq   \me^{-\tilde{D}_{u_{m+1}}\tilde{D}_{v_{m+1}}/2{L}_n}\nonumber\\
& \leq   \me^{-{D}_{u_{m+1}}{D}_{v_{m+1}}/4{L}_n}.
\end{align}
Combining this with~\eqref{eq:Bnuc} and the fact that $2\gamma-\beta>0$ completes the proof of~\eqref{eq:psubbound}.
\end{proof}

\subsection{An optimization problem}\label{sec:opt}
We now use Lemma~\ref{lem:condprob} to study the probability that a subgraph is present on vertices $(v_1, \ldots, v_k)$ of specific degrees. 
Assume that $D_{{v_i}}\in[\varepsilon,1/\varepsilon]n^{\alpha_i}$ with $\alpha_i\in[0,1/(\tau-1)]$ for all $i$, so that $D_{{v_i}}=\Theta(n^{\alpha_i})$.

Let $H$ be a subgraph on $k$ vertices labeled as $1,\ldots,k$, and with $m$ edges labeled a ${\Ecal_H=\big\{\{i_1,j_1\},\ldots,\{i_m,j_m\}\big\}}$. Furthermore, let $\ECMnD|_{\boldsymbol{v}}$ be the induced subgraph of $\ECMnD$ on the vertices ${\boldsymbol{v}=(v_1,\ldots,v_k)}$. We aim to study the probability that this occurs.

When $\alpha_i+\alpha_j<1$, by~\eqref{eq:pijsmall},
	\begin{equation}
	\label{eq:pijsmallalpha}
	\Probn{X_{v_i,v_j}=1}=\big(1-\me^{-\Theta(n^{\alpha_i+\alpha_j-1})}\big)(1+o(1))=\Theta\left(n^{\alpha_i+\alpha_j-1}\right).
	\end{equation}
Furthermore, by~\eqref{eq:pijsmall}, $\Probn{X_{v_i,v_j}=1}=\Theta(1)$ and $\Probn{X_{v_i,v_j}=0}=\Theta(1)$ when $\alpha_i+\alpha_j=1$. When $\alpha_i+\alpha_j>1$ instead, by~\eqref{eq:pijlargebound},
	\begin{equation}
	\label{eq:pijlarge2}
	\Probn{X_{v_i,v_j}=1}=1-O\big(\me^{-n^{\alpha_i+\alpha_j-1}/(4\mu )}\big),
	\end{equation}
so that it equals 1 minus a stretched exponentially small term.
For vertices $v_i,v_j\in [n]$ with $D_{v_i}\in[\varepsilon,1/\varepsilon]n^{\alpha_i}$ and $D_{v_j}\in[\varepsilon,1/\varepsilon]n^{\alpha_j}$, denote 
	\begin{equation}
	w_{i,j}=n^{\alpha_i+\alpha_j-1-\gamma},
	\end{equation}
	 with $\gamma$ as in~\eqref{eq:psubdetail}.
By Lemma~\ref{lem:condprob}, for any set of $m$ edges in $[n]$, and $D_{u_p}\in[\varepsilon,1/\varepsilon]n^{\alpha_{i_p}}, D_{v_p}\in[\varepsilon,1/\varepsilon]n^{\alpha_{j_p}}$,
	\begin{align}
	\Probn{X_{u_1,v_1}=\cdots = X_{u_m,v_m}=0} & = \prod_{\mathclap{p\colon  \alpha_{i_p}+\alpha_{j_p}<1}}  (1+O(w_{i_p,j_p}))\big(1-\Theta\big(n^{\alpha_{i_p}+\alpha_{j_p}-1}\big)\big)\nonumber \\
	& \quad \times  \prod_{\mathclap{p\colon \alpha_{i_p}+\alpha_{j_p}=1}}\me^{-{D_{u_{i_p}}D_{v_{j_p}}}/(\mu n)}(1+O(n^{-({\tau-2})/({\tau-1})}))\nonumber\\
	& \quad \times \prod_{\mathclap{p\colon \alpha_{i_p}+\alpha_{j_p}>1}}O\Big(\me^{-n^{\alpha_{i_p}+\alpha_{j_p}-1}/(4\mu)}\Big).
	\end{align}
For ease of notation, we denote
	\begin{equation}
	q(i,j)= \begin{cases}
	(1+O(w_{i,j}))\big(1-\Theta\big(n^{\alpha_{i}+\alpha_{j}-1}\big)\big) & \text{if }\alpha_i+\alpha_j< 1,\\
	\me^{-D_{v_i}D_{v_j}/(\mu n)}(1+O(n^{-\gamma})) &\text{if }\alpha_i+\alpha_j=1,\\
	O\Big(\me^{-n^{\alpha_{i}+\alpha_{j}-1}/(4\mu)}\Big) & \text{if }\alpha_i+\alpha_j>1.
	\end{cases}
	\end{equation}

We write the probability that $H$ is present on a specified subset of vertices $\boldsymbol{v}=(v_1,\ldots,v_k)$ as
	\begin{equation}\label{eq:phsub2}
	\begin{aligned}[b]
	&\Probn{\ECMnD|_{\boldsymbol{v}}\supseteq \Ecal_H}\\
	& = 1-\sum_{l=1}^m\prob_n(X_{v_{i_l},v_{j_l}}=0)+\sum_{{l\neq p}}\prob_n(X_{v_{i_l},v_{j_l}}=X_{v_{i_p},v_{j_p}}=0) \\
	&\quad -\sum_{\mathclap{{l\neq p\neq r}}}\prob_n({X_{v_{i_l},v_{j_l}}=X_{v_{i_p},v_{j_p}}=X_{v_{i_r},v_{i_r}}=0})+\cdots \\
	&\quad + (-1)^{m}\prob_n({X_{v_{i_1},v_{j_1}}=\cdots = X_{v_{i_m},v_{j_m}}}=0)\\
	& = 1-\sum_{l=1}^mq(i_l,j_l)+\sum_{l\neq p}q(i_l,j_l)q(i_p,j_p)\\
	&\quad -\sum_{\mathclap{l\neq p\neq r}}q(i_l,j_l)q(i_p,j_p)q({i_r,j_r})+\cdots  + (-1)^{m} \prod_{l\in[m]}q(i_l,j_l)\\
	& = \prod_{l\in[m]}\big(1-q(i_l,j_l)\big)=  \Theta\bigg(\prod_{\{i,j\}\in \Ecal_H\colon \alpha_{i}+\alpha_j<1}n^{\alpha_i+\alpha_j-1}\bigg).
	\end{aligned}
	\end{equation}
Here we have used that, for $\alpha_i+\alpha_j<1$,
	\begin{equation}
	1-q(i,j)=1-(1-\Theta(n^{\alpha_i+\alpha_j-1}))(1+O(w_{i,j}))=\Theta(n^{\alpha_i+\alpha_j-1}),
	\end{equation}
and that, for $\alpha_i+\alpha_j>1$,
	\begin{equation}
	1-q(i,j)=1-O\big(\me^{-n^{\alpha_{i}+\alpha_{j}-1}/(4\mu)}\big)=1+o(1).
	\end{equation}
Furthermore, for $D_{v_i}\in[ \varepsilon,1/\varepsilon] n^{\alpha_i}$, $D_{v_j}\in[ \varepsilon,1/\varepsilon] n^{\alpha_j}$ and $\alpha_i+\alpha_j=1$, 
	\begin{equation}
	1-q(i,j)=(1+O(n^{-\gamma}))\big(1-\me^{-D_iD_j/(\mu n)}\big)=\Theta(1),
	\end{equation}
so that edges $\{i,j\}$ with $\alpha_i+\alpha_j\geq 1$ do not contribute to the order of magnitude of the last term in~\eqref{eq:phsub2}.
The degrees are an i.i.d.\ sample from a power-law distribution. Therefore,
	\begin{align}
	\Prob{D_1\in[\varepsilon,1/\varepsilon] (\mu n)^{\alpha}}& \cs{= \sum_{x=\varepsilon (\mu n)^\alpha}^{(\mu n)^\alpha/\varepsilon}cx^{-\tau}(1+o(1))}\nonumber \\
	&=\cs{O(1)}\int_{\varepsilon (\mu n)^\alpha}^{1/\varepsilon(\mu n)^\alpha}cx^{-\tau}\dd x = K(\varepsilon)\cs{O}((\mu n)^{\alpha(1-\tau)}),
	\end{align}
for some $K(\varepsilon)$ not depending on $n$. 
The number of vertices with degrees in $[\varepsilon,1/\varepsilon](\mu n)^\alpha$ is 
Binomial$(n,\Prob{D_1\in[\varepsilon,1/\varepsilon] (\mu n)^{\alpha}})$, so that the number of vertices with degrees in $[\varepsilon,1/\varepsilon](\mu n)^\alpha$ is $\Theta_{\sss{\prob}}(n^{(1-\tau)\alpha+1})$ for $\alpha\leq \frac{1}{\tau-1}$. Then, for $M_n^{(\boldsymbol{\alpha})}$ as in~\eqref{eq:Mnalph},
	\begin{equation}\label{eq:numhdeg}
	\# \text{ sets of vertices with degrees in }M_n^{(\boldsymbol{\alpha})}=\Theta_{\sss{\prob}}( n^{k+(1-\tau)\sum_i\alpha_i}).
	\end{equation}
Combining~\eqref{eq:phsub2} and~\eqref{eq:numhdeg} yields
	\begin{equation}\label{eq:Nalph}
	N^\gsub(H,M_n^{(\boldsymbol{\alpha})}(\varepsilon))=\Theta_{\sss{\prob}}\Big( n^{k+(1-\tau)\sum_i\alpha_i} \ \ \prod_{\mathclap{\{i,j\}\in \Ecal_H:\alpha_i+\alpha_j<1}} \ \  n^{\alpha_{i}+\alpha_j-1}\Big).
	\end{equation}
The maximum contribution is obtained for $\boldsymbol{\alpha}$ that maximizes
	\begin{equation}\label{eq:maxalph}
	\begin{aligned}[b]
	&\max (1-\tau)\sum_{i}\alpha_i +\sum_{\{i,j\}\in \Ecal_H\colon \alpha_i+\alpha_j<1}(\alpha_i+\alpha_j-1)\\
	&\qquad\quad\text{s.t. } \alpha_i\in [0,\tfrac{1}{\tau-1}] \ \forall i.
	\end{aligned}
	\end{equation} 
The following lemma shows that this optimization problem attains its maximum for specific values of the exponents $\alpha_i$:

\begin{lemma}[Maximum contribution to subgraphs]\label{lem:maxmotif}
	Let $H$ be a connected graph on $k$ vertices. If the solution to~\eqref{eq:maxalph} is unique, then the optimal solution satisfies $\alpha_i\in\{0,\tfrac{\tau-2}{\tau-1},\tfrac{1}{2},\tfrac{1}{\tau-1}\}$ for all $i$. If it is not unique, then there exist at least 2 optimal solutions with $\alpha_i\in\{0,\tfrac{\tau-2}{\tau-1},\tfrac{1}{2},\tfrac{1}{\tau-1}\}$  for all $i$. In any optimal solution $\alpha_i=0$ if and only if vertex $i$ has degree one in $H$.
\end{lemma}
\begin{proof}
	Defining $\beta_i=\alpha_i-\tfrac{1}{2}$ yields that~\eqref{eq:maxalph} equals
	\begin{equation}\label{eq:maxeqbeta}
	\max \frac{1-\tau}{2}k+ (1-\tau)\sum_{i}\beta_i +\sum_{\{i,j\}\in \Ecal_H\colon \beta_i+\beta_j<0}(\beta_i+\beta_j),
	\end{equation} 
	over all possible values of $\beta_i\in[-\tfrac12,\tfrac{3-\tau}{2(\tau-1)}]$. Then, we have to prove that $\beta_i\in\{-\tfrac 12, \tfrac{\tau-3}{2(\tau-1)},0,\tfrac{3-\tau}{2(\tau-1)}\}$ for all $i$ in the optimal solution.
	Note that~\eqref{eq:maxeqbeta} is a piecewise linear function in $\beta_1,\dots,\beta_k$. Therefore, if~\eqref{eq:maxeqbeta} has a unique maximum, then it must be attained at the boundary for $\beta_i$ or at a border of one of the linear sections. Thus, any unique optimal value of $\beta_i$ satisfies $\beta_i=-\tfrac{1}{2}$, $\beta_i=\tfrac{\tau-3}{2(\tau-1)}$ or $\beta_i+\beta_j=0$ for some $j$.
	We ignore the constant factor of $(1-\tau)\tfrac{k}{2} $ in~\eqref{eq:maxeqbeta}, since it does not influence the optimal $\beta$ values.
	Rewriting~\eqref{eq:maxeqbeta} without the constant factor yields

	\begin{equation}\label{eq:maxbeta}
	\max \sum_i \beta_i\Big(1-\tau +|\{ s\in[k]: (s,i)\in \Ecal_H\text{ and }\beta_s<-\beta_i\}|\Big).
	\end{equation}
	The proof of the lemma then consists of three steps: \\
	\textit{Step 1.} Show that $\beta_i=-\tfrac{1}{2}$ if and only if vertex $i$ has degree 1 in $H$ in any optimal solution.\\
	\textit{Step 2.} Show that any unique solution does not contain $i$ with $\abs{\beta_i}\in(0,\tfrac{3-\tau}{2(\tau-1)})$.\\
	\textit{Step 3.} Show that any optimal solution that is not unique can be transformed into two different optimal solutions with $\beta_i\in\{-\tfrac 12, \tfrac{\tau-3}{2(\tau-1)},0,\tfrac{3-\tau}{2(\tau-1)}\}$ for all $i$.\\
	
	\textit{Step 1.}
	Let $i$ be a vertex of degree 1 in $H$, and $j$ be the neighbor of $i$. Let $N_j$ denote the number of edges in $H$ from $j$ to other vertices $v$ not equal to $i$ with $\beta_v<-\beta_j$. The contribution from vertices $i$ and $j$ to~\eqref{eq:maxbeta} is
	\begin{equation}
	\beta_j(1-\tau+N_j)+\beta_i(1-\tau+\ind{\beta_i>-\beta_j})+\beta_j\ind{\beta_i<-\beta_j}.
	\end{equation} 
	For any value of $\beta_j\in[-\tfrac 12,\tfrac{3-\tau}{2(\tau-1)}]$, this contribution is maximized when choosing $\beta_i=-\tfrac{1}{2}$. 
	Thus, $\beta_i=- \tfrac{1}{2}$ in the optimal solution if the degree of vertex $i$ is one.
	
Let $i$ be a vertex in $V_H$, and recall that $d_i$ denotes the degree of $i$ in $H$. Let $i$ be such that $d_i\geq 2$ in $H$, and suppose that $\beta_i<\tfrac{\tau-3}{2(\tau-1)}$. Because the maximal value of $\beta_j$ for $j\neq i$ is $\tfrac{3-\tau}{2(\tau-1)}$, the contribution to the $i$th term of~\eqref{eq:maxbeta} is
	\begin{equation}
	-\tfrac{1}{2}(1-\tau+d_i)<0,
	\end{equation}
		irrespective of the values of the $\beta_j$, $j\neq i$.
	Increasing $\beta_i$ to $\tfrac{\tau-3}{2(\tau-1)}$ then gives a higher contribution. 
	\cs{Thus, $\beta_i\geq \tfrac{\tau-3}{2(\tau-1)}$ when $d_i\geq 2$.}

	\textit{Step 2.}
	Now we show that when the solution to~\eqref{eq:maxbeta} is unique, it is never optimal to have $\abs{\beta}\in(0,\tfrac{3-\tau}{2(\tau-1)})$. 
	Let 
	\begin{equation}\label{eq:tildebeta}
	\tilde{\beta}=\min_{i:\abs{\beta_i}>0}\abs{\beta_i}.
	\end{equation}
	Let  $N_{\tilde{\beta}^-}$ denote the number of vertices with their $\beta$ value equal to $-\tilde{\beta}$, and $N_{\tilde{\beta}^+}$ the number of vertices with value $\tilde{\beta}$, where $N_{\tilde{\beta}^+}+N_{\tilde{\beta}^-}\geq 1$. Furthermore, let $E_{\tilde{\beta}^-}$ denote the number of edges from vertices with value $-\tilde{\beta}$ to other vertices $j$ such that $\beta_j<\tilde{\beta}$, and $E_{\tilde{\beta}^+}$ the number of edges from vertices with value $\tilde{\beta}$ to other vertices $j$ such that $\beta_j<-\tilde{\beta}$. Then, the contribution from these vertices to~\eqref{eq:maxbeta} is
	\begin{equation}\label{eq:Nbeta}
	\tilde{\beta}\big((1-\tau)(N_{\tilde{\beta}^+}-N_{\tilde{\beta}^-})+E_{\tilde{\beta}^+}-E_{\tilde{\beta}^-}\big).
	\end{equation}
	Because we assume ${\beta}$ to be optimal, and the optimum to be unique, the value inside the brackets cannot equal zero. The contribution is linear in $\tilde{\beta}$ and it is the optimal contribution, and therefore $\tilde{\beta}\in\{0,\tfrac{3-\tau}{2(\tau-1)}\}$.
	This shows that $\beta_i\in\{\tfrac{\tau-3}{2(\tau-1)},0,\tfrac{3-\tau}{2(\tau-1)}\}$ for all $i$ such that $d_i\geq 2$.
	
	\textit{Step 3.}
	\cs{
	Suppose that the solution to~\eqref{eq:maxbeta} is not unique. Suppose that $\beta_*$ appears in one of the optimizers of~\eqref{eq:maxbeta}. In the same notation as in~\eqref{eq:Nbeta}, the contribution from vertices with $\beta$-values $\beta_*$ and $-\beta_*$ equals
	\begin{equation}
	{\beta}_*\Big[(1-\tau)\big(N_{{\beta}_*^+}-N_{{\beta}_*^-}\big)+E_{{\beta}_*^+}-E_{{\beta}_*^-}\Big]. 
	\end{equation}
	Since this contribution is linear in $\beta_*$, the contribution of these vertices can only be non-unique if the term within the square brackets equals zero. 
	Thus, for the solution to~\eqref{eq:maxbeta} to be non-unique, there must exist $\hat{\beta}_1,\ldots,\hat{\beta}_s>0$ for some $s\geq 1$ such that  
	\begin{equation}
	\hat{\beta}_j\Big((1-\tau)\big(N_{\hat{\beta}_j^+}-N_{\hat{\beta}_j^-}\big)+E_{\hat{\beta}_j^+}-E_{\hat{\beta}_j^-}\Big)=0 \quad \forall j\in[s].
	\end{equation}
	Setting all $\hat{\beta}_j=0$ and setting all $\hat{\beta}_j=\tfrac{3-\tau}{2(\tau-1)}$ are both optimal solutions. Thus, if the solution to~\eqref{eq:maxbeta} is not unique, at least 2 solutions exist with $\beta_i\in\{\tfrac{\tau-3}{2(\tau-1)},0,\tfrac{3-\tau}{2(\tau-1)}\}$ for all $i\in V_H$. }
\end{proof}

\begin{proof}[Proof of Theorem~\ref{thm:motifs}(ii) for subgraphs]
	Let $\boldsymbol{\alpha}^\gsub$ be the unique optimizer of~\eqref{eq:maxalph}. 
	By Lemma~\ref{lem:maxmotif}, the maximal value of~\eqref{eq:maxalph} is attained by partitioning $V_H\setminus V_1$ into the sets $S_1,S_2,S_3$ such that vertices in $S_1$ have $\alpha_i^\gsub=\tfrac{\tau-2}{\tau-1}$, vertices in $S_2$ have $\alpha_i^\gsub =\tfrac{1}{\tau-1}$, vertices in $S_3$ have $\alpha_i^\gsub=\tfrac{1}{2}$ and vertices in $V_1$ have $\alpha_i^\gsub =0$. Then, the edges with $\alpha_i^\gsub+\alpha_j^\gsub <1$ are edges inside $S_1$, edges between $S_1$ and $S_3$ and edges from degree 1 vertices. Recall that the number of edges inside $S_1$ is denoted by $E_{S_1}$, the number of edges between $S_1$ and $S_3$ by $E_{S_1,S_3}$ and the number of edges between $V_1$ and $S_i$ by $E_{S_1,V_1}$. Then we can rewrite~\eqref{eq:maxalph} as
	\begin{equation}
	\label{eq:maxtemp}
	\begin{aligned}[b]
	\max_{\mathcal{P}} \ &\Big[(1-\tau) \left(\frac{\tau-2}{\tau-1}\abs{S_1}+\frac{1}{\tau-1}\abs{S_2}+\tfrac 12 \abs{S_3}\right)+\frac{\tau-3}{\tau-1}E_{S_1}\\
	&\qquad +\frac{\tau-3}{2(\tau-1)}E_{S_1,S_3}
	-\frac{E_{S_1,V_1}}{\tau-1}-\frac{\tau-2}{\tau-1}E_{S_2,V_1}-\frac 12 E_{S_3,V_1}\Big],
	\end{aligned}
	\end{equation}
over all partitions $\mathcal{P}=(S_1,S_2,S_3)$ of $V_H\setminus V_1$. Using that $|S_3|=k-\abs{S_1}-\abs{S_2}-k_1$, ${E_{S_3,V_1}=k_1-E_{S_1,V_1}-E_{S_2,V_1}}$, where $k_1=\abs{V_1}$ and extracting a factor $(3-\tau)/2$ shows that this is equivalent to
	\begin{equation}
	\label{eq:maxtemp2}
	\begin{aligned}[b]
	&\frac{1-\tau}{2}k+	\max_{\mathcal{P}} \ \frac{(3-\tau)}{2}\Big( \abs{S_1}-\abs{S_2}+\frac{\tau-2}{3-\tau} k_1-\frac{2E_{S_1}+E_{S_1,S_3}}{\tau-1}  \\
	&\qquad\quad -\frac{E_{S_1,V_1}-E_{S_2,V_1}}{\tau-1}\Big).
	\end{aligned}
	\end{equation}
Since $k$ and $k_1$ are fixed and $3-\tau>0$, we need to maximize
	\begin{equation}
	\label{eq:maxeq}
	B^\gsub(H)=\max_{\mathcal{P}}\left[\abs{S_1}-\abs{S_2}-\frac{2E_{S_1}+E_{S_1,S_3}+E_{S_1,V_1}-E_{S_2,V_1}}{\tau-1}\right],
	\end{equation}
which equals \eqref{eq:maxeqsub}.
	By~\eqref{eq:Nalph}, the contribution of the maximum is then given by 
	\begin{equation}\label{eq:maxcontrscaling}
	n^{\frac{3-\tau}{2}(k+B^\gsub(H))+\frac{\tau-2}{2}k_1}=n^{\frac{3-\tau}{2}(k_{2+}+B^\gsub(H))+{k_1/2}} ,
	\end{equation}
which proves Theorem~\ref{thm:motifs}(ii) for subgraphs. 
\end{proof}

\section{Proof of Theorem~\ref{thm:sqrtsub}}\label{sec:proof2}
 Define the special case of $M_n^{\sss{(\boldsymbol{\alpha})}}(\varepsilon)$ of~\eqref{eq:Mnalph} where $\alpha_i=\tfrac{1}{2}$ for all $i\in V_H=[k]$ as
	\begin{equation}
		W_n^k(\varepsilon)=\{(v_1,\ldots,v_k)\colon D_{{v_s}}\in[\varepsilon,1/\varepsilon]\sqrt{\mu n} \quad \forall s \in[k]\},
	\end{equation}
\cs{and let $\bar{W}_n^k(\varepsilon)$ denote the complement of $W_n^k(\varepsilon)$.}
	Denote the number of subgraphs $H$ with all vertices in $W_n^k(\varepsilon)$ by $N^\gsub(H,W_n^k(\varepsilon))$. 
\begin{lemma}[Major contribution to subgraphs]\label{lem:convNH}
	Let $H$ be a connected graph on $k{\geq 3}$ vertices such that~\eqref{eq:maxeqsub} is uniquely optimized at $S_3=[k]$, so that $B^\gsub(H)=0$. Then,
	\begin{enumerate}[(i)]
	\item \label{lem:convNH1} the number of subgraphs with vertices in $W_n^k(\varepsilon)$ satisfies\cs{
	\begin{align}
		\frac{N^\gsub(H,W_n^k(\varepsilon))}{n^{\frac{k}{2}(3-\tau)}} 
		= & (1+\op(1))c^k\mu^{-\frac{k}{2}(\tau-1)} \int_{\varepsilon}^{1/\varepsilon}\!\!\cdots \int_{\varepsilon}^{1/\varepsilon}(x_1\cdots x_k)^{-\tau}\nonumber\\
		& \times \prod_{\mathclap{\{i,j\}\in \Ecal_H}}(1-\me^{-x_ix_j})\dd x_1\cdots \dd x_k +f_n(\varepsilon),
	\end{align}
for some function $f_n(\varepsilon)$ such that, for any $\delta>0$, 
	\begin{equation}\lim_{\varepsilon\searrow 0}\limsup_{n\to\infty}\Prob{f_n(\varepsilon)>\delta\mid J_n}=0;
	\end{equation}
	}
	\item \label{lem:Afinite}
$A^\gsub(H)$ defined in~\eqref{eq:Asub} satisfies $A^\gsub(H)<\infty$.
	\end{enumerate}
	\end{lemma}

The proof of Lemma~\ref{lem:convNH} can be found in Section~\ref{sec:prooflem1}. 
We now prove Theorem~\ref{thm:sqrtsub} using this lemma. 

\begin{proof}[Proof of Theorem~\ref{thm:sqrtsub}]
	We start by studying the expected number of subgraphs with vertices outside $W_n^k(\varepsilon)$. First, we investigate the expected number of subgraphs in the case where vertex 1 of the subgraph has degree smaller than $\varepsilon\sqrt{\mu n}$. 
	\cs{Similarly to~\eqref{eq:phsub2}, we can use Lemma~\ref{lem:condprob} to show that the probability that $H$ is present on a specified subset of vertices $\boldsymbol{v}=(v_1,\ldots,v_k)$ can be written as
		\begin{align}
		\Probn{\ECMnD|_{\boldsymbol{v}}\supseteq \Ecal_H} & =\Theta \Big( \prod_{\{i,j\}\in \Ecal_H\colon D_{v_i}D_{v_j}<L_n}(1-\me^{-D_{v_i}D_{v_j}/L_n})\Big)
		\nonumber\\
		& =\Theta\Bigg( \prod_{\{i,j\}\in \Ecal_H}(1-\me^{-D_{v_i}D_{v_j}/L_n})\Bigg).
		\end{align}
	
}
	
Furthermore, by~\eqref{D-tail}, there exists $C_0$ such that $\Prob{D=k}\leq C_0k^{-\tau}$ for all $k$. Let $I^\gsub(H,\boldsymbol{v})=\ind{\ECMnD|_{\boldsymbol{v}}\supseteq \Ecal_H},$ so that $N^\gsub(H)=\sum_{\boldsymbol{v}} I^\gsub(H,\boldsymbol{v})$. Then, the expected number of subgraphs in the case where vertex 1 of the subgraph has degree smaller than $\varepsilon\sqrt{\mu n}$ is bounded by

\begin{equation}
\begin{aligned}
 &\sum_{\boldsymbol{v}}\Exp{I^\gsub(H,\boldsymbol{v})\ind{D_{v_1}<\varepsilon\sqrt{\mu n}}\mid J_n}\\
	&\leq \Theta(1)n^k\int_{1}^{\varepsilon\sqrt{\mu n}}\int_{1}^{\infty}\cdots\int_{1}^{\infty}(x_1\cdots x_k)^{-\tau} \ \prod_{\mathclap{\{i,j\}\in \Ecal_H}} \ (1-\me^{-x_ix_j/(\mu n)})\dd x_1\cdots \dd x_k\\
	&=\Theta(1)n^k(\mu n)^{\frac{k}{2}(1-\tau)} \int_{0}^{\varepsilon}\int_{0}^{\infty}\cdots\int_{0}^{\infty}(t_1\cdots t_k)^{-\tau} \ \prod_{\mathclap{\{i,j\}\in \Ecal_H}} \ (1-\me^{-t_it_j})\dd t_1\cdots \dd t_k\\
	& = \bigO{n^{\frac{k}{2}(3-\tau)}}h_1(\varepsilon),
	\end{aligned}
	\end{equation}
where $h_1(\varepsilon)$ is a function of $\varepsilon$. By Lemma~\ref{lem:convNH}\ref{lem:Afinite}, $h_1(\varepsilon)\to 0$ as $\varepsilon\searrow 0$. We can bound the situation where one of the other vertices has degree smaller than $\varepsilon\sqrt{n}$, or where one of the vertices has degree larger than $\sqrt{n}/\varepsilon$, similarly. This yields
	\begin{equation}
	\Exp{N^\gsub(H,\bar{W}_n^k(\varepsilon))\mid J_n} = \bigO{n^{\frac{k}{2}(3-\tau)}}h(\varepsilon),
	\end{equation}
for some function $h(\varepsilon)$ not depending on $n$ such that $h(\varepsilon)\to 0$ when $\varepsilon\searrow 0$. Then, by the Markov inequality, conditionally on $J_n$,
	\begin{equation}
	\begin{aligned}[b]
	N^\gsub(H,\bar{W}_n^k(\varepsilon))=h(\varepsilon)\bigOp{n^{\frac{k}{2}(3-\tau)}}.
	\end{aligned}
	\end{equation}
Therefore, for any $\delta>0$,
	\begin{equation}
	\limsup_{\varepsilon\to 0}\limsup_{n\to\infty} \Prob{\frac{N^\gsub(H,\bar{W}_n^k(\varepsilon))}{n^{k(3-\tau)/2}}>\delta \mid J_n}=0.
	\end{equation}
Combining this with the fact that $\Prob{J_n}\to 1$ and Lemma~\ref{lem:convNH}\ref{lem:convNH1} gives
	\begin{align}
	\frac{N^\gsub(H)}{n^{\frac{k}{2}(3-\tau)}}\plim & c^k\mu^{-\frac{k}{2}(\tau-1)}\! \int_{0}^{\infty}\! \cdots\!  \int_{0}^{\infty}(x_1,\cdots x_k)^{-\tau}\prod_{\mathclap{\{i,j\}\in \Ecal_{H}}} \ (1-\me^{-x_ix_j})	\dd x_1\cdots \dd x_k.
	\end{align}
\end{proof}

\section{Major contribution to subgraphs: proof of Lemma~\ref{lem:convNH}}\label{sec:prooflem1}
We first prove Lemma~\ref{lem:convNH}(i). We compute the expected value of the number of subgraphs in the quenched sense in Lemmas~\ref{lem:condexsub} and~\ref{lem:convsub}. Then, we study the variance of the number of subgraphs in the quenched sense in Lemma~\ref{lem:varsub}. 
	Together, these lemmas prove Lemma~\ref{lem:convNH}(i). 


\subsection{Conditional expectation}
In this section, we study the expected number of subgraphs in $\ECMnD$. 
Let $H$ be a subgraph on $k$ vertices, labeled as ${[k]}$, and $m$ edges, denoted by $e_1={\{i_1,j_1\},\ldots,e_m=\{i_m,j_m\}}$.

\begin{lemma}[Conditional expectation of subgraphs]\label{lem:condexsub}
	Let $H$ be a subgraph such that~\eqref{eq:maxeqsub} has a unique maximum, attained at $S_3^\gsub=[k]$ so that $B^\gsub(H)=0$. Then, on the event $J_n$ defined in~\eqref{eq:Jn},
	\begin{equation}\label{eq:condex}
	\Expn{N^\gsub(H,W_n^k(\varepsilon))}=\sum_{(v_1,\ldots,v_k)\in W_n^k(\varepsilon)}\prod_{\{i,j\}\in \Ecal_H}(1-\me^{-D_{v_i}D_{v_j}/L_n})(1+o(1)).
	\end{equation}
\end{lemma}

\begin{proof}
Let $\boldsymbol{v}=(v_1,\ldots,v_k)$ and $\ECMnD|_{\boldsymbol{v}}$ again be the induced subgraph of $\ECMnD$ on $\boldsymbol{v}$. We first derive a more detailed expression for the probability that a subgraph is present on $\boldsymbol{v}$ than~\eqref{eq:phsub2} which holds when $\boldsymbol{v}\in W_n^k(\varepsilon)$. Because $\boldsymbol{v}\in W_n^k(\varepsilon)$, we may use~\eqref{eq:psubdetail} for all edge probabilities to obtain
	\begin{equation}
		\prob_n(X_{v_{i_1},v_{j_1}}=\cdots = X_{v_{i_m},v_{j_m}}=0) = \prod_{l=1}^m\prob_n(X_{v_{i_l},v_{j_l}}=0)(1+O(n^{(\tau-2)/(\tau-1)})).
	\end{equation}
When $D_{v_i},D_{v_j}\in[\varepsilon\sqrt{n},\sqrt{n}/\varepsilon]$, $\Probn{X_{v_i,v_j}=0}=\Theta(1)$ and $\Probn{X_{v_i,v_j}=1}=\Theta(1)$.
Therefore, similarly to~\eqref{eq:phsub2}, for $\boldsymbol{v}\in W_n^k(\varepsilon)$,
	\begin{equation}\label{eq:psubsqrt}
	\begin{aligned}[b]
	&\Probn{\ECMnD|_{\boldsymbol{v}}\supseteq \Ecal_H}\\
	& = 1-\sum_{l=1}^m\prob_n(X_{v_{i_l},v_{j_l}}=0)+\sum_{l\neq p}\prob_n(X_{v_{i_l},v_{j_l}}=X_{v_{i_p},v_{j_p}}=0) \\
	&\quad -\sum_{\mathclap{l\neq p\neq r}}\prob_n(X_{v_{i_l},v_{j_l}}=X_{v_{j_p},v_{i_p}}=X_{v_{i_r},v_{j_r}}=0)+\cdots \\
	&\quad + (-1)^{m}\prob_n(X_{v_{i_1},v_{j_1}}=\cdots = X_{v_{i_m},v_{j_m}}=0)\\
	& =(1+o(1))\prod_{l=1}^{m}\left(1-\Probn{X_{v_{i_l},v_{j_l}}=0}\right).
	\end{aligned}
	\end{equation}
Thus, the conditonal expected value satisfies
	\begin{align}
	\Expn{N^\gsub(H,W_n^k(\varepsilon))}& = \sum_{\boldsymbol{v}\in W_n^k(\varepsilon)}\Probn{\ECMnD|_{\boldsymbol{v}}\supseteq \Ecal_H}\nonumber\\
	&=(1+o(1))\sum_{\boldsymbol{v}\in W_n^k(\varepsilon)}\prod_{l=1}^{m}\left(1-\Probn{X_{v_{i_l},v_{i_l}}=0}\right),
	\end{align}
Because $D_{v_i}D_{v_j}=O(n)$ and $L_n=\mu n(1+o(1))$ under $J_n$, by~\eqref{eq:pijsmall},
	\begin{equation}\label{eq:pij1}
	\Probn{X_{v_i,v_j}=1}=1-\me^{-D_{v_i}D_{v_j}/L_n}+\bigO{\frac{D_{v_i}^2D_{v_j}}{L_n^2}}= (1+o(1))\left(1-\me^{-D_{v_i}D_{v_j}/L_n}\right).
	\end{equation}
This results in
	\begin{equation}
	\Expn{N^\gsub(H,W_n^k(\varepsilon))}=(1+o(1))\sum_{\boldsymbol{v}\in W_n^k(\varepsilon)}\prod_{\{i,j\}\in \Ecal_H}(1-\me^{-D_{v_i}D_{v_j}/L_n}).
	\end{equation}
\end{proof}

\subsection{Convergence of conditional expectation}
We now study the asymptotic behavior of the expected number of subgraphs using Lemma~\ref{lem:condexsub}:

\begin{lemma}[Convergence of conditional expectation of $\sqrt{n}$ subgraphs]\label{lem:convsub}
	Let $H$ be a subgraph such that~\eqref{eq:maxeqsub} has a unique maximizer, and the maximum is attained at 0. Then,\cs{
	\begin{align}
	\frac{\Expn{N^\gsub(H,W_n^k(\varepsilon))}}{n^{\frac{k}{2}(3-\tau)}}=&  (1+\op(1))c^k\mu^{-\frac{k}{2}(\tau-1)}\int_{\varepsilon}^{1/\varepsilon}\!\!\cdots \int_{\varepsilon}^{1/\varepsilon}(x_1\cdots x_k)^{-\tau}\nonumber\\
	&  \times \prod_{\mathclap{\{i,j\}\in \Ecal_{H}}}(1-\me^{-x_ix_j})\dd x_1\cdots \dd x_k +f_n(\varepsilon),
	\end{align}
	for some function $f_n(\varepsilon)$ such that, for any $\delta>0$, 
	\begin{equation}\lim_{\varepsilon\searrow 0}\limsup_{n\to\infty}\Prob{f_n(\varepsilon)>\delta\mid J_n}=0.\end{equation}}
\end{lemma}

\begin{proof} Let $\abs{\Ecal_H}=m$ and denote the edges of $H$ by $\{i_1,j_1\},\ldots,\{i_m,j_m\}$. Define
	\begin{equation}\label{eq:g}
	g(t_1,\ldots,t_k):=\prod_{\{i,j\}\in \Ecal_H}(1-\me^{-t_ut_v}).
	\end{equation}
Using the Taylor expansion of $1-\me^{-xy}$ on $[\varepsilon,1/\varepsilon]^2$ results in
	\begin{equation}
	1-\me^{-xy}= \sum_{i=1}^s\frac{(xy)^i}{i!}(-1)^i+\bigO{\frac{\varepsilon^{-s}}{(s+1)!}}.
	\end{equation}
Since $g$ is bounded on $F=[\varepsilon,1/\varepsilon]^{k}$, we can find $s_1,\ldots,s_{m}$ and $\eta(t_1,\dots,t_k)$ such that $|\eta(t_1,\dots,t_k)|\leq \varepsilon^{k(\tau-1)+1}$ such that 
	\begin{align}\label{eq:taylorg}
	g(t_1,\ldots,t_k)&= \sum_{p_1=1}^{s_1}\cdots \sum_{p_{m}=1}^{s_{m}}\bigg((-1)^{p_1}\frac{t_{u_1}^{p_1}t_{v_1}^{p_1}}{p_1!}
	\cdots (-1)^{p_{m}}\frac{t_{u_{m}}^{p_{m}}t_{v_{m}}^{p_{m}}}{p_{m}!}\bigg)+\eta(t_1,\dots,t_k) \nonumber\\
	&=\sum_{p_1=1}^{s_1}\cdots \sum_{p_{m}=1}^{s_{m}}\left(\frac{(-1)^{p_1+\cdots +p_m}}{p_1!\cdots p_m!}t_{1}^{\gamma_1}t_{2}^{\gamma_2}\cdots t_{k}^{\gamma_k}\right)+\eta(t_1,\dots,t_k),
	\end{align}
where
	\begin{equation}
	\gamma_j:=\gamma_j(p_1,\ldots,p_m)=\sum_{l}p_l\ind{i_l=j \text{ or }j_l=j}.
	\end{equation}
Let $\Mn$ denote the random measure
	\begin{equation}
	\Mn([a,b])=(\mu n)^{\frac 12(\tau-1)}n^{-1}\sum_{\cj{v\in[n]}\ind{D_v\in \sqrt{\mu n}[a,b]}}.
	\end{equation}
The number of vertices with degrees in a certain interval $[a,b]$ is binomially distributed. By~\eqref{D-tail}, we thus get $(\mu n)^{\frac 12 (\tau-1)}\Prob{D_1\in \sqrt{n}[a,b]}\plim\lambda([a,b])$, where
	\begin{equation}
	\begin{aligned}[b]
	\lambda([a,b]):=c \int_{a}^{b}x^{-\tau}\dd x.
	\end{aligned}
	\end{equation}
Hence, by the weak law of large numbers, as $n\to\infty$,
	\begin{equation}\label{eq:Mn}
	\begin{aligned}[b]
	\Mn([a,b])&
	\plim\lambda([a,b]).
	\end{aligned}
	\end{equation}
Let $\Nn$ denote the product measure $\Mn\times \Mn\times\cdots \times \Mn$ ($k$ times). Then~\eqref{eq:taylorg} together with Lemma~\ref{lem:condexsub} yields
	\begin{equation}\label{eq:convexg}
	\begin{aligned}[b]
	& \frac{\Expn{N^\gsub(H,W_n^k(\varepsilon))}}{n^{\frac{k}{2}(3-\tau)}\mu^{\frac{k}{2}(1-\tau)}} 
	=\int_F g(t_1,\ldots,t_k)\dd \Nn(t_1,\ldots,t_k)\\
	& =\int_F \sum_{p_1=1}^{s_1}\cdots \sum_{p_{m}=1}^{s_{m}}\Bigg(\left(\frac{(-1)^{p_1+\cdots +p_m}}{p_1!\cdots p_m!}t_{1}^{\gamma_1}t_{2}^{\gamma_2}\cdots t_{k}^{\gamma_k}\right) 
	+\eta(t_1,\dots,t_k) \Bigg)\dd \Nn(t_1,\ldots,t_k)\\
	& = \sum_{p_1=1}^{s_1}\cdots \sum_{p_{m}=1}^{s_{m}}\frac{(-1)^{p_1+\cdots +p_m}}{p_1!\cdots p_m!}\int_{\varepsilon}^{1/\varepsilon}t_{1}^{\gamma_1}\dd \Mn(t_1) 
	\cdots \int_{\varepsilon}^{1/\varepsilon}t_{k}^{\gamma_k}\dd \Mn(t_k)+f_n(\varepsilon).
	\end{aligned}
	\end{equation}
Here 
	\begin{align}
	f_n(\varepsilon) & = \int_F \sum_{p_1=1}^{s_1}\cdots \sum_{p_{m}=1}^{s_{m}}\eta(t_1,\dots,t_k) \dd \Nn(t_1,\ldots,t_k)\nonumber\\
	& \leq \int_F \sum_{p_1=1}^{s_1}\cdots \sum_{p_{m}=1}^{s_{m}}\varepsilon^{k(\tau-1)+1} \dd \Nn(t_1,\ldots,t_k)\nonumber\\
	& = \varepsilon^{k(\tau-1)+1}s_1\cdots s_m \Mn([\varepsilon,1/\varepsilon])^k \nonumber\\
	& = \varepsilon^{k(\tau-1)+1} \bigOp{\lambda([\varepsilon,1/\varepsilon])^k }\nonumber\\
	&=  \varepsilon^{k(\tau-1)+1}(\varepsilon^{1-\tau}-\varepsilon^{\tau-1})^k\bigOp{1}=\bigOp{\varepsilon},
	\end{align}
which shows that, for any $\delta>0$, 
	\begin{equation}
	\lim_{\varepsilon\searrow 0}\limsup_{n\to\infty}\Prob{f_n(\varepsilon)>\delta\mid J_n}=0.
	\end{equation}
	
As in~\cite[Eq. (55)]{stegehuis2017b}, for any $\gamma$,
	\begin{equation}
	\int_{\varepsilon}^{1/\varepsilon}x^{\gamma}\dd \Mn(x)\plim  	\int_{\varepsilon}^{1/\varepsilon}x^{\gamma}\dd \lambda(x).
	\end{equation}
Combining this with~\eqref{eq:convexg} results in
	\begin{equation}
	\begin{aligned}[b]
	\allowdisplaybreaks
	&\frac{\Expn{N^\gsub(H,W_n^k(\varepsilon))}}{n^{\frac{k}{2}(3-\tau)}\mu^{\frac{k}{2}(1-\tau)}}\\
	& =
	(1+\op(1))\sum_{p_1=1}^{s_1}\cdots \sum_{p_{m}=1}^{s_{m}}\frac{(-1)^{p_1+\cdots +p_m}}{p_1!\cdots p_m!}
	\int_{\varepsilon}^{1/\varepsilon}t_{1}^{\alpha_1}\dd \lambda(t_1) \cdots \int_{\varepsilon}^{1/\varepsilon}t_{k}^{\alpha_k}\dd \lambda(t_k)+f_n(\varepsilon)\\
	& =(1+\op(1))\int_F \sum_{p_1=1}^{s_1}\cdots \sum_{p_{m}=1}^{s_{m}}\frac{(-1)^{p_1+\cdots +p_m}}{p_1!\cdots p_m!}t_{1}^{\alpha_1} \cdots t_{k}^{\alpha_k}
	\dd \lambda(t_1)\cdots \dd \lambda(t_k)+f_n(\varepsilon)\\
	& =(1+\op(1))\int_F g(t_1,\ldots,t_k)\dd \lambda(t_1)\cdots \dd \lambda(t_k)+f_n(\varepsilon).
	\end{aligned}
	\end{equation}
Then, by~\eqref{eq:Mn},
	\begin{align}
	\frac{\Expn{N^\gsub(H,W_n^k(\varepsilon))}}{n^{\frac{k}{2}(3-\tau)}}& =(1+\op(1)) c^k \mu^{-\frac{k}{2}(\tau-1)}\int_{\varepsilon}^{1/\varepsilon}
	\cdots \int_{\varepsilon}^{1/\varepsilon}(t_1\cdots t_k)^{-\tau }  \nonumber\\
	& \quad \times g(t_1,\ldots,t_k)\dd t_1\cdots \dd t_k +f_n(\varepsilon), 
	\end{align}
which proves the claim.
\end{proof}

\subsection{Conditional variance}
We now study the conditional variance of the number of subgraphs in the quenched setting for the degrees. The following lemma shows that the conditional variance of the number of subgraphs is small compared to its expectation:
\begin{lemma}[Conditional variance for subgraphs]\label{lem:varsub}
	Let $H$ be a subgraph such that~\eqref{eq:maxeqsub} has a unique maximum attained at 0. Then, on the event $J_n$ defined in~\eqref{eq:Jn},
	\begin{equation}
	\frac{\Varn{N^\gsub(H,W_n^k(\varepsilon))}}{\Expn{N^\gsub(H,W_n^k(\varepsilon))}^2}\plim 0.
	\end{equation}
\end{lemma}
\begin{proof}
	By Lemma~\ref{lem:convsub}, 
	\begin{equation}
	\Expn{N^\gsub(H,W_n^k(\varepsilon))}^2=\Theta_{\sss{\prob}}(n^{(3-\tau)k}),
	\end{equation}
	Thus, we need to prove that the variance is small compared to $n^{(3-\tau)k}$. Denote $\boldsymbol{v}=(v_1,\ldots,v_k)$ and ${\boldsymbol{u}}=(u_1,\ldots,u_k)$ and, for ease of notation, we denote $G=\ECMnD$. 
	We write the variance as
	\begin{align}\label{eq:varmotif}
	\Varn{N^\gsub(H,W_n^k(\varepsilon))}&= \sum_{\boldsymbol{v}\in W_n^k(\varepsilon)}\sum_{\boldsymbol{u}\in W_n^k(\varepsilon)}
	\Big(\Probn{G|_{\boldsymbol{v}}\supseteq \Ecal_H,G|_{\boldsymbol{u}}\supseteq \Ecal_H}\nonumber\\
	& \quad -\Probn{G|_{\boldsymbol{v}}\supseteq \Ecal_H}\Probn{G|_{\boldsymbol{u}}\supseteq \Ecal_H}\Big).
	\end{align}
	This splits into various cases, depending on the overlap of $\boldsymbol{v}$ and $\boldsymbol{u}$. When $\boldsymbol{v}$ and $\boldsymbol{u}$ do not overlap, similarly to~\eqref{eq:psubsqrt},
	\begin{equation}\label{eq:varbound}
	\begin{aligned}[b]
	&\sum_{\boldsymbol{v}\in W_n^k(\varepsilon)}\sum_{\boldsymbol{u}\in W_n^k(\varepsilon)}\big(\Probn{G|_{\boldsymbol{v}}\supseteq \Ecal_H,G|_{\boldsymbol{u}}\supseteq \Ecal_H} -\Probn{G|_{\boldsymbol{v}}\supseteq \Ecal_H}\Probn{G|_{\boldsymbol{u}}\supseteq \Ecal_H}\big)\\
	& = \sum_{\boldsymbol{v}\in W_n^k(\varepsilon)} \sum_{\boldsymbol{u}\in W_n^k(\varepsilon)}\!\!\Big((1+o(1))\prod_{l=1}^{m}\big(1-\prob_n\big(X_{v_{i_l},v_{j_l}}=0\big)\big) \big(1-\prob_n\big(X_{u_{i_l},u_{j_l}}=0\big)\big)\\
	& \quad -(1+o(1))\prod_{l=1}^{m}\big(1-\prob_n\big(X_{v_{i_l},v_{j_l}}=0\big)\big) \big(1-\prob_n\big(X_{u_{i_l},u_{j_l}}=0\big)\big)\Big)\\
	& = \Expn{N^\gsub(H,W_n^k(\varepsilon))}^2o(1).
	\end{aligned}
	\end{equation}
The other contributions are when $\boldsymbol{v}$ and $\boldsymbol{u}$ overlap. In this situation, we use the bound $\Probn{X_{u,v}=1}\leq1$.
When $\boldsymbol{v}$ and $\boldsymbol{u}$ overlap on $s\geq 1$ vertices, we bound the contribution to~\eqref{eq:varmotif} as
	\begin{equation}\label{eq:varsqrt}
	\begin{aligned}[b]
	\sum_{\mathclap{\boldsymbol{v},\boldsymbol{u}\in W_n^k(\varepsilon)\colon \abs{\boldsymbol{v}\cup\boldsymbol{u}}=2k-s}} \ \ \Probn{G|_{\boldsymbol{v}}\supseteq \Ecal_H,G|_{\boldsymbol{u}}\supseteq \Ecal_H}& \leq\abs{ \{i\colon D_i\in \sqrt{\mu n}[\varepsilon,1/\varepsilon]\}}^{2k-s}\\
	& =\bigOp{n^\frac{(3-\tau)(2k-s)}{2}},
	\end{aligned}
	\end{equation}
	which is $o(n^{(3-\tau)k})$, as required. 
\end{proof}

\begin{proof}[Proof of Lemma~\ref{lem:convNH}]
	We start by proving part (i). By Lemma~\ref{lem:varsub} and Chebyshev's inequality, conditionally on the degrees
	\begin{equation}
		N^\gsub(H,W_n^k(\varepsilon))=\Expn{N^\gsub(H,W_n^k(\varepsilon)}(1+\op(1)).
	\end{equation}
	Combining this with Lemma~\ref{lem:convsub} proves Lemma~\ref{lem:convNH}(i). 
	Lemma~\ref{lem:convNH}(ii) is a direct consequence of Lemma~\ref{lem:S3int} in the next section, when we take $|S_3^*|=k$. Here, we remark that the proof of Lemma~\ref{lem:S3int} is entirely self-contained, so no circularity in our argument arises.
\end{proof}

\section{Major contribution to general subgraphs: proof of Theorem~\ref{thm:motifs}(i)}
\label{sec:proofsec2}
In this section we prove Theorem~\ref{thm:motifs}(i) for subgraphs. We start by giving an overview of the proof. We restrict this overview to the expected value of $N^\gsub\big(H,M_n^{(\boldsymbol{\alpha}^\gsub)}\left(\varepsilon_n\right)\big)$.

To compute $\expec\big[N^\gsub\big(H,M_n^{(\boldsymbol{\alpha}^\gsub)}\left(\varepsilon_n\right)\big)\big],$ we need to count the expected number of copies of $H$ on vertices $v_1, \ldots, v_k\in[n]$ for which $D_{v_i}\approx n^{\alpha_i^\gsub}$. This means that we sum the probabilities that $D_{v_i}=k_i$ for all $i\in [k]$ over all $k_i$ that are of the order $n^{\alpha_i^\gsub}$, multiplied by the probabilities that $\{v_i,v_j\}$ is an edge in $\CMnD$ for all $\{i,j\}\in \Ecal_H$, conditionally on the degrees. We rescale the arising sum over $k_i$, and instead integrate over $x_{i}=k_i n^{-\alpha_i^\gsub}$.
\smallskip

After rescaling, we are left with a $k$-fold integral over the variables $x_i$ for all $i\in [k]$ of which we aim to show that it satisfies the appropriate bounds. Here, we will crucially rely on the uniqueness of the optimization problem in \eqref{eq:maxeqsub}. While we are `merely' left with a $k$-fold integral over relatively simple functions, due to the somewhat implicit information that \eqref{eq:maxeqsub} has a unique solution, proving the finiteness of the integral is quite challenging. Indeed, we will need to resort to comparisons over different partitions, and \cj{use} that they provide a smaller value of the functional in \eqref{eq:maxeqsub} to establish the finiteness of the integral.

\smallskip
Let us give some more details about the nature of the integral, and on the organisation of the proof. To simplify notation, we write $S_i^*=S^\gsub_i$ for $i=1,2,3$ for the optimal partition $\cal{P}$ in \eqref{eq:maxeqsub}. Recall that, by Lemma \ref{lem:condexsub} and on the event $J_n$, 
	\eqn{
	\label{edge-prob-ECM}
	\prob_n(\text{$\{v_i,v_j\}$ is an edge in $\CMnD$})\leq \min(D_{v_i}D_{v_j}/(\mu n), 1),
	}
and these events are close to being independent for different edges. We bound the minimum in~\eqref{edge-prob-ECM} by $D_{v_i}D_{v_j}/(\mu n)$ for $i,j\in S_1^*$, for $i$ or $j$ in $V_1$ and for $i\in S_1^*$, $j\in S_3^*$ or vice versa. We bound the minimum by 1 for $i,j\in S_2^*$ and $i\in S_2^*$, $j\in S_3^*$ or vice versa. This means that in the integral over the rescaled variables $x_{i}$, a factor $x_i^{\zeta_i}$ appears for a certain $\zeta_i\geq 0$ (for a precise definition of $\zeta_i$, see \eqref{eq:zeta}). Further, the nice aspect of this bound is that the integrals over $x_i$ for $i\in [k]$ {\em factorize} into integrals over $x_i$ for $i\in S_3^*$ and $i\in S_1^*\cup S_2^*$. This allows us to study these integrals {\em separately}. It turns out that our proof of the finiteness of these integrals depends sensitively on the optimization problem in \eqref{eq:maxeqsub} \cj{having a} {\em unique} \cj{solution}. In turn, this explains why some of these integrals are quite hard to bound, as the only ingredient we have is that the optimization problem in \eqref{eq:maxeqsub} \cj{has a} unique \cj{solution}.
\smallskip

The remainder of the proof is now organised as follows. In Lemma \ref{lem:dmotif}, we derive bounds on the additional powers $\zeta_i$ of $x_i$ in the rescaled integral, which will prove crucial in bounding the arising integrals. In Lemma \ref{lem:S3int}, we derive a bound on the integrals over $x_i$ for $i\in S_3^*$, and in Lemma \ref{lem:S1S2int}, we bound the integrals over $x_i$ for $i\in S_1^*\cup S_2^*$. After stating Lemmas \ref{lem:S3int} and \ref{lem:S1S2int}, we complete the proof of Theorem~\ref{thm:motifs}(i). Then we give the (rather involved) proofs of Lemmas \ref{lem:S3int} and \ref{lem:S1S2int}.
\smallskip

Before giving the details of the argument, we introduce some further notation.
For any $W\subseteq V_H$, we denote by $d_{i,W}$ the number of edges from vertex $i$ to vertices in $W$.
Let $H$ be a connected subgraph, such that the optimum of~\eqref{eq:maxeqsub} is unique, and let ${\mathcal{P}}=(S_1^*,S_2^*,S_3^*)$ be the optimal partition. Define
	\begin{equation}\label{eq:zeta}
	\zeta_i=
	\begin{cases}
	1 & \text{if }d_i=1,\\
	d_{i,S_1^*}+d_{i,S_3^*}+d_{i,V_1} & \text{if }i\in S_1^*,\\
	d_{i,V_1} & \text{if }i\in S_2^*,\\
	d_{i,S_1^*}+d_{i,V_1} & \text{if }i\in S_3^*.
	\end{cases}
	\end{equation}
The following lemma states several properties of the number of edges between vertices in the different optimizing sets:
\begin{lemma}[Bounds on the additional powers of rescaled variables]
\label{lem:dmotif}
	Let $H$ be a connected subgraph, such that the optimum of~\eqref{eq:maxeqsub} is unique, and let ${{\mathcal{P}}=(S_1^*,S_2^*,S_3^*)}$ be the optimal partition. Then
	\begin{enumerate}[label={\upshape(\roman*)}]
		\item $\zeta_i\leq 1$ for $i\in S_1^*$;
		\item $d_{i,S_1^*}+\zeta_i\geq 2$ for $i\in S_2^*$;
		\item $\zeta_i\leq 1$ and $d_{i,S_3^*}+\zeta_i\geq 2$ for $i\in S_3^*$.
	\end{enumerate}
\end{lemma}

\begin{proof} Suppose first that $i\in S_1^*$. Now consider the partition $\hat{S}_1=S_1^*\setminus \{i\}$, $\hat{S}_2=S_2^*$, $S_3=S_3^*\cup \{i\}$. Then, $E_{\hat{S}_1}=E_{S_1^*}-d_{i,S_1^*}$  and $E_{\hat{S}_1,\hat{S}_3}=E_{S_1^*,S_3^*}+d_{i,S_1^*}-d_{i,S_3^*}$. Furthermore, $E_{\hat{S}_1,V_1}=E_{S^*_1,V_1}-d_{i,V_1}$ and $E_{\hat{S}_2,V_1}=E_{S_2^*,V_1}$.  Because the partition into $S_1^*,S_2^*$ and $S_3^*$ achieves the unique optimum of~\eqref{eq:maxeqsub},
	\begin{equation}
	\begin{aligned}[b]
	&|S_1^*|-|S_2^*|-\frac{2E_{S_1^*}-E_{S_1^*,S_3^*}+E_{S_2^*,V_1}-E_{S_1^*,V_1}}{\tau-1}\\
	&>|S_1^*|-1-|S_2^*|-\frac{2E_{S_1^*}-E_{S_1^*,S_3^*}-d_{i,S_1^*}-d_{i,S_3^*}+E_{S_2^*,V_1}-E_{S_1^*,V_1}+d_{i,V_1}}{\tau-1},
	\end{aligned}
	\end{equation}
which reduces to
	\begin{equation}
	d_{i,S_1^*}+d_{i,S_3^*}+d_{i,V_1}=\zeta_i<\tau-1.
	\end{equation}
Using that $\tau\in(2,3)$ then yields $d_{i,S_1^*}+d_{i,S_3^*}+d_{i,V_1}\leq 1$. 
	
Similar arguments give the other inequalities. For example, for $i\in S_3^*$, considering the partition where $i$ is moved to $S_1^*$ gives the inequality $d_{i,S_3^*}+d_{i,S_1^*}+d_{i,V_1}\geq 2$, and considering the partition where $i$ is moved to $S_2^*$ results in the inequality $d_{i,S_1^*}+d_{i,V_1}\leq 1$, so that $\zeta_i\leq 1$.
\end{proof}

We now show that two integrals related to the solution of the optimization problem~\eqref{eq:maxeqsub} are finite, using Lemma~\ref{lem:dmotif}. These integrals are the key ingredient in proving Theorem~\ref{thm:motifs}(i) for subgraphs.
\begin{lemma}[{Subgraph integrals over $S_3^*$}]
\label{lem:S3int}
	Suppose that the maximum in~\eqref{eq:maxeqsub} is uniquely attained by ${\mathcal{P}}=(S_1^*,S_2^*,S_3^*)$ with $|S_3^*|=s>0$, and say $S_3^*=[s]$. Then
	\begin{equation}\label{eq:S3int}
	\int_{0}^{\infty}\cdots \int_{0}^\infty \prod_{i \in [s]}x_i^{-\tau+\zeta_i}\prod_{\{i,j\}\in \Ecal_{S_3^*}}\min(x_ix_j,1)\dd x_s\cdots\dd x_1<\infty.
	\end{equation} 
\end{lemma}
\smallskip

The proof of Lemma \ref{lem:S3int} is deferred to after the proof of Theorem~\ref{thm:motifs}(i). We continue with the integrals over $S_1^*\cup S_2^*$:

\begin{lemma}[{Subgraph integrals over $S_1^*\cup S_2^*$}]
\label{lem:S1S2int}
	Suppose the optimal solution to~\eqref{eq:maxeqsub} is unique, and attained by ${\mathcal{P}}=(S_1^*,S_2^*,S_3^*)$. Say that $S_2^*=[t_2]$ and $S_1^*=[t_2+t_1]\setminus [t_2]$. Then, 
	for every $a>0$,
	\begin{equation}
	\label{eq:S1S2int}
	\int_{0}^{a}\cdots \int_0^a\int_0^\infty\cdots\int_0^\infty \prod_{j\in[t_1+t_2]}x_j^{-\tau+\zeta_j} \ \prod_{\mathclap{\{i,j\}\in \Ecal_{S_1^*,S_2^*}}}\min(x_ix_j,1)\dd x_{t_1+t_2}\cdots \dd x_1<\infty .
	\end{equation}
\end{lemma}
\smallskip

The proof of Lemma \ref{lem:S1S2int} is deferred to after the proof of Theorem~\ref{thm:motifs}(i). Now we are ready to complete the proof of Theorem~\textup{\ref{thm:motifs}(i)} for subgraphs:

\begin{proof}[Proof of Theorem~\textup{\ref{thm:motifs}(i)}]
Because $D_{\max}=\bigOps(n^{1/(\tau-1)})$, for any $\eta_n\to 0$,  $D_{\max}\leq n^{1/(\tau-1)}/\eta_n$ with high probability. Define 
	\begin{equation}
	\gamma_i^u(n)=\begin{cases}
	n^{1/(\tau-1)}\cs{/\eta_n}& \text{if }i\in S_2^*,\\
	n^{\alpha_i^\gsub}/\varepsilon_n & \text{else,}
	\end{cases}
	\end{equation}
with $\alpha_i^\gsub$ as in~\eqref{eq:alphasub}, and denote
	\begin{equation}
	\gamma_i^l(n)=\begin{cases}
	1& \text{if }i\in V_1,\\
	\varepsilon_n n^{\alpha_i^\gsub}& \text{else.}
	\end{cases}
	\end{equation}
We then show that the expected number of subgraphs where the degree of at least one vertex $i$ satisfies $D_i\notin[\gamma_i^l(n),\gamma_i^u(n)]$ is small, similarly to the proof of Theorem~\ref{thm:sqrtsub} in Section \ref{sec:proof2}.

With loss of generality, we assume that vertex $1\in V_H$ satisfies $1\in V_1$. We count the expected number of $\boldsymbol{v}=(v_1, \ldots, v_k)$ for which the edge $\{v_i,v_j\}$ is present in $\CMnD$ for every $\{i,j\}\in \Ecal_H$. We first study the expected number of copies of $H$ where vertex $V_1$ has degree $D_{v_1}\in[1,\gamma_1^l(n))$ and all other vertices satisfy $D_{v_i}\in[\gamma_i^l(n),\gamma_i^u(n)]$, by integrating the probability that subgraph $H$ is formed over the range where vertex $v_1$ has degree $D_{v_1}\in[1,\gamma_1^l(n))$ and all other vertices satisfy $D_{v_i}\in[\gamma_i^l(n),\gamma_i^u(n)]$. Using that the connection probabilities can be bounded by ${M_1\min(D_{v_i}D_{v_j}/n,1)}$ for some $M_1>0$ (recall Lemma \ref{lem:condexsub}, and in particular \eqref{edge-prob-ECM}), and the degree distribution can be bounded as $\Prob{D=k}\leq M_2k^{-\tau}$ for some $M_2>0$ by \eqref{D-tail}, we bound the expected number of such copies of $H$ by
	\begin{equation}
	\label{eq:Exp1small}
	\begin{aligned}[b]
	&\sum_{\boldsymbol{v}}\Exp{I^\gsub(H, \boldsymbol{v})\ind{D_{v_1}<\gamma^l_1(n),D_{v_i}\in [\gamma_i^l(n),\gamma_i^u(n)] \ \forall i>1}}\\
	& \leq Kn^k\int_{1}^{\gamma_1^l(n)}\int_{\gamma_2^l(n)}^{\gamma_2^u(n)}\cdots \int_{\gamma_k^l(n)}^{\gamma_k^u(n)}(x_1\cdots x_k)^{-\tau}
	\prod_{\mathclap{\{i,j\}\in \Ecal_H}}\min\left(\frac{x_ix_j}{n},1\right)\dd x_k\cdots\dd x_1,
	\end{aligned}
	\end{equation}
for some $K>0$, and where we recall that $I^\gsub(H, \boldsymbol{v})=\ind{\ECMnD|_{\boldsymbol{v}}\supseteq \Ecal_H}$.  This integral equals zero when vertex 1 is in $V_1$, since then $[1,\gamma_1^l(n))=\varnothing$. 
Suppose {that} vertex 1 is in $S_2^*$. W.l.o.g.\ assume that $S_2^*={[t_2]}$, $S_1^*={[t_1+t_2]\setminus [t_2]}$ and $S_3^*={[t_1+t_2+t_3]\setminus [t_1+t_2]}$. 
We bound the minimum in~\eqref{eq:Exp1small} by 
	\begin{itemize}
	\item[(a)] $x_ix_j/n$ for $i,j\in S_1^*$;
	\item[(b)] $x_ix_j/n$ for $i$ or $j$ in $V_1$; 
	\item[(c)] $x_ix_j/n$ for $i\in S_1^*$, $j\in S_3^*$ or vice versa; and
	 \item[(d)] 1 for $i,j\in S_2^*$ and $i\in S_2^*$, $j\in S_3^*$ or vice versa.
	 \end{itemize}

Applying the change of variables $y_i=x_i/n^{\alpha_i^\gsub}$ results, for some $\tilde{K}>0$,
in the bound
	\begin{equation}\label{eq:expnhsmall}
	\begin{aligned}[b]
	&{\sum_{\boldsymbol{v}}\Exp{I^\gsub(H, \boldsymbol{v})\ind{D_{v_1}<\gamma^l_1(n),D_{v_i}\in [\gamma_i^l(n),\gamma_i^u(n)] \ \forall i>1}}}\leq \tilde{K} n^{|S_1^*|(2-\tau)+|S_3^*|(1-\tau)/2-|S_2^*|}\\
	& \quad \times n^kn^{\frac{\tau-3}{\tau-1}E_{S_1^*}+\frac{\tau-3}{2(\tau-1)}E_{S_1^*,S_3^*}-\frac{1}{\tau-1}E_{S_1^*,V_1}-\frac{1}{2}E_{S_3^*,V_1}-\frac{\tau-2}{\tau-1}E_{S_2^*,V_1}}\nonumber\\
	& \quad \times  \int_{0}^{\varepsilon_n}\int_{0}^{\cs{1/\eta_n}}\cdots\int_{0}^{\cs{1/\eta_n}}\int_{0}^{\infty}\cdots \int_{0}^{\infty}\prod_{i\in V_H\setminus V_1}y_i^{-\tau+\zeta_i}\\
	& \quad  \times\prod_{\mathclap{\{i,j\}\in \Ecal_{S_3^*}\cup E_{S_1^*,S_2^*}}}\min(y_iy_j,1)\dd y_{t_1+t_2+t_3}\cdots \dd y_{1} \prod_{j \in V_1}\int_{1}^{\infty} y_j^{1-\tau}\dd y_j,
	\end{aligned}
	\end{equation}
where the integrals from 0 to $1/\eta_n$ correspond to vertices in $S_2^*$ and the integrals from 0 to $\infty$ to vertices in $S_1^*$ and $S_3^*$. Since $\tau\in(2,3)$, the integrals corresponding to vertices in $V_1$ are finite. By the analysis from~\eqref{eq:maxtemp} to~\eqref{eq:maxcontrscaling}, 
	\begin{align}
	&|S_1^*|(2-\tau)+|S_3^*|(1-\tau)/2-|S_2^*|+k+\frac{\tau-3}{\tau-1}E_{S_1^*}+\frac{\tau-3}{2(\tau-1)}E_{S_1^*,S_3^*}\\
	&\qquad\qquad -\frac{1}{\tau-1}E_{S_1^*,V_1}
	-\frac{1}{2}E_{S_3^*,V_1}-\frac{\tau-2}{\tau-1}E_{S_2^*,V_1}\nonumber\\
	&\qquad= \frac{3-\tau}{2}(k_{2+}+B^\gsub(H))+k_1/2.\nonumber
	\end{align}
The integrals over $y_i\in V_H\setminus V_1$ can be split into
	\begin{equation}
	\label{eq:ints2s3}
	\begin{aligned}[b]
	& \int_{0}^{\varepsilon_n}\int_{0}^{\cs{1/\eta_n}}\cdots\int_{0}^{\cs{1/\eta_n}}\int_{0}^{\infty}\cdots \int_{0}^{\infty} \ \prod_{\mathclap{i\in S_1^*\cup S_2^*}} \ y_i^{-\tau+\zeta_i}\prod_{\mathclap{\{i,j\}\in \Ecal_{S_1^*,S_2^*}}}\min(y_iy_j,1)\dd y_{t_1+t_2}\cdots \dd y_{1}\\
	& \quad \times\int_{0}^{\infty}\cdots \int_{0}^{\infty}\prod_{i\in S_3^*}y_i^{-\tau+\zeta_i}\prod_{\mathclap{\{i,j\}\in \Ecal_{S_3^*}}} \ \min(y_iy_j,1)\dd y_{t_1+t_2+t_3}\cdots \dd y_{t_1+t_2+1}.
	\end{aligned}
	\end{equation}
	By Lemma~\ref{lem:S3int} the set of integrals on the second line of~\eqref{eq:ints2s3} is finite. Lemma~\ref{lem:S1S2int} shows that the set of integrals on the first line of~\eqref{eq:ints2s3} tends to zero for $\eta_n$ fixed and $\varepsilon_n\to 0$. Thus, choosing $\eta_n\to 0$ sufficiently slowly compared to $\varepsilon_n$ yields
	\begin{align}
	 \int_{0}^{\varepsilon_n}\!&\int_{0}^{{1/\eta_n}}\!\!\cdots\int_{0}^{{1/\eta_n}}\!\!\int_{0}^{\infty}\!\!\cdots \!\int_{0}^{\infty} \ \prod_{\mathclap{i\in S_1^*\cup S_2^*}} \ y_i^{-\tau+\zeta_i}\prod_{\mathclap{\{i,j\}\in \Ecal_{S_1^*,S_2^*}}}\min(y_iy_j,1)\dd y_{t_1+t_2}\cdots \dd y_{1}\nonumber\\
	& = o(1).
	\end{align}
	Therefore,
	\eqan{
	&{\sum_{\boldsymbol{v}}\Exp{I^\gsub(H, \boldsymbol{v})\ind{D_{v_1}<\gamma^l_1(n),D_{v_i}\in [\gamma_i^l(n),\gamma_i^u(n)] \ \forall i>1}}}\nonumber\\
	&\qquad=o\left(n^{\frac{3-\tau}{2}(k_{2+}+B^\gsub(H))+k_1/2}\right),
	}
when vertex 1 satisfies $1\in S_2^*$. Similarly, we can show that the expected contribution from $D_{v_1}<\gamma_1^l(n)$ satisfies the same bound when vertex 1 is in $S_1^*$ or $S_3^*$. The expected number of subgraphs where $D_{v_1}>\gamma_1^u(n)$ if vertex 1 is in $S_1^*$, $S_3^*$ or $V_1$ can be bounded similarly, as well as the expected contribution where multiple vertices have $D_{v_i}\notin [\gamma_i^l(n),\gamma_i^u(n)]$. 
	
Denote
		\begin{equation}
		\Gamma_n(\varepsilon_n,\eta_n) = \{(v_1,\dots,v_k)\colon D_{v_i}\in[\gamma_{v_i}^l(n),\gamma_{v_i}^u(n)] \},
		\end{equation}	
and define $\bar{\Gamma}_n(\varepsilon_n,\eta_n)$ as its complement. Denote the number of subgraphs with vertices in $\bar{\Gamma}_n(\varepsilon_n,\eta_n)$ by $N^\gsub(H,\bar{\Gamma}_n(\varepsilon_n,\eta_n))$. Since $D_{\max}\leq n^{1/(\tau-1)}/\eta_n$ with high probability, $\Gamma_n(\varepsilon_n,\eta_n)={M}_n^{(\boldsymbol{\alpha}^\gsub)}$ with high probability. Therefore, with high probability,
		\begin{equation}
		N^\gsub\Big(H,\bar{M}_n^{(\boldsymbol{\alpha}^\gsub)}\left(\varepsilon_n\right)\Big) = N^\gsub\Big(H,\bar{\Gamma}_n(\varepsilon_n,\eta_n)\Big),
		\end{equation}
where $N^\gsub\Big(H,\bar{M}_n^{(\boldsymbol{\alpha}^\gsub))}\left(\varepsilon_n\right)\big)$ denotes the number of copies of $H$ on vertices not in $M_n^{(\boldsymbol{\alpha}^\gsub)}\left(\varepsilon_n\right)$. 
By the Markov inequality,
	\begin{equation}
	N^\gsub\Big(H,\bar{\Gamma}_n(\varepsilon_n,\eta_n)\Big)=\op\left(n^{\frac{3-\tau}{2}(k_{2+}+B^\gsub(H))+k_1/2}\right).
	\end{equation}

	Combining this with the fact that by Theorem~\ref{thm:motifs}(ii) as proved in Section \ref{sec:maxcont}, for fixed $\varepsilon>0$, 
	\begin{align} 
	N^\gsub(H)&= N^\gsub(H,M_n^{(\boldsymbol{\alpha}^\gsub)}(\varepsilon))+N^\gsub(H,\bar{M}_n^{(\boldsymbol{\alpha}^\gsub)}(\varepsilon))\nonumber\\
	& =\bigOps(n^{\frac{3-\tau}{2}(k_{2+}+B^\gsub(H))+k_1/2})
	\end{align}
	shows that
	\begin{equation}
	\frac{N^\gsub\Big(H,M_n^{(\boldsymbol{\alpha}^\gsub)}\left(\varepsilon_n\right)\big)}{N^\gsub(H)}\plim 1,
	\end{equation}
as required. This completes the proof of Theorem~\ref{thm:motifs}(i).
\end{proof}
\smallskip

We close this section by proving the integral Lemmas \ref{lem:S3int} and \ref{lem:S1S2int}:

\begin{proof}[Proof of Lemma \ref{lem:S3int}] Recall that $S_3^*=[s]$. W.l.o.g.\ we may assume that $x_1<x_2<\cdots <x_s$. Let $U=[t]$ be such that $x_i<1$ precisely when $i\in[t]$. Here $U=\varnothing$ when $t=0$.

The integral~\eqref{eq:S3int} consists of multiple regions that will be characterized by $U$, and we will deal with all of them in the sequel. The first region is where $U=\varnothing$, so that $x_1, \ldots, x_s\geq 1$. Since $-\tau+\zeta_i<-1$ by Lemma~\ref{lem:dmotif}(iii), this integral can be bounded by the full integral over $[1,\infty)$ for all variables, which is bounded by
	\begin{equation}
	\label{eq:intxlarge}
	\int_{1}^{\infty}\cdots \int_{1}^{\infty}\prod_{j\in[s]}x_j^{-\tau+\zeta_j}\dd x_1\cdots \dd x_s<\infty.
	\end{equation}
The second region is where $U=[s]$, so that $x_1,\ldots,x_s\in [0,1]$. Since by Lemma~\ref{lem:dmotif}, any vertex in $S_3^*$ satisfies $\zeta_i+d_{i,S_3^*}\geq 2$, this integral can be bounded as
	\begin{equation}\label{eq:intxsmall}
	\begin{aligned}[b]
	&\int_{0}^{1}\cdots \int_{0}^{1}\prod_{j\in[s]}x_j^{-\tau+\zeta_j}\prod_{\{i,j\}\in \Ecal_{S_3^*}}x_ix_j\dd x_1\cdots \dd x_s \\
	&{\qquad =\int_{0}^{1}\cdots \int_{0}^{1}\prod_{j\in[s]}x_j^{-\tau+\zeta_j+d_{j,S_3^*}}\dd x_1\cdots \dd x_s}\\
	& \qquad \leq \int_{0}^{1}\cdots \int_{0}^{1}(x_1\cdots x_s)^{2-\tau}\dd x_1\cdots \dd x_s<\infty.
	\end{aligned}
	\end{equation}
	
The other regions arise when $U\neq \varnothing$ and $U\neq [t]$. For these cases, the integral runs from 1 to $\infty$ for $i\in U$, and from 0 to $1$ for $i\in\bar{U}=S_3^*\setminus U$. In such a region, $\min(x_ix_j,1)=x_ix_j$ when $i,j\notin U$, and $\min(x_ix_j,1)=1$ when $i,j\in U$. Then, as we assumed that $x_1<x_2<\cdots <x_s$, the contribution to~\eqref{eq:S3int} from the region described by $U$ can be bounded by
	\begin{equation}
	\label{eq:intS}
	\int_{1}^{\infty}\int_{x_1}^\infty\cdots \int_{x_{t-1}}^{\infty}\prod_{j \in[t]}x_j^{-\tau+\zeta_j}\prod_{i=t+1}^s h(i,\boldsymbol{x})\dd x_t\cdots \dd x_1,
	\end{equation}
where $\boldsymbol{x}=(x_i)_{i\in[t]}$ and
	\begin{equation}\label{eq:h}
	h(i,\boldsymbol{x})= \int_0^1x_{i}^{-\tau+\zeta_i+d_{i,\bar{U}}}\prod_{\mathclap{j\in U\colon \{i,j\}\in \Ecal_{S_3^*}}} \ \min(x_ix_j,1)\dd x_i,
	\end{equation}
for $i\in {U=[s]\setminus [t]}$.

The integral in $h(i,\boldsymbol{x})$ consists of multiple regions, depending on whether $x_ix_j<1$ or not. Suppose vertex $i\in\bar{U}$ is connected in $H$ to vertices $j_1,j_2,\ldots,j_l\in U$, where $j_1<j_2<\cdots<j_l$ so that also $1<x_{j_1}<x_{j_2}<\cdots<x_{j_l}$ and $l+d_{i,\bar{U}}=d_{i,S_3^*}$. Then, by splitting the integral depending on how many $j_l'$s are such that $x_ix_{j_l}\leq 1$, we obtain
	\begin{align}
	\label{eq:intu}
	h(i,\boldsymbol{x}) &= \int_{0}^{1}x_i^{-\tau+\zeta_i+d_{i,\bar{U}}}\min(x_ix_{j_1},1)\min(x_ix_{j_2},1)\cdots\min(x_ix_{j_l},1)\dd x_i\nonumber\\
	&= \int_{1/x_{j_1}}^{1}x_i^{-\tau+\zeta_i+d_{i,\bar{U}}}\dd x_i+\cdots +x_{j_1}\cdots x_{j_{l-1}}\int_{1/x_{j_l}}^{1/x_{j_{l-1}}}\! x_i^{-\tau+\zeta_i+l-1+d_{i,\bar{U}}}\dd x_i
	\nonumber \\
	&\qquad +x_{j_1}\cdots x_{j_l}\int_{0}^{1/x_{j_l}}\! x_i^{-\tau+\zeta_i+l+d_{i,\bar{U}}}\dd x_i.
	\end{align}
Since $\zeta_i+d_{i,\bar{U}}+l-\tau=\zeta_i+d_{i,S_3^*}-\tau>-1$ by Lemma~\ref{lem:dmotif}(iii), the last integral is finite.

Computing these integrals yields
	\begin{align}
	h(i,\boldsymbol{x}) & {= C_{0}+C_1x_{j_1}^{\tau-\zeta_i-d_{i,\bar{U}}-1}+\cdots +C_{l-1}x_{j_1}\cdots x_{j_{l-2}}x_{j_{l-1}}^{\tau-\zeta_i-l-d_{i,\bar{U}}+1}} \nonumber\\
	& \quad {+C_{l} x_{j_1}x_{j_2}\cdots x_{j_{l-1}}x_{j_l}^{\tau-\zeta_i-l-d_{i,\bar{U}}}}\nonumber\\
	& {= :  C_0h_0(i,\boldsymbol{x}) +C_1 h_1(i,\boldsymbol{x}) +\dots+C_{l}h_{l}(i,\boldsymbol{x})},
	\end{align}	
for some constants $C_0,\dots,C_{l}$. These terms (except for the first term) are all products of powers of $x_{j_1},\dots,x_{j_l}$, such that the sum of these powers is $\tau-\zeta_i-d_{i,\bar{U}}-1$. Furthermore, the exponents of $x_{j_1},\dots,x_{j_b}$ equal 1 for some $b\in[l]$, and the exponents of $x_{j_{b+2}},\dots,x_{j_l}$ equal zero, and
	\eqn{
	\frac{h_{p+1}(i,\boldsymbol{x})}{h_{p}(i,\boldsymbol{x})}
	=x_{j_p}\frac{x_{j_{p+1}}^{\tau-\zeta_i-(p+1)}}{x_{j_{p}}^{\tau-\zeta_i-p+1}}
	=\Big(\frac{x_{j_{p+1}}}{x_{j_p}}\Big)^{\tau-\zeta_i-(p+1)},
	}
which is at most 1 for $p\leq \tau-\zeta_i$, and smaller than 1 for $p>\tau-\zeta_i$. Thus, $p^*=p^*_i={\rm argmax}_{p} h_p(i,\boldsymbol{x})=\lfloor \tau-\zeta_i\rfloor$.
Therefore, there exists a $K>0$ such that
	\begin{align}
	\label{eq:hstar}
	h(i,\boldsymbol{x})\leq K h_{p^*_i}(i,\boldsymbol{x}).
	\end{align}
In particular, $p^*_i=\lfloor \tau-\zeta_i\rfloor\geq 1$ by Lemma~\ref{lem:dmotif}(iii).

Then, for some $\tilde{K}>0$,
		\begin{align}
		\label{eq:intSh1}
		&\int_{1}^{\infty}\int_{x_{1}}^{\infty}\cdots \int_{x_{t-1}}^{\infty}\prod_{j \in[t]}x_j^{-\tau+\zeta_j}\prod_{i=t+1}^sh(i,\boldsymbol{x})\dd x_t\cdots \dd x_1\nonumber\\
		& \leq \tilde{K} \int_{1}^{\infty}\int_{x_{1}}^{\infty}\cdots \int_{x_{t-1}}^{\infty}\prod_{j \in[t]}x_j^{-\tau+\zeta_j}\prod_{i=t+1}^sh_{p^*_i}(i,\boldsymbol{x})\dd x_t\cdots \dd x_1.
		\end{align}
The above steps effectively perform the integrals over $x_i$ for $i\in \bar{U}$, and we are left with the integrals over $x_j$ for $j\in U$. It is here that we will rely on the fact that the optimization problem in \eqref{eq:maxeqsub} \cj{has a {\em unique} solution}. We start by rewriting the integral in \eqref{eq:intSh1} so that we can effectively use the uniqueness of \eqref{eq:maxeqsub}, for which we need to make the dependence on the various $x_j$ for $j\in U$ explicit. We start by introducing some notation to simplify this analysis.
\smallskip

Let $T_i{=\{j_q\colon q\in [p^*_i]\}}\subseteq U$ denote the set of neighbors of $i$ that appear in $h_{p^*_i}(i,\boldsymbol{x})$.
For all $j\in U$, let 
	\begin{align}
	\label{eq:Q}
		Q_j = \{i\in \bar{U}\colon \{i,j\}\in \Ecal_{S_3^*},j_{p_i^*}\geq j\}
	\end{align}
denote the set of neighbors $i\in\bar{U}$ of $j\in U$ such that $x_j$ appears in $h_{p^*_i}(i,\boldsymbol{x})$ (note that $i<j$ for all $i\in \bar{U}, j\in U$). Then,
	\begin{align}
	\label{eq:intcontrf}
	&{
	\prod_{j =1}^tx_j^{-\tau+\zeta_j}\prod_{i=t+1}^sh_{p^*_i}(i,\boldsymbol{x})}
	\\
	&= \tilde{K}
	\prod_{j=1}^{t}x_j^{-\tau+\zeta_j+|Q_j|}\prod_{i=t+1}^{s}x_{j_{p^*_i}}^{\tau-1-\zeta_i-d_{i,\bar{U}}-p^*_i,}\nonumber
	\end{align}
for some constant $\tilde{K}>0$. We now simplify the above integral.

Let $W_j=\{i\in \bar{U}\colon {x_{j_{p^*_i}}}=j\}$ for $j\in[t]$, so that $W_j$ denotes the set of neighbors $i$ of $j$ in $\bar{U}$ such that the factor {$x_j^{\tau-\zeta_j-p_i^*-d_{j,\bar{U}}}$} appears in ${h_{j_{p^*_i}}(i,\boldsymbol{x})}.$ Furthermore, let $\hat{W}_j=(V_1\cup S_1^*\cup [j]\cup \bar{U})\setminus W_j$. 
Then, by~\eqref{eq:zeta} {and the fact that $p_i^*=d_{i,[j_{p_i^*}]}$,} 
	\eqn{
	\sum_{i\in W_j}\zeta_i+d_{i,\bar{U}}+p^*_i=\sum_{i\in W_j}d_{i,V_1}+d_{i,S_1^*}+d_{i,\bar{U}}+p^*_i= 2E_{W_j}+E_{W_j,\hat{W}_j},
	}
where 
	\eqn{
	{E_{W_j}=\big|\big\{\{i,j\}\in \Ecal_{S_3^*}\colon i,j\in W_j\big\}}\big|
	}
denotes the number of edges inside $W_j$ and $E_{W_j,\hat{W}_j}$ denotes the number of edges in $S_3^*$ between $W_j$ and $\hat{W}_j$. As a result,~\eqref{eq:intSh1} becomes
	\begin{equation}\label{eq:intW}
	\begin{aligned}[b]
	\tilde{K}\int_{1}^{\infty} \int_{x_1}^{\infty}\cdots \int_{x_{t-1}}^{\infty}\prod_{j=1}^{t}x_j^{-\tau+\zeta_j+{|Q_j|}+(\tau-1)\abs{W_j}-2E_{W_j}-E_{W_j,\hat{W}_j}}\dd x_{t}\cdots \dd x_1.
	\end{aligned}
	\end{equation} 

We aim to perform the integrals one by one, starting with the integral over $x_t$, followed by $x_{t-1}$, etc. For this, we crucially use the uniqueness of \eqref{eq:maxeqsub} to show that 
	\begin{equation}
	\label{exp-smaller-min1}
	-\tau+\zeta_t+{|Q_t|}+(\tau-1)\abs{W_t}-2E_{W_t}-E_{W_t,\hat{W}_t}<-1,
	\end{equation}
so that the integral in~\eqref{eq:intW} over $x_t$ is finite. Indeed, note that 
	\begin{equation}
	Q_t=\{i\in\bar{U}\colon \{i,t\}\in\mathcal{E}_{S_3^*},j_{p^*_i}=t\}
	=\{i\in W_t\colon \{i,t\}\in\mathcal{E}_{S_3^*}\},
	\end{equation}
because $t$ is the maximal index in $U$, so that ${|Q_t|}=d_{t,W_t}$. Also, $\hat{W}_t=(V_1\cup S_1^*\cup S_3^*)\setminus W_t$ because $[t]\cup \bar{U}=S_3^*$.
\smallskip
	
Define $\hat{S}_2=\hat{S}_2^*\cup \{ t \}$, $\hat{S}_1=\hat{S}_1^*\cup {W}_t$ and $\hat{S}_3=S_3^*\setminus(W_t\cup \{t\})$. This gives 
	\begin{align}
	E_{\hat{S}_1}-E_{S_1^*}&=E_{W_t}+E_{W_t,S_1^*} \label{eq:comparesets1},\\
	E_{\hat{S}_1,\hat{S}_3}-E_{S_1^*,S_3^*}
	& =E_{W_t,S_3^*}-E_{W_t}-E_{W_t,S_1^*}-{|Q_t|}-d_{t,S_1^*},\\
	E_{\hat{S}_1,V_1}-E_{S_1^*,V_1}&=E_{W_t,V_1},\\
	E_{\hat{S}_2,V_1}-E_{S_2^*,V_1}&=d_{t,V_1}.\label{eq:comparesets4}		
	\end{align}
Because~\eqref{eq:maxeqsub} is uniquely optimized by $S_1^*$, $S_2^*$ and $S_3^*$, 
	\begin{equation}\label{eq:SShat}
	\begin{aligned}[b]
	&|\hat{S}_1|-|\hat{S}_2|-\frac{2E_{\hat{S}_1}+E_{\hat{S}_1,\hat{S}_3}+E_{\hat{S}_1,V_1}-E_{\hat{S}_2,V_1}}{\tau-1}\\
	&\quad  <|S_1^*|-|S_2^*|-\frac{2E_{S_1^*}+E_{S_1^*,S_3^*}+E_{S_1^*,V_1}-E_{S_2^*,V_1}}{\tau-1}.
	\end{aligned}
	\end{equation}
Using~\eqref{eq:comparesets1}-\eqref{eq:comparesets4}, this reduces to
	\begin{equation}\label{eq:maxcor}
	\abs{W_t}-1-\frac{2E_{W_t}+E_{W_t,\hat{W}_t}-|Q_t|-d_{t,S_1^*}-d_{t,V_1}}{\tau-1}<0,
	\end{equation}
which is equivalent to
	\begin{equation}
	-\tau+(\tau-1)\abs{W_t}+|Q_t|+d_{t,S_1^*}+d_{t,V_1}-2E_{W_t}-E_{W_t,\hat{W}_t}<-1,
	\end{equation}
which is \eqref{exp-smaller-min1}. Since $\zeta_t=d_{t,V_1}+d_{t,S_1^*}$ by \eqref{eq:zeta}, the inner integral of~\eqref{eq:intW} is finite. As $W_t$ and $W_{t-1}$ are disjoint, we obtain 
that the integral over $x_t$ can be evaluated as
	\begin{align}\label{eq:xtm1}
	& \int_{x_{t-1}}^{\infty}\prod_{j=t-1}^{t}x_j^{-\tau+\zeta_j+|Q_j|+(\tau-1)\abs{W_j}-2E_{W_j}-E_{W_j,\hat{W}_j}}\dd x_{t}\\
	& = 
	K x_{t-1}^{1-2\tau+\zeta_{t-1}+\zeta_t+{|Q_{t-1}|}+{|Q_{t}|}+(\tau-1)\abs{W_t\cup W_{t-1}}-2E_{W_t\cup W_{t-1}}-E_{W_t\cup W_{t-1},\widehat{W_t\cup W}_{t-1}}},\nonumber
	\end{align}
for some $K>0$, where $\widehat{W_t\cup W}_{t-1}=V_1\cup S^*_1\cup S_3^*\setminus(W_t\cup W_{t-1})$.
\medskip

We next repeat the above procedure to evaluate the integral over $x_{t-1}$. Choosing $\hat{S}_2=S_2^*\cup \{  t,t-1\}$, $\hat{S}_1=S_1^*\cup W_t\cup W_{t-1}$ and $\hat{S}_3={{S}_3^*}\setminus(W_t\cup W_{t-1}\cup \{t,t-1\})$, we can again use~\eqref{eq:SShat} to prove that the power of $x_{t-1}$ in~\eqref{eq:xtm1} is smaller than -1, so that integral~\eqref{eq:xtm1} over $x_{t-1}$ from $x_{t-2}$ to $\infty$ as in~\eqref{eq:intW} results in a power of $x_{t-2}$. We continue this process until we arrive at the integral over $x_1$ and show that this final integral is finite. 
%
\medskip

In general, fix $b\in [t-1]$. We let $Z_b=W_t\cup W_{t-1}\cup\cdots \cup W_{t-b}$ and $\hat{Z}=V_1\cup S_1^*\cup S_3^*\setminus Z_b$. Choosing $\hat{S}_2^{\sss(b)}=S_2^*\cup ([t]\setminus [t-b-1])$, $\hat{S}_1^{\sss(b)}=S_1^*\cup Z_b$ and $\hat{S}_3^{\sss(b)}={{S}_3^*}\setminus(Z_a\cup ([t]\setminus [t-b-1]))$ gives
	\begin{align}
	E_{\hat{S}_1^{\sss(b)}}-E_{S_1^*}&=E_{Z_a}+E_{Z_b,S_1^*} ,\\
	E_{\hat{S}_1^{\sss(b)},\hat{S}_3}-E_{S_1^*,S_3^*}
	& =E_{Z_b,S_3^*}-E_{Z_b}-E_{Z_b,S_1^*}-{|Q_t|}-\cdots -|Q_{t-b}|,\nonumber\\
	& \quad -d_{t,S_1^*}-\cdots-d_{t-b,S_1^*}\\
	E_{\hat{S}_1^{\sss(b)},V_1}-E_{S_1^*,V_1}&=E_{Z_b,V_1},\\
	E_{\hat{S}_2^{\sss(b)},V_1}-E_{S_2^*,V_1}&=d_{t,V_1}+d_{t-1,V_1}+\cdots +d_{t-b,V_1}.	
	\end{align}
Then,~\eqref{eq:SShat} reduces to
	\begin{align}
	\abs{Z_b}-b-1-\frac{2E_{Z_b}+E_{Z_b,\hat{Z}_b}-{|Q_t|}-\cdots -|Q_{t-b}|-\zeta_{t}-\cdots -\zeta_{t-b}}{\tau-1}<0,
	\end{align}
which is equivalent to
	\eqan{
	&-(b+1)\tau+b+(\tau-1)\abs{Z_b}+|Q_t|+\cdots +|Q_{t-b}|+\zeta_{t}\\
	&\qquad +\cdots +\zeta_{t-b}-2E_{Z_a}-E_{Z_b,\hat{Z}_b}<-1.\nonumber
}
This is precisely the exponent that appears in the variable $x_{t-b}$ when integrating for $x_{t-b}$ from $x_{t-b-1}$ to $\infty$ (as in~\eqref{eq:xtm1} for $b=1$). Thus indeed, evaluating the integrals in~\eqref{eq:intW} one by one does not result in diverging integrals at $\infty$ as their exponents are smaller than -1. 
Therefore~\eqref{eq:intW} is also finite, so that the claim in \eqref{eq:S3int} follows.
\end{proof}

\begin{proof}[Proof of Lemma \ref{lem:S1S2int}] We first argue that we may assume that $a=1$. Indeed, the integral with $a<1$ is upper bounded by that with $a=1$, while for $a>1$, we can do a change of variables and use that $\min(bx_ix_j,1)\leq b \min(x_ix_j,1)$ for all $b>1$. Thus, from now on, we will assume that $a=1$.
\smallskip

The proof that this integral is finite has a similar structure as the proof of Lemma~\ref{lem:S3int}, and we will be more concise here to avoid repetitions. Recall that $S_2^*=[t_2]$ and $S_1^*=[t_2+t_1]\setminus [t_2]$. Now, it will be convenient to order the $x_i$ for $i\in [t_2]$ such that $x_1>x_2>\cdots>x_{t_2}$, which we can do w.l.o.g. We first rewrite the integral as
	\begin{equation}
	\begin{aligned}[b]
	\int_0^{1}\cdots \int_{0}^{1}\prod_{j\in[t_2]}x_j^{-\tau+\zeta_j}\prod_{i=t_2+1}^{t_1+t_2}\tilde{h}(i,\boldsymbol{x})\dd x_{t_2}\cdots\dd x_1,
	\end{aligned}
	\end{equation}
where $\boldsymbol{x}=(x_j)_{j\in[t_2]}$ and, for ${i\in [t_1+t_2]\setminus [t_2]}$,
	\begin{equation}
	\tilde{h}(i,\boldsymbol{x})=\int_{0}^{\infty}x_i^{-\tau+\zeta_i}\prod_{j\in [t_2]\colon \{i,j\}\in \Ecal_{S_1^*,S_2^*}}\min(x_ix_j,1)\dd x_i.
	\end{equation}
Similarly to~\eqref{eq:intu}, suppose that vertex ${i\in S_1^*=[t_2+t_1]\setminus [t_2]}$ has vertices $j_1,j_2,\ldots,j_l$ as neighbors in ${S_2^*=[t_2]}$, where {$j_1<j_2<\cdots<j_l$, so that now $1>x_{j_1}>x_{j_2}>\cdots>x_{j_l}$. Note that $l=d_{i,S_2^*}$.}{} Then,
	\begin{align}\label{eq:hbound2}
	\tilde{h}(i,\boldsymbol{x})& = \int_{0}^{\infty}x_i^{-\tau+\zeta_i}\min(x_ix_{j_1},1)\min(x_ix_{j_2},1)\cdots\min(x_ix_{j_l},1)\dd x_i\nonumber\\
	& = \int_{1/x_{j_l}}^{\infty}x_i^{-\tau+\zeta_i}\dd x_i+\cdots+\int_{1/x_{j_2}}^{1/x_{j_3}}x_i^{-\tau+\zeta_i+l-1}x_{j_2}\cdots x_{j_l}\dd x_i\nonumber\\
	& \quad +\int_{0}^{1/x_{j_1}}x_i^{-\tau+\zeta_i+l}x_{j_1}\cdots x_{j_l}\dd x_i.
	\end{align}
Because $\zeta_i+l=\zeta_i+d_{i,S_2^*}=d_i\geq 2$ by \eqref{eq:zeta}, and $\zeta_i\leq 1$ by Lemma~\ref{lem:dmotif}(i), the first and the last integrals are finite. 
Computing the integrals yields that for some $C_1,\dots, C_l$,
		\begin{align}
		\tilde{h}(i,\boldsymbol{x})& =  C_lx_{j_l}^{\tau-\zeta_i-1}+\dots +C_2x_{j_2}^{\tau-\zeta_i-l+1}x_{j_3}\cdots x_{j_l}+C_1x_{j_1}^{\tau-\zeta_i-l}x_{j_2}\cdots x_{j_l}\nonumber\\
		& =: C_l\tilde{h}_l(i,\boldsymbol{x})+ C_{l-1}\tilde{h}_{l-1}(i,\boldsymbol{x})+\dots +  C_1\tilde{h}_1(i,\boldsymbol{x}).
		\end{align}
Similarly to the argument leading to~\eqref{eq:hstar}, for all $i\in[t_1+t_2]\setminus [t_2],$ there exists $p^*_i$ such that, for all $1>x_{v_1}>x_{v_2}>\cdots>x_{v_l}$,
	\begin{equation}
	\tilde{h}(i,\boldsymbol{x}) \leq K \tilde{h}_{p^*_i}(i,\boldsymbol{x}),
	\end{equation}
for some $K>0$. 
	Thus,
	\begin{align}
	&\int_0^{1}\int_{0}^{x_1}\cdots \int_{0}^{x_{t_2-1}}\prod_{j\in[t_2]}x_j^{-\tau+\zeta_j}\prod_{i=t_2+1}^{t_1+t_2}\tilde{h}(i,\boldsymbol{x})\dd x_{t_2}\cdots \dd x_1\nonumber\\
	& \leq K \int_0^{1}\int_{0}^{x_1}\!\!\cdots\!\! \int_{0}^{x_{t_2-1}}\!\!\prod_{j\in[t_2]}x_j^{-\tau+\zeta_j}\prod_{i=t_2+1}^{t_1+t_2}\tilde{h}_{p^*_i}(i,\boldsymbol{x})\dd x_{t_2}\cdots \dd x_1.
	\end{align}
Let $T_i{=\{j_q\colon q\geq p^*_i\}}$, so that $|T_i| \leq l$, denote the set of neighbors of $i$ whose terms appear in $h_{p^*_i}(i,\boldsymbol{x})$. Since $\zeta_i+l=\zeta_i+d_{i,S_2^*}=d_i\geq 2$ by \eqref{eq:zeta}, and $\zeta_i\leq 1$ by Lemma~\ref{lem:dmotif}(i), we have that $l\geq 1$, and therefore $|T_i|\geq 1$ for all $i{\in [t_2]}$. 
For $j\in S_2^*$, let
	\begin{equation}
	\label{eq:Ql2}
	Q_j=\{i\in S_1^*\colon \{i,j\}\in \Ecal_H, j_{p_i^*}\leq j\}
	\end{equation}
be the set of indices $i$ such that such that $x_j$ appears in $\tilde{h}_{p^*_i}(i,\boldsymbol{x})$.
	Then,
	\begin{equation}\label{eq:intbound2}
	\begin{aligned}[b]
	&
	\prod_{j\in[t_2]}x_j^{-\tau+\zeta_j}\prod_{i=t_2+1}^{t_1+t_2}\tilde{h}_{p^*_i}(i,\boldsymbol{x})\\
	& \leq \tilde{K} 
	\prod_{j\in[t_2]}x_j^{-\tau+\zeta_j+|Q_j|}\prod_{i=t_2+1}^{t_1+t_2}(1/x_{{j_{p^*_i}}})^{\tau-1-\zeta_i-{(l-p^*_i+1)}},
	\end{aligned}
	\end{equation}
for some $\tilde{K}>0$. 

Define $W_j=\{i\in S_1^*\colon {j_{p^*_i}}=j\}$ for $j\in S_2^*$ and let $\bar{W}_j=V_H \setminus (W_j\cup [j-1])$. 
Using that $\zeta_i=d_{i,V_1}+d_{i,S_1^*}+d_{i,S_3^*}$ for $i\in S_1^*$ by~\eqref{eq:zeta}, and that $l-p_i^*+1=d_{i,S_2^*\setminus [j_{p_i^*}-1]}$, ~\eqref{eq:intbound2} reduces to
	\begin{equation}\label{eq:intK}
	\tilde{K} 
	\prod_{j\in[t_2]}x_j^{-\tau+\zeta_j+|Q_j|+(\tau-1)|W_j|-2E_{W_j}-E_{W_j,\bar{W}_j}}.
	\end{equation} 
We set $\hat{S}_1=S_1^*\setminus W_{t_2}$, $\hat{S}_2=S_2^*\setminus \{t_2\}$ and $\hat{S}_3=S_3^*\cup W_{t_2}\cup \{t_2\}$. Notice that 
	\begin{align}
	E_{S_1^*}-E_{\hat{S}_1}&=E_{W_{t_2}}+E_{W_{t_2},S_1^*\setminus W_{t_2}},\\ E_{S_1^*,S_3^*}-E_{\hat{S}_1,\hat{S}_3}
	&=E_{W_{t_2},S_3^*}-d_{{t_2},S_1^*\setminus W_{t_2}}-E_{W_{t_2},S_1^*\setminus W_{t_2}},\\ 
	E_{S_1^*,V_1}-E_{\hat{S}_1,V_1}&=E_{W_{t_2},V_1}\\
	E_{S_2^*,V_1}-E_{\hat{S}_2,V_1}&={d_{t_2,V_1}}.
	\end{align} 
Because the optimal solution to~\eqref{eq:maxeqsub} is unique, we obtain using~\eqref{eq:SShat} that
	\begin{align}\label{eq:taubound1}
	&-\tau+(\tau-1)|W_{t_2}| -2E_{W_{t_2}}-E_{W_{t_2},S_1^*\setminus W_{t_2}}\nonumber\\
	& -E_{W_{t_2},S_3^*}-E_{W_{t_2},V_1}+d_{t_2,S_1^*\setminus W_{t_2}}+d_{t_2,V_1}>-1.
	\end{align}
Note that $\bar{W}_{t_2}=V_1\cup S_1^*\cup S_3^*\cup\{t_2\}\setminus W_{t_2}$ since $S_2^*=[t_2]$. Therefore, 
		\eqn{
		E_{W_{t_2},\bar{W}_{t_2}}=E_{W_{t_2},S_1^*\setminus W_{t_2}}+E_{W_{t_2},S_3^*}+d_{t_2,W_{t_2}}+E_{W_{t_2},V_1}.
		}
Using~\eqref{eq:taubound1} and that by~\eqref{eq:zeta} $\zeta_{t_2}=d_{t_2,V_1}$ then shows that
	\begin{equation}
	-\tau+(\tau-1)|W_{t_2}| -2E_{W_{t_2}}-E_{W_{t_2},\bar{W}_{t_2}}+E_{W_{t_2},S_2^*}+d_{t_2,S_1^*\setminus W_{t_2}}+\zeta_{t_2}>-1.
	\end{equation}
We then use that $d_{t_2,S_1^*\setminus W_{t_2}}+d_{t_2,W_{t_2}}=d_{t_2,S_1^*}$ to obtain
	\begin{equation}
	\label{exponent-xt2}
	-\tau+(\tau-1)|W_{t_2}| -2E_{W_{t_2}}-E_{W_{t_2},\bar{W}_{t_2}}+d_{t_2,S_1^*}+\zeta_{t_2}>-1.
	\end{equation}
 Finally, by~\eqref{eq:Ql2}, $|Q_{t_2}|=d_{t_2,S_1^*}$ as $t_2$ is the largest index in $S_1^*$.

This shows that the integral of~\eqref{eq:intK} over $x_{t_2}\in [0,x_{t_2-1})$ equals a power of $x_{t_2-1}$. A similar argument, setting $\hat{S}_1=S_1^*\setminus (W_{t_2}\cup W_{t_2-1})$ and $\hat{S}_2=S_2^*\setminus \{t_2,t_2-1\}$ shows that the integral of ~\eqref{eq:intK} over $x_{t_2-1}\in [0,x_{t_2-2})$ equals a power of $x_{t_2-2}$, and we can proceed to show that the outer integral of~\eqref{eq:intK} is finite. We conclude that~\eqref{eq:S1S2int} is finite.
\end{proof}


\section{Induced subgraphs}\label{sec:graphlets}
We now describe how to adapt the analysis of subgraphs to induced subgraphs.
For induced subgraphs we can define a similar optimization problem as~\eqref{eq:maxeqbeta}. When $\alpha_i+\alpha_j< 1$,~\eqref{eq:pijsmall} results in
	\begin{equation}
	\Probn{X_{v_i,v_j}=0}=\me^{-\Theta(n^{\alpha_i+\alpha_j-1})}(1+o(1))=1+o(1),
	\end{equation}
whereas for $\alpha_i+\alpha_j>1$,~\eqref{eq:pijlarge2} yields
	\begin{equation}
	\Probn{X_{v_i,v_j}=0}=o(1),
	\end{equation}
and for $\alpha_i+\alpha_j=1$~\eqref{eq:pijsmall} yields $\Probn{X_{v_i,v_j}=0}=\Theta(1)$. 
Similar to~\eqref{eq:phsub2}, we can write the probability that $H$ occurs as an induced subgraph on $\boldsymbol{v}=(v_1,\cdots,v_k)$ as
	\begin{equation}\label{eq:Ginduced}
	\begin{aligned}[b]
	&\Probn{\ECMnD|_{\boldsymbol{v}}=\Ecal_H} = \Theta_{\sss{\prob}}\bigg(\prod_{\{i,j\}\in \Ecal_H\colon \alpha_{i}+\alpha_{j}<1}\!\!\!\!\!\! n^{\alpha_{i}+\alpha_{j}-1}\prod_{\mathclap{\{i,j\}\notin \Ecal_H\colon 	\alpha_i+\alpha_j>1}} \me^{-n^{\alpha_i+\alpha_j-1}/2}\bigg).
	\end{aligned}
	\end{equation}
Similarly to~\eqref{eq:phsub2}, edges with $\alpha_i+\alpha_j=1$ do not contribute to the order of magnitude of~\eqref{eq:Ginduced}.
Thus, the probability that $H$ is an induced subgraph on $\boldsymbol{v}$ is stretched exponentially small in $n$ when two vertices $i$ and $j$ with $\alpha_i+\alpha_j>1$ are not connected in $H$. Then the corresponding optimization problem to~\eqref{eq:maxalph} for induced subgraphs becomes
	\begin{equation}\label{eq:maxeqalphaind}
	\begin{aligned}[b]
	& \max (1-\tau)\sum_{i}\alpha_i +\sum_{\{i,j\}\in \Ecal_H\colon \alpha_i+\alpha_j<1}\alpha_i+\alpha_j-1,\\
	&\qquad\text{s.t. }\alpha_i+\alpha_j\leq 1 \quad \forall \{i,j\}\notin \Ecal_H.
	\end{aligned}
	\end{equation} 

The following lemma shows that this optimization problem attains its optimum for very specific values of $\alpha$ (similarly to Lemma~\ref{lem:maxmotif} for subgraphs):
\begin{lemma}[Maximum contribution to induced subgraphs]\label{lem:maxgraphlet}
	Let $H$ be a connected graph on $k$ vertices. If the solution to~\eqref{eq:maxeqbetaind} is unique, then the optimal solution satisfies $\alpha_i\in\{0,\tfrac{\tau-2}{\tau-1},\tfrac{1}{2},\tfrac{1}{\tau-1}\}$ for all $i$. If it is not unique, then there exist at least 2 optimal solutions with $\alpha_i\in\{0,\tfrac{\tau-2}{\tau-1},\tfrac{1}{2},\tfrac{1}{\tau-1}\}$ for all $i$. In any optimal solution, $\alpha_i=0$ if and only if vertex $i$ has degree one in $H$.
\end{lemma}
\begin{proof}
	This proof is similar to the proof of Lemma~\ref{lem:maxmotif}. First, we again define $\beta_i=\alpha_i-\tfrac{1}{2}$, so that~\eqref{eq:maxeqalphaind} becomes
	\begin{equation}\label{eq:maxeqbetaind}
	\begin{aligned}[b]
	& \max \frac{1-\tau}{2}k+(1-\tau)\sum_{i}\beta_i +\sum_{\{i,j\}\in \Ecal_H\colon \beta_i+\beta_j<0}\beta_i+\beta_j,\\
	&\qquad \text{s.t. }\beta_i+\beta_j\leq 0 \quad \forall \{i,j\}\notin \Ecal_H.
	\end{aligned}
	\end{equation} 
	The proof of Step 1 from Lemma~\ref{lem:maxmotif} then also holds for induced subgraphs. Now we prove that if the optimal solution to~\eqref{eq:maxeqbetaind} is unique, it satisfies $\beta_i\in\{-\tfrac 12, \tfrac{\tau-3}{2(\tau-1)},0,\tfrac{3-\tau}{2(\tau-1)}\}$ for all $i$. We take $\tilde{\beta}$ as in~\eqref{eq:tildebeta}, and assume that $\tilde{\beta}<\tfrac{3-\tau}{2(\tau-1)}$. The contribution of the vertices with $\abs{\beta_i}=\tilde{\beta}$ is as in~\eqref{eq:Nbeta}. By increasing $\tilde{\beta}$ or by decreasing it to zero, the constraints on $\beta_i+\beta_j$ are still satisfied for all $\{i,j\}$. Thus, we can use the same argument as in Lemma~\ref{lem:maxmotif} to conclude that $\beta_i\in\{\tfrac{\tau-3}{2(\tau-1)},0,\tfrac{3-\tau}{2(\tau-1)}\}$ for all $i$ with $d_i\geq 2$. A similar argument as in Step 3 of Lemma~\ref{lem:maxmotif} shows that if the solution to~\eqref{eq:maxeqbetaind} is not unique, it can be transformed into two optimal solutions that satisfy $\beta_i\in\{-\tfrac 12, \tfrac{\tau-3}{2(\tau-1)},0,\tfrac{3-\tau}{2(\tau-1)}\}$ for all $i$ with degree at least 2. 
\end{proof}

Following the same lines as the proof of Theorem~\ref{thm:motifs}(ii) for subgraphs, Theorem~\ref{thm:motifs}(ii) for induced subgraphs follows, where we now use Lemma~\ref{lem:maxgraphlet} instead of~\ref{lem:maxmotif}.
We now state an equivalent lemma to Lemma~\ref{lem:convNH} for induced subgraphs:
\begin{lemma}[Convergence of major contribution to induced subgraphs]\label{lem:convNHind}
	Let $H$ be a connected graph on $k>2$ vertices such that~\eqref{eq:maxeqind} is uniquely optimized by $S_3^*=V_{H}$ with $B^\gind(H)=0$. Then,
	\begin{enumerate}[label={\upshape(\roman*)}]
		\item 	
		the number of induced subgraphs with vertices in $W_n^k(\varepsilon)$ satisfies\cs{
		\begin{align}
		\frac{N^\gind(H,W_n^k(\varepsilon))}{n^{\frac{k}{2}(3-\tau)}} = &  (1+\op(1))c^k\mu^{-\frac{k}{2}(\tau-1)} \int_{\varepsilon}^{1/\varepsilon}\!\!\cdots 
		\int_{\varepsilon}^{1/\varepsilon}(x_1\cdots x_k)^{-\tau}\nonumber\\
		&  \times \prod_{\mathclap{\{i,j\}\in \Ecal_H}}(1-\me^{-x_ix_j})\prod_{\mathclap{\{i,j\}\notin \Ecal_H}}\me^{-x_ix_j}\dd x_1\cdots \dd x_k +f_n(\varepsilon),
		\end{align}
		for some function $f_n(\varepsilon)$  such that for any $\delta>0$, 
		\begin{equation}
		\lim_{\varepsilon\searrow 0}\limsup_{n\to\infty}\Prob{f_n(\varepsilon)>\delta\mid J_n}=0.
		\end{equation}}
		\item
		$A^\gind(H)$ defined in~\eqref{eq:Aind} satisfies $A^\gind(H)<\infty$.
	\end{enumerate}
\end{lemma}
The proof of Theorem~\ref{thm:sqrtsub} for induced subgraphs is similar to the proof of Theorem~\ref{thm:sqrtsub} for subgraphs, using Lemma~\ref{lem:convNHind} instead of Lemma~\ref{lem:convNH}. The  proof of Lemma~\ref{lem:convNHind}(i) in turn follows from straightforward extensions of Lemmas~\ref{lem:condexsub}, \ref{lem:convsub} and~\ref{lem:varsub} to induced subgraphs, now also using that the probability that an edge $\{i,j\}\notin \Ecal_H$ is not present in the subgraph can be approximated by $\exp({-D_{v_i}D_{v_j}/L_n})$. Lemma~\ref{lem:convNHind}(ii) is an application of the following equivalent lemma to Lemma~\ref{lem:S3int} for $S_3^*=V_H$:

\begin{lemma}[{Induced subgraph integrals over $S_3^*$}]
\label{lem:S3intind} 
Suppose that the maximum in~\eqref{eq:maxeqind} is uniquely attained by ${\mathcal{P}}=(S_1^*,S_2^*,S_3^*)$ with $|S_3^*|=s>0$, and say that $S_3^*=[s]$.
Then
	\begin{equation}
	\label{eq:S3intind}
	\int_{0}^{\infty}\cdots \int_{0}^\infty \prod_{i \in [s]}x_i^{-\tau+\zeta_i}\prod_{\mathclap{\{i,j\}\in \Ecal_{S_3^*}}}\min(x_ix_j,1)\prod_{\mathclap{\{i,j\}\notin \Ecal_{S_3^*}}}\me^{-x_ix_j}\dd x_s\cdots\dd x_1<\infty.
	\end{equation} 
\end{lemma}

\begin{proof} The proof follows that of Lemma~\ref{lem:S3int}, where now we need to rely on the uniqueness of \eqref{eq:maxeqind} instead of that of \eqref{eq:maxeqsub}, and we obtain extra factors $\me^{-x_ix_j}$ for all $\{i,j\}\notin \Ecal_H$. We will therefore be more brief, and focus on the differences compared to the proof of Lemma~\ref{lem:S3int}.

This integral is finite if 
	\begin{equation}\label{eq:S3intind2}
	\int_{0}^{\infty}\cdots \int_{0}^\infty \prod_{\mathclap{i \in [s]}}x_i^{-\tau+\zeta_i} \ \prod_{\mathclap{\{i,j\}\in \Ecal_{S_3^*}}} \ \min(x_ix_j,1)
	\prod_{\mathclap{\{i,j\}\notin \Ecal_{S_3^*}}} \ \ind{x_ix_j<1}\dd x_s\cdots\dd x_1<\infty,
	\end{equation}
since if 
	\begin{equation}
	\int_{a}^{b}\int_{0}^{1/x_1}x_1^{\gamma_1}x_2^{\gamma_2}\me^{-x_1x_2}\dd x_2\dd x_1<\infty,
	\end{equation} 
then also
	\begin{equation}
	\int_{a}^{b}\int_{1/x_1}^{\infty}x_1^{\gamma_1}x_2^{\gamma_2}\me^{-x_1x_2}\dd x_2\dd x_1<\infty.
	\end{equation}
We can show similarly to~\eqref{eq:intxlarge} and~\eqref{eq:intxsmall} that the integral is finite when all integrands are larger than one, or when all are smaller than one. We compute the contribution to~\eqref{eq:S3intind2} where the integrand runs from 1 to $\infty$ for vertices in some nonempty set $U$, and from 0 to 1 for vertices in $\bar{U}=S_3^*\setminus{U}$. W.l.o.g., assume that $U={[t]}$ for some $1\leq t< s$. Define, for $i\in \bar{U}$,
	\begin{equation}
	\hat{h}(i,\boldsymbol{x})=\int_{0}^{1}x_i^{-\tau+\zeta_i+d_{i,\bar{U}}}\prod_{\mathclap{j\in U\colon \{i,j\}\in \Ecal_{S_3^*}}} \ \min(x_ix_j,1)\prod_{j\in U\colon \{i,j\}\notin \Ecal_{S_3^*}}\ind{x_ix_j<1}\dd x_i.
	\end{equation}
Then~\eqref{eq:S3intind} results in
	\begin{equation}\label{eq:intSind}
	\int_{1}^{\infty}\cdots \int_{1}^{\infty}\prod_{p \in [t]}x_p^{-\tau+\zeta_j}\ \prod_{\mathclap{i,j\in U\colon \{i,j\}\notin \Ecal_{S_3^*}}} \ \ind{x_ix_j<1}
	\prod_{i=t+1}^k\hat{h}(i,\boldsymbol{x})\dd x_t\cdots \dd x_1.
	\end{equation}
When the induced subgraph of $H$ formed by the vertices of $U$ is not a complete graph, this integral equals zero, as $\mathbbm{1}_{\{x_ix_j<1\}}=0$ when $i,j\in U$. Thus, we assume that the induced subgraph of $H$ formed by the vertices of $U$ is a complete graph so that $\{\{i,j\}\in U\colon \{i,j\}\notin \Ecal_{H}\}=\varnothing$.
	
We first bound the region of~\eqref{eq:S3intind} where $1<x_1<\cdots<x_t$. When $i\in\bar{U}$ is connected to all vertices in $U$, $\hat{h}(i,\boldsymbol{x})$ equals $h(i,\boldsymbol{x})$ defined in~\eqref{eq:h}, which can be bounded by~\eqref{eq:hstar}. Otherwise, define
	\begin{equation}\label{eq:ai}
	a_i=\max\{j\in[t]\colon \{i,j\}\notin \Ecal_H\}.
	\end{equation}
Thus, $i$ is connected to vertices ${[t]\setminus [a_i]}$ and we can write $\hat{h}(i,\boldsymbol{x})$ as
	\begin{align}\label{eq:hia}
	\hat{h}(i,\boldsymbol{x})& = \int_{0}^{1/x_t}x_i^{-\tau+\zeta_i+d_{i,S_3^*}}\dd x_i\cdot x_t\cdots x_{a_i+1}\prod_{j \in [a_i]:\{i,j\}\in \Ecal_H}x_j\nonumber\\
	& + \int_{1/x_t}^{1/x_{t-1}}x_i^{-\tau-1+\zeta_i+d_{i,S_3^*}}\dd x_i\cdot x_{t-1}\cdots x_{a_i+1}\prod_{j \in [a_i]:\{i,j\}\in \Ecal_H}x_j	+\cdots  \nonumber\\
	& + \int_{1/x_{a_i+1}}^{1/x_{a_i}}x_i^{-\tau+\zeta_i+d_{i,S_3^*}-t+a_i}\dd x_i\prod_{j \in [a_i]:\{i,j\}\in \Ecal_H}x_j.
	\end{align}
By a similar argument as in Lemma~\ref{lem:dmotif}, $\zeta_i+d_{i,S_3^*}\geq 2$ for $i\in S_3^*$ so that the first integral is finite. Thus, for some constants $C_t,\dots,C_{a_i},$
	\begin{align}
	\hat{h}(i,\boldsymbol{x}) =\!\!\prod_{j \in [a_i]\colon 
	\{i,j\}\in \Ecal_H}\!\!\!\!& x_j \Big( C_tx_t^{\tau-\zeta_i-d_{i,S_3^*}}x_{t-1}\cdots x_{a_i+1}+ C_{t-1}x_{t-1}^{\tau-\zeta_i-d_{i,S_3^*}+1}x_{t-2}\cdots x_{a_i+1}\nonumber\\
	& +\cdots+  C_{a_i}x_{a_i}^{\tau-\zeta_i-d_{i,S_3^*}+t-a_i} \Big)\nonumber\\
	& =: \hat{h}_t(i,\boldsymbol{x})+\cdots+\hat{h}_{a_i}(i,\boldsymbol{x}).
	\end{align}
As in~\eqref{eq:hstar}, for every $i$ we can find a $p^*_i$ such that, for all $1>x_1>\cdots>x_t$,
	\begin{equation}
	\hat{h}(i,\boldsymbol{x})\leq Kh_{p^*_i}(i,\boldsymbol{x})
	\end{equation}
for some $K>0$. Again, let $T_i$ denote the set of neighbors of vertex $i$ appearing in $h_{p^*_i}(i,\boldsymbol{x})$, and $Q_j$ as in~\eqref{eq:Q}, and let ${W_j=\{i\in\bar{U}\colon j_{p^*_i}=j\}}$.
Then,
	\begin{equation}
	\begin{aligned}[b]
	&\int_{1}^{\infty}\cdots \int_{x_{t-1}}^{\infty}\prod_{j \in [t]}x_j^{-\tau+\zeta_j} \ind{x_ix_j<1}\prod_{i=t+1}^k\hat{h}(i,\boldsymbol{x})\dd x_t\cdots \dd x_1\\
	&\leq \tilde{K}\int_{1}^{\infty}\cdots \int_{x_{t-1}}^{\infty}\prod_{j \in [t]}x_j^{-\tau+\zeta_j}\prod_{i=t+1}^k\hat{h}_{p^*_i}(i,\boldsymbol{x})\dd x_t\cdots \dd x_1\\
	& \leq \int_{1}^{\infty} \int_{x_1}^{\infty}\cdots \int_{x_{t-1}}^{\infty}\prod_{j \in [t]}x_j^{-\tau+\zeta_j+|Q_j|+(\tau-1)\abs{W_j}-2E_{W_j}-E_{W_j,\hat{W}_j}}\dd x_{t}\cdots \dd x_1
	\end{aligned}
	\end{equation}
for some $\tilde{K}>0$, where $\hat{W}_j=(V_1\cup S_1^*\cup [j]\cup \bar{U})\setminus W_j$. We can now show that the integral over $x_t$ is finite in a similar manner as in Lemma~\ref{lem:S3int}. 
\medskip
	
Indeed, we will now again use the uniqueness of the solution of the optimization problem in \eqref{eq:maxeqind} to again prove that this integral is finite. For this, we 
define $\hat{S}_1=S_1^*\cup W_t$, $\hat{S}_2=S_2^*\cup\{t\}$ and $\hat{S}_3=S_3^*\setminus (W_t\cup\{t\})$. Because $t\in S_3^*$, by constraint~\eqref{eq:maxeqind}, $t$ is connected to all other vertices in $\hat{S}_2$, so that the vertices of $\hat{S}_2$ still form a complete graph. Furthermore, $t\in U$, so that $t$ is connected to all other vertices in $U$, since the vertices of $U$ formed a complete graph. Also, when $i\in \bar{U}$ is not connected to $t$, then $i\in W_t$ by~\eqref{eq:hia} and the definition of $a_i$ in~\eqref{eq:ai}. Thus, $t$ is connected to all vertices in $U\cup \bar{U}\setminus (W_t\cup \{t\})=\hat{S}_3$.
	
We conclude that $\hat{S}_1$, $\hat{S}_2$ and $\hat{S}_3$ still satisfy the constraint in~\eqref{eq:maxeqind}, and we may proceed as in Lemma~\ref{lem:S3int} using ~\eqref{eq:SShat} to show that the integral over $x_t$ finite. Iterating this proves Lemma~\ref{lem:S3intind}.
\end{proof}  

The following lemma is the counterpart of Lemma~\ref{lem:S1S2int} for induced subgraphs:
\begin{lemma}[{Induced subgraph integrals over $S_1^*\cup S_2^*$}]
\label{lem:S1S2intind}
Suppose the optimal solution to~\eqref{eq:maxeqind} is unique, and attained by ${\mathcal{P}}=(S_1^*,S_2^*,S_3^*)$. Say that $S_2^*=[t_2]$ and $S_1^*=[t_2+t_1]\setminus [t_2]$. Then, for every $a>0$,
	\begin{equation}\label{eq:S1S2intind}
	\begin{aligned}[b]
	\int_{0}^{a}\cdots \int_0^a\int_0^\infty\cdots\int_0^\infty& \prod_{\mathclap{j\in[t_1+t_2]}}x_j^{-\tau+\zeta_j} \ \prod_{\mathclap{\{i,j\}\in \Ecal_{S_1^*,S_2^*}}}\min(x_ix_j,1) \\
	&  \times \prod_{\mathclap{\{i,j\}\notin \Ecal_{S_1^*,S_2^*}}}\me^{-x_ix_j}\dd x_{t_1+t_2}\cdots \dd x_1<\infty .
	\end{aligned}
	\end{equation}
\end{lemma}

\begin{proof} This lemma can be proven along similar lines as Lemma~\ref{lem:S1S2int}, with similar adjustments as the adjustments to prove Lemma~\ref{lem:S3intind} for induced subgraphs from its counterpart for subgraphs, Lemma~\ref{lem:S3int}.
\end{proof}
From these lemmas, the proof of Theorem~\ref{thm:motifs}(i) for induced subgraphs follows along the same lines as the proof of Theorem~\ref{thm:motifs}(i) for subgraphs.


\bibliographystyle{imsart-number}
\DeclareRobustCommand{\VAN}[3]{#3}

\end{document}